\documentclass{article}

\PassOptionsToPackage{numbers, compress}{natbib}
\usepackage[preprint]{neurips}

\usepackage[utf8]{inputenc} 
\usepackage[T1]{fontenc}    
\usepackage{hyperref}       
\usepackage{url}            
\usepackage{booktabs}       
\usepackage{amsfonts}       
\usepackage{nicefrac}       
\usepackage{microtype}      
\usepackage{xcolor}         

\usepackage{amsmath,amsfonts,amssymb,amsthm}
\usepackage{mathtools}
\usepackage{algorithm}
\usepackage{algorithmic}
\usepackage{graphicx}
\usepackage{subcaption}
\usepackage{wrapfig}

\theoremstyle{plain}
\theoremstyle{definition}

\newtheorem{assumption}{Assumption}[section]
\newtheorem{lemma}{Lemma}[section]
\newtheorem{theorem}{Theorem}[section]

\newtheorem{corollary}{Corollary}[section]
\theoremstyle{remark}
\newtheorem{remark}{Remark}[section]

\title{Adam-SHANG: A Convergent Adam-Type Method for Stochastic Smooth Convex Optimization} 

\author{
	Yaxin Yu \\
	School of Mathematics \\
	Sichuan University \\
	Chengdu, China \\
	\texttt{yaxin\_yu@stu.scu.edu.cn}
	\And
	Long Chen\thanks{Corresponding author.} \\
	Department of Mathematics \\
	University of California, Irvine \\
	Irvine, CA, USA \\
	\texttt{lchen7@uci.edu}
	\And
	Minfu Feng \\
	School of Mathematics \\
	Sichuan University \\
	Chengdu, China \\
	\texttt{fmf@scu.edu.cn}
}

\begin{document}
	
\maketitle

\begin{abstract}
We propose \emph{Adam-SHANG}, a Lyapunov-guided Adam-type method that couples momentum, adaptive preconditioning, and a curvature-aware correction through a more stable lagged-preconditioner update. For stochastic smooth convex optimization, we prove convergence in expectation under an admissible stepsize condition that can always be satisfied by a conservative spectral bound, without imposing global monotonicity on the second-moment sequence. To obtain a less conservative practical rule, we introduce a computable trace-ratio stepsize, motivated by a local coordinatewise alignment condition. The same structural update is also tested beyond the convex setting with simplified parameters. Experiments validate the predicted stochastic decay and show competitive training performance against Adam and AdamW on deep learning tasks.
\end{abstract}

\section{Introduction}

We consider the unconstrained optimization problem
\begin{equation}\label{eq:problem}
	\min_{x\in\mathbb{R}^d} f(x),
\end{equation}
where $f:\mathbb{R}^d\to\mathbb{R}$ is continuously differentiable, and let
$x^{\star} \in \arg\min_x f(x)$ denote a global minimizer.
Adam~\citep{kingma2017adammethodstochasticoptimization} is among the most widely used optimizers for solving \eqref{eq:problem} in deep learning. Its success is often attributed to coordinatewise adaptive preconditioning: an exponential moving average (EMA) of the second moment rescales the gradient to account for different coordinate scales. This is especially useful in modern architectures, where gradient scales may vary greatly across parameter groups. Yet the theory faces a basic obstacle: momentum and adaptive preconditioning are coupled in a highly nonlinear way.

A recent work~\citep{yu2026adamhnagconvergentreformulationadam} addresses this obstacle in the \emph{deterministic, full-batch} setting. It applies a variable and operator splitting (VOS) reformulation~\citep{chen2025acceleratedgradientmethodsvariable}, which splits the coupled momentum as $m/\sqrt V=x-y$, and combines it with the Hessian-driven Nesterov accelerated gradient (HNAG) flow~\citep{chen2019orderoptimizationmethodsbased}. This leads to the \emph{Adam-HNAG flow}
\begin{equation}\label{eq:Adam-HNAG flow}
	\left\{
	\begin{aligned}
		x' &= y-x-\beta P^{-1}g(x),\\
		y' &= -P^{-1}g(x),\\
		P' &= -P+\gamma P^{-1}G^2,
	\end{aligned}
	\right.
\end{equation}
where $'$ denotes differentiation with respect to $t$, $g(x)=\nabla f(x)$, $G^2=\mathrm{diag}\big(g(x).^2\big)$, $\beta:[0,\infty)\to(0,\infty)$ is continuously differentiable, and $\gamma:[0,\infty)\to[0,\infty)$ is a time-scaling factor. With initial conditions $x(0)=x_0$, $y(0)=y_0$, and $P(0)=P_0\succ0$, this flow admits an exponentially decaying Lyapunov function. Its discretizations lead to Adam-HNAG methods with provable convergence.

\noindent \textbf{Contributions.}
1. We propose \emph{Adam-SHANG}, a Lyapunov-guided Adam-type method obtained from a stochastic discretization of the flow~\eqref{eq:Adam-HNAG flow}. 
The method couples momentum, adaptive preconditioning, and a curvature-aware correction. 
For stochastic smooth convex optimization, we prove convergence in expectation under an admissible stepsize condition, which admits a conservative spectral realization and does not require global monotonicity of the second-moment sequence.

2. We introduce a computable trace-ratio stepsize based on
$\mathrm{Tr}(P_k^{-1})/\mathrm{Tr}(P_k^{-2})$, which is the main algorithmic ingredient of Adam-SHANG and a new feature beyond deterministic Adam-HNAG. 
The trace-ratio rule provides a less restrictive practical choice, theoretically motivated by a local coordinatewise alignment condition. 
Driven by the evolving preconditioner $P_k$, this rule gives an internal stepsize adaptation mechanism, thereby reducing the need for an externally prescribed learning-rate schedule as in practical Adam.

3. We show that the same {\it Adam-SHANG} algorithm can be used beyond the convex setting with minor parameter changes. Stochastic convex experiments support the predicted decay and validate the trace-ratio rule. Deep learning experiments show that {\it Adam-SHANG} is competitive with Adam and AdamW on practical training tasks.

4. As an ablation study, we include the synchronous variant {\it Adam-SHANG-s}, which uses the updated preconditioner $P_{k+1}$ instead of $P_k$. This variant is a direct modification of Adam and provides a natural comparison for isolating the effect of the trace-ratio stepsize. The counterexample of \citet{reddi2019convergenceadam} further suggests that the nonlinear feedback in {\it Adam-SHANG} provides a stabilization mechanism different from Adam, while {\it Adam-SHANG-s} and Adam exhibit the same instability.

\noindent \textbf{Related work.}
Adam-type methods combine EMA momentum with diagonal adaptive preconditioning, but their convergence theory remains setting-dependent. In online convex optimization, they are usually studied through regret and are known to satisfy sublinear bounds under suitable assumptions~\citep{kingma2017adammethodstochasticoptimization,
	reddi2019convergenceadam,Huang_2019}. Another major line studies nonconvex optimization, where typical results give stationarity guarantees under additional assumptions on the stepsizes and momentum parameters~\citep{de2018convergenceguaranteesrmspropadam,
	zou2019sufficientconditionconvergencesadam,
	chen2019convergenceclassadamtypealgorithms,
	alacaoglu2020convergenceadaptivealgorithmsweakly,
	défossez2022simpleconvergenceproofadam,
	taniguchi2024adoptmodifiedadamconverge,
	Jiang_2024,
	zhou2024convergenceadaptivegradientmethods}.

Results for Adam-type methods on a single convex objective are more limited. The deterministic Adam-HNAG framework~\citep{yu2026adamhnagconvergentreformulationadam} provides the structural starting point for the present work. Other deterministic analyses include IMEX discretizations, quadratic convergence studies, and preconditioned accelerated methods~\citep{Bhattacharjee_2024,dereich2025sharphigherorderconvergence,trifonov2025incorporatingpreconditioningacceleratedapproaches}. In stochastic convex optimization, SHANG~\citep{yu2026shangrobuststochasticacceleration} gives a robust acceleration method without adaptive feedback. The present method can be viewed as SHANG equipped with an Adam-style adaptive preconditioner. Closest among existing Adam-type analyses, \citet{TONG2022333} prove an $O(1/\sqrt{T})$ rate, but under a monotonicity condition on the second moment that is not guaranteed by the EMA update itself. Our analysis avoids this global monotonicity condition and uses local alignment only to motivate the trace-ratio rule.

A complementary line studies Adam from a dynamical-systems viewpoint, including ODE limits, IMEX discretizations, asymptotic pseudo-trajectories, stability analysis, SDE approximations, and higher-order dynamics~\citep{dasilva2019generaldifferentialequationsmodel,
	barakat2020convergencedynamicalbehavioradam,
	Bhattacharjee_2024,
	dereich2024convergenceratesadamoptimizer,
	dereich2025odeapproximationadamalgorithm,
	ma2021qualitativestudydynamicbehavior,
	gould2024continuoustimeanalysisadaptiveoptimization,
	li2018stochasticmodifiedequationsdynamics,
	malladi2024sdesscalingrulesadaptive,
	heredia2026adamadamlikelagrangianssecondorder,
	heredia2025modelingadagradrmspropadam}. These works provide useful continuous-time insight for the discrete stochastic analysis developed here.

\noindent \textbf{Notation.}
We write $\langle\cdot,\cdot\rangle$ and $\|\cdot\|$ for the standard Euclidean inner product and norm. For a symmetric matrix $A$, define
$\langle x,y\rangle_A:=\langle Ax,y\rangle$ and $\|x\|_A^2:=\langle x,x\rangle_A$. We use MATLAB-style notation $.{\tt op}$ for componentwise operations, write $A\succ0$ if $A$ is symmetric positive definite, and denote by $\mathrm{Tr}(D)$ the trace of a diagonal matrix $D$. For a differentiable function $f$, the Bregman divergence is
\begin{equation}\label{bregman_definition}
	D_f(y,x):=f(y)-f(x)-\langle \nabla f(x),y-x\rangle.
\end{equation}
Let $\mathcal{S}_L$ denote the class of convex, $L$-smooth differentiable functions. For any $f\in\mathcal{S}_L$,
\begin{equation}\label{bregman_bound}
	\frac{1}{2L}\|\nabla f(x)-\nabla f(y)\|^2
	\le D_f(x,y)
	\le \frac{L}{2}\|x-y\|^2,
	\qquad \forall x,y\in\mathbb{R}^d.
\end{equation}

\section{Adam-SHANG: Discretization and Convergence Analysis}\label{sec:Adam-SHANG}

\paragraph{Scheme and Algorithm.}
Consider the following discretization of the Adam-HNAG flow~\eqref{eq:Adam-HNAG flow}:
\begin{subequations}\label{eq:AdamSHANG}
	\begin{align}
		\frac{x_{k+1}-x_k}{\alpha_k} &= y_k - x_{k+1} - \beta_k P_{k-1}^{-1} g_k, \label{eq:AdamSHANG-a}\\
		\frac{y_{k+1}-y_k}{\alpha_k} &= -P_k^{-1} g_{k+1}, \label{eq:AdamSHANG-b}\\
		\frac{P_{k+1}-P_k}{\alpha_k} &= - P_{k+1} + \gamma_{k} P_k^{-1} G_{k+1}^2, \label{eq:AdamSHANG-c}
	\end{align}
\end{subequations}
with initial values $x_0$, $y_0$, and $P_0\succ 0$. Here $\alpha_k>0$ is the time stepsize, $\beta_k,\gamma_k>0$ are parameters, and $G_{k+1}^2=\mathrm{diag}\big(g_{k+1}.^2\big)$, where $g_k$ is an unbiased estimator of the full gradient, $\mathbb E[g_k]=\nabla f(x_k)$. The scheme~\eqref{eq:AdamSHANG} may be viewed as SHANG~\citep{yu2026shangrobuststochasticacceleration} equipped with an Adam-style adaptive preconditioner, and is referred to as {\it Adam-SHANG}. 

In the convex experiments, we use the scheme~\eqref{eq:AdamSHANG} directly, with the coupled parameter choice in the convergence analysis. 
For deep learning tasks, we use the following simplified parameterization. Here $\lambda\in(0,1]$ controls the overall stepsize scale, $\beta\in(0,1]$ sets the strength of the correction term, typically with $\beta\le\lambda$, and $\gamma\in(0,1]$ controls the incorporation of new gradient information into the preconditioner. A small $\varepsilon>0$ is added for numerical stability.

\begin{algorithm}[htp]
	\caption{Adam-SHANG for training}
	\label{alg:adamshang}
	
	\begin{algorithmic}[1]
		
		\REQUIRE Initial parameters $x_0$, $y_0$, $P_{-1}=P_0\succ0$, positive parameters $\lambda, \beta, \gamma, \varepsilon>0$
		
		\FOR{$k = 0,1,2,\ldots$}
		
		\STATE Compute step size:
		\[
		\alpha_k =
		\lambda
		\sqrt{\frac{\mathrm{Tr}((P_k+\varepsilon I)^{-1})}
		{\mathrm{Tr}((P_k+\varepsilon I)^{-2})}}
		\]
		
		\STATE Update iterate $x_{k+1}$:
		\[
		x_{k+1}
		=
		\frac{1}{1+\alpha_k}x_k
		+
		\frac{\alpha_k}{1+\alpha_k}y_k
		-
		\frac{\alpha_k\beta}{1+\alpha_k}(P_{k-1}+\varepsilon I)^{-1}g_k
		\]
		
		\STATE Compute a stochastic gradient estimator $g_{k+1}$ of $\nabla f(x_{k+1})$
		
		\STATE Update iterate $y_{k+1}$:
		\[
		y_{k+1}
		=
		y_k-\alpha_k(P_k+\varepsilon I)^{-1}g_{k+1}
		\]
		
		\STATE Update preconditioner:
		\[
		P_{k+1}
		=
		\frac{1}{1+\alpha_k}P_k
		+
		\frac{\alpha_k\gamma}{1+\alpha_k}
		(P_k+\varepsilon I)^{-1}\mathrm{diag}(g_{k+1}.^2)
		\]
		
		\ENDFOR
		
	\end{algorithmic}
\end{algorithm}

\paragraph{Comparison with Adam.}
The main structural difference is the splitting of the momentum as $m/\sqrt{V}=x-y$. This splitting is inspired by the VOS (variable and operator splitting) framework~\cite{chen2025acceleratedgradientmethodsvariable} and leads naturally to a Lyapunov analysis.

A second difference is the extra gradient term $-\beta_k P_{k-1}^{-1} g_k$ in \eqref{eq:AdamSHANG-a}. Subtracting \eqref{eq:AdamSHANG-a} from \eqref{eq:AdamSHANG-b} gives an update for the momentum and introduces a curvature-aware correction through
$$
\nabla f(x_{k+1})-\nabla f(x_k)\approx \nabla^2 f(\xi_k)(x_{k+1}-x_k),
$$
which is the mechanism behind Hessian-driven Nesterov acceleration~\cite{chen2019orderoptimizationmethodsbased}.

A third difference is the update of the preconditioner. Adam discretizes $\tau_2 V'=-V+G^2$ and then defines the preconditioner by $P_{k+1}=\sqrt{V_{k+1}}$. In contrast, {\it Adam-SHANG} discretizes the ODE for the preconditioner $P$ directly, which yields a key cancellation in the discrete Lyapunov analysis.

\paragraph{Sufficient Decay of SGD.}

Let $(\Omega,\mathcal{F},\mathbb{P})$ be a probability space, and let $\{\mathcal{F}_k\}_{k\ge0}$ be a filtration such that $x_{k+1}$ and $P_k$ are $\mathcal{F}_k$-measurable for each $k\ge0$. At iteration $k+1$, let $g_{k+1}$ be a stochastic estimator of $\nabla f(x_{k+1})$. Throughout the one-step analysis, expectations are conditional on $\mathcal F_k$ and will be omitted when clear. Thus $x_{k+1}$ and $P_k$ are deterministic, while $y_{k+1}$ and $P_{k+1}$ are random.

\begin{assumption}[Stochastic gradient model]\label{assump:sg-model}
We assume
\begin{equation}\tag{G1}\label{eq:sg-unbiased-direct}
	\mathbb E\!\left[g_{k+1}\mid \mathcal F_k\right]=\nabla f(x_{k+1}),
\end{equation}
and that there exist constants $\sigma_0,\sigma_1\ge0$ such that
\begin{equation}\tag{G2}\label{eq:sg-metric-secondmoment}
	\mathbb E\!\left[\|g_{k+1}\|_{P_k^{-1}}^2\mid \mathcal F_k\right]
	\le
	\sigma_0^2+(1+\sigma_1^2)\|\nabla f(x_{k+1})\|_{P_k^{-1}}^2,
	\qquad \forall k\ge0.
\end{equation}
\end{assumption}

Assumption~\eqref{eq:sg-metric-secondmoment} is a preconditioned analogue of the standard additive--multiplicative noise condition. A sufficient coordinatewise condition is
	\[
	\operatorname{Var}\!\big[ g_{k+1,i}\big]
	\le
	\nu_i^2 p_{k,i}
	+
	\sigma_1^2\big(\partial_i f(x_{k+1})\big)^2,
	\qquad i=1,\dots,d,
	\]
	for some constants $\nu_i\ge0$. Then \eqref{eq:sg-metric-secondmoment} holds with
	$\sigma_0^2=\sum_{i=1}^d\nu_i^2$. This matches the role of $P_k$: larger coordinatewise scales $p_{k,i}$ allow proportionally larger noise in that coordinate.

To simplify the analysis, define the learning rate
\(\eta:=\alpha\beta\), induced by the discretization stepsize \(\alpha\) and the correction parameter \(\beta\), and introduce
\begin{equation}\label{auxi_vari}
	x^+ := x-\eta P^{-1} g(x),
\end{equation}
which is one preconditioned stochastic gradient step from $x$ with learning rate $\eta$. With this notation, {\it Adam-SHANG}~\eqref{eq:AdamSHANG} can be rewritten equivalently as \eqref{eq:AdamSHANG3}. We first prove a sufficient descent estimate for the auxiliary point $x_{k+1}^+$; the proof is deferred to Appendix~\ref{app:proof-stoch-descent}.

\begin{lemma}\label{lem:stoch-descent}
	Assume that $f$ is $L$-smooth and that $\eta_{k+1}$ is $\mathcal F_k$-measurable. Suppose $g_{k+1} \neq 0$ and Assumptions (G1)–(G2) hold. 
	If the learning rate $\eta_{k+1}$ satisfies the admissibility condition:
\begin{equation}\label{equ:Q_kinlemma}
	0<\eta_{k+1}
	\le
	\frac{q_k}{(1+\sigma_1^2)L}, \ \text{ with } q_k:=\mathbb E\!\left[\|g_{k+1}\|_{P_k^{-1}}^2\right]/ \mathbb E\!\left[\|g_{k+1}\|_{P_k^{-2}}^2\right],
\end{equation}
	then
	\begin{equation}\label{eq:stoch_descent}
		\mathbb{E}\!\left[f(x_{k+1}^+)\right]
		\le
		f(x_{k+1})
		-\frac{\eta_{k+1}}{2(1+\sigma_1^2)}
		\mathbb{E}\!\left[\|g_{k+1}\|_{P_k^{-1}}^2\right]
		+\frac{\eta_{k+1}\sigma_0^2}{1+\sigma_1^2}.
	\end{equation}
\end{lemma}

\paragraph{A practical choice of the learning rate.}\label{sec:practical alpha} A key new difficulty in the stochastic setting is that the full-gradient rule 
$\|g_{k+1}\|_{P_k^{-1}}^2/\|g_{k+1}\|_{P_k^{-2}}^2$ used in deterministic Adam-HNAG is no longer computable. 
The admissibility condition can always be satisfied by the conservative lower bound $\lambda_{\min}(P_k)$, since $q_k\ge \lambda_{\min}(P_k)$, but this choice is often too restrictive. 
We therefore seek a sharper computable lower bound for $q_k$. 
The following coordinatewise alignment assumption is used only for this lower-bound motivation, not for the main descent lemma.

\begin{assumption}[Coordinatewise alignment]\label{assump:alignment}
Let $P_k=\mathrm{diag}(p_1,\dots,p_d)$ and $s_i:=\mathbb E[g_{k+1,i}^2]$. Assume that $\{s_i\}_{i=1}^d$ and $\{p_i\}_{i=1}^d$ are similarly ordered:
$$
(s_i-s_j)(p_i-p_j)\ge 0,
\qquad \forall\, i,j\in\{1,\dots,d\}.
$$
\end{assumption}
Assumption~\ref{assump:alignment} is not implied by the update alone, but it is natural in slowly varying regimes: since each $p_i$ is updated from past squared gradients through \eqref{eq:AdamSHANG-c}, it tracks the recent coordinatewise gradient scale, making it reasonable to expect $s_i$ and $p_i$ to be similarly ordered.

\begin{lemma}[Trace lower bound]\label{lem:trace-lower-bound}
Under Assumption~\ref{assump:alignment},
\begin{equation}\label{equ:Qk inequality}
q_k = \frac{\sum_{i=1}^d p_i^{-1} s_i}
	{\sum_{i=1}^d p_i^{-2} s_i}
	\ge
	\frac{\sum_{i=1}^d p_i^{-1}}
	{\sum_{i=1}^d p_i^{-2}}
	=
	\frac{\mathrm{Tr}(P_k^{-1})}{\mathrm{Tr}(P_k^{-2})}.
\end{equation}
\end{lemma}
The proof of Lemma~\ref{lem:trace-lower-bound}, based on a weighted Chebyshev sum inequality, is given in Appendix~\ref{app:practical-alpha}. 
This lower bound motivates replacing the unknown ratio $q_k$ in~\eqref{equ:Q_kinlemma} by the computable trace ratio. 

To account for possible deviations from the coordinatewise alignment condition and make the rule more robust, we introduce a safety factor $\lambda\in(0,1]$ and use the learning-rate choice
\begin{equation}\label{equ:eta choice}
	\eta_{k+1}
	=
	\frac{\lambda^2}{(1+\sigma_1^2)L}
	\frac{\mathrm{Tr}(P_k^{-1})}{\mathrm{Tr}(P_k^{-2})},
	\qquad \lambda\in(0,1].
\end{equation}

This trace-ratio rule is a new ingredient and performs well in practice. It is not needed in deterministic Adam-HNAG, where the full gradient is available, and it differs from standard Adam analyses, where the learning-rate schedule is prescribed externally. 

\paragraph{Convergence analysis}
Let
\(z_k^+=(x_k^+,y_k)^\top\), \(z_k=(x_k,y_k)^\top\), and define
\begin{equation}\label{lyapunov_discrete}
	\mathcal{E}(z_k^+,P_k):= f(x_k^+)-f(x^\star)+\frac{1}{2}\|y_k-x^\star\|_{P_k}^2.
\end{equation}
We now state the Lyapunov decay estimate in expectation.

\begin{theorem}\label{theorem:discrete stochastic}
	Let $f\in\mathcal{S}_{L}$ and assume the stochastic gradient model in Assumption~\ref{assump:sg-model}. Starting from $x_0$, $y_0$, and $P_0\succ0$, generate $(x_k,y_k,P_k)$ by \eqref{eq:AdamSHANG} and define $x_k^+$ by \eqref{auxi_vari}. Assume that there exists $R>0$ such that
	$
	\sup_{k\ge0}\|y_k-x^\star\|_\infty\le R.
	$
	Choose the coupling
	\[
	\gamma_k R^2=\alpha_k,
	\qquad
	\eta_{k+1}:=\alpha_{k+1}\beta_{k+1}=2\alpha_k^2(1+\sigma_1^2),
	\]
	and assume that the stepsize $\alpha_k$ satisfies
	\begin{equation}\label{eq:alpha-admissible}
	\alpha_k
	\le
	\frac{1}{1+\sigma_1^2}
	\left(\frac{q_k}{2L}\right)^{1/2}.
	\end{equation}
	Then for all $k\ge0$,
	$$
	\mathbb{E}\!\left[\mathcal{E}(z_{k+1}^+,P_{k+1})\right]
	\le
	\prod_{\tau=0}^k(1+\alpha_\tau)^{-1}
	\,\mathbb{E}\!\left[\mathcal{E}(z_0^+,P_0)\right]
	+
	2\sigma_0^2
	\sum_{\tau=1}^{k+1}
	\prod_{j=\tau}^{k+1}(1+\alpha_j)^{-1}\alpha_\tau^2 .
	$$
\end{theorem}

The admissibility condition~\eqref{eq:alpha-admissible} for $\alpha_k$ can always be enforced by using the conservative lower bound $\lambda_{\min}(P_k)$ of $q_k$, but this is often too restrictive. Under Assumption~\ref{assump:alignment}, Lemma~\ref{lem:trace-lower-bound} gives a sharper computable lower bound
\(
q_k
\ge
\mathrm{Tr}(P_k^{-1})/\mathrm{Tr}(P_k^{-2}),
\)
which motivates the practical learning-rate rule~\eqref{equ:eta choice}. Combining~\eqref{equ:eta choice} with the coupling
\(\eta_{k+1}=2\alpha_k^2(1+\sigma_1^2)\) yields
\begin{equation}\label{eq:alpha-trace}
	\alpha_k
	=
	\frac{\lambda}{1+\sigma_1^2}
	\sqrt{
		\frac{1}{2L}
		\frac{\mathrm{Tr}(P_k^{-1})}{\mathrm{Tr}(P_k^{-2})}
	},
	\qquad \lambda\in(0,1].
\end{equation}
The trace-ratio formula \eqref{eq:alpha-trace} is the practical stepsize used in the experiments. Ignoring the constant prefactor, it induces stepsize decay through the evolution of $P_k$. Thus the method does not impose a prescribed schedule such as $\alpha_k\asymp 1/k$ or $\alpha_k\asymp 1/\sqrt{k}$; instead, the stepsize adapts to the current preconditioner, remaining larger while the metric is evolving and decreasing as the dynamics stabilize.

\begin{remark}
The boundedness assumption $\sup_{k\ge0}\|y_k-x^\star\|_\infty\le R$ is used only to control the diagonal feedback term. Such boundedness conditions are common in adaptive-method analyses. It can also be enforced by projecting or clipping $y_k$ onto a large box; this projection does not increase the Lyapunov term, so the result remains valid; see Appendix~\ref{app:proof-main-stoch}.
\end{remark}

\begin{remark}
Existing convergence analyses often impose monotonicity on the second-moment sequence, which is not guaranteed by the EMA update. Our analysis avoids this global condition. The coordinatewise alignment in Assumption~\ref{assump:alignment} is used only to motivate a computable trace-ratio surrogate for the unknown ratio $q_k$ in the admissible learning rate. This condition is sufficient but not necessary: even when Assumption~\ref{assump:alignment} fails, the admissibility condition~\eqref{equ:Q_kinlemma} may still hold.
\end{remark}

We outline the proof of Theorem~\ref{theorem:discrete stochastic} here and defer the full details to Appendix~\ref{app:proof-main-stoch}.
Write
\(
\mathbb{E}\!\left[\mathcal{E}(z_{k+1}^+,P_{k+1})-\mathcal{E}(z_k^+,P_k)\right]
= \mathbb{I}_1+\mathbb{I}_2+\mathbb{I}_3,
\)
where the three terms are given below.

\noindent \textbf{Bound on $\mathbb{I}_1 : = \mathbb{E}\!\left[ \mathcal{E}(z_{k+1}^+,P_{k+1}) - \mathcal{E}(z_{k+1},P_{k+1}) \right]$.} 
Recall that $\eta_{k+1} = 2 \alpha_k^2 (1+\sigma_1^2)$ will satisfy the admissible condition \eqref{equ:Q_kinlemma}. 
Applying Lemma~\ref{lem:stoch-descent} at $x_{k+1}$ gives
\begin{equation}\label{eq:I1-bound}
	\mathbb{I}_1 = \mathbb{E}\!\left[ f(x_{k+1}^+) - f(x_{k+1}) \right]
	\le
	-\frac{\eta_{k+1} }{2(1+\sigma_1^2)}\mathbb{E}\!\left[ \|g_{k+1}\|_{P_k^{-1}}^2	\right]
	+\frac{\eta_{k+1} \sigma_0^2}{1+\sigma_1^2}.
\end{equation}

\noindent \textbf{Bound on $\mathbb{I}_2 : = \mathbb{E}\!\left[ \mathcal{E}(z_{k+1},P_{k+1}) - \mathcal{E}(z_{k+1},P_{k}) \right]$.}
Using the update \eqref{eq:AdamSHANG-c},
\begin{equation}\label{eq:I2-bound}
	\begin{aligned}
		\mathbb{I}_2
		=
		\mathbb{E}\!\left[
		\frac{1}{2}\|y_{k+1}-x^\star\|^2_{P_{k+1}-P_k}
		\right]
		=
		\mathbb{E}\!\left[
		-\frac{\alpha_k}{2}\|y_{k+1}-x^\star\|^2_{P_{k+1}}
		+\frac{\alpha_k\gamma_{k}}{2}\|y_{k+1}-x^\star\|^2_{P_k^{-1} G_{k+1}^2}
		\right].
	\end{aligned}
\end{equation}  
The correction term in the $x$-update contributes 
$
- \|g_{k+1}\|_{P_k^{-1}}^2 = - \mathrm{Tr}(P_k^{-1}G_{k+1}^2).
$
Since both $P_k$ and $G_{k+1}^2$ are diagonal, this term has the same coordinatewise structure as the term $\frac{\alpha_k\gamma_{k}}{2}\|y_{k+1}-x^\star\|^2_{P_k^{-1} G_{k+1}^2}$ in \eqref{eq:I2-bound} and thus can be canceled with a suitable choice of $\beta$ and $\gamma$.

\noindent \textbf{Bound on $\mathbb{I}_3 : = \mathbb{E}\!\left[ \mathcal{E}(z_{k+1},P_{k}) - \mathcal{E}(z_{k}^+,P_{k}) \right]$.} Applying Lemma~\ref{lem:I3-bound}, we obtain
\begin{equation}\label{eq:I3-bound1}
	\mathbb{I}_3
	\le	
	-\alpha_k \big( f(x_{k+1}) - f(x^\star)\big)
	+\frac{\alpha_k^2}{2}\mathbb{E}\!\left[\|g_{k+1}\|_{P_k^{-1}}^2
	\right].
\end{equation}
Combining the bounds on $\mathbb{I}_1$, $\mathbb{I}_2$, and $\mathbb{I}_3$ with an appropriate choice of parameters yields Theorem~\ref{theorem:discrete stochastic}.

\paragraph{Adaptive Recursion and Decay Analysis.}
The next corollary summarizes the decay profiles under two representative envelopes for the adaptive stepsizes; the proof is deferred to Appendix~\ref{app:two-phase-stoch}.

\begin{corollary}[Two-phase behavior]
	\label{cor:two-phase-stoch}
	Assume the hypotheses of Theorem~\ref{theorem:discrete stochastic}.
	
	\smallskip\noindent
	(i) If there exist constants $a_+\ge a_->1$ and $b>0$ such that
	\(
	\frac{a_-}{k+b}\le \alpha_k \le \frac{a_+}{k+b}
	\)
	a.s. for all \(k\ge0\), then
	\[
	\mathbb{E}\!\left[\mathcal{E}(z_{k+1}^+,P_{k+1})\right]
	\lesssim
	(k+b)^{-a_-}\,\mathbb{E}\!\left[\mathcal{E}(z_0^+,P_0)\right]
	+\sigma_0^2 (k+b)^{-1}.
	\]
	
	\smallskip\noindent
	(ii) If there exist constants $\overline{\alpha}\ge \underline{\alpha}>0$ such that
	\(
	\underline{\alpha}\le \alpha_k \le \overline{\alpha}
	\)
	a.s. for all \(k\ge0\), then
	\[
	\mathbb{E}\!\left[\mathcal{E}(z_{k+1}^+,P_{k+1})\right]
	\lesssim
	(1+\underline{\alpha})^{-(k+1)}\,\mathbb{E}\!\left[\mathcal{E}(z_0^+,P_0)\right]
	+\sigma_0^2\,\frac{\overline{\alpha}^2}{\underline{\alpha}}.
	\]
\end{corollary}

The contraction in Theorem~\ref{theorem:discrete stochastic} shows that the decay rate is governed by $\alpha_k$. When $\alpha_k\sim 2/k$, the contraction gives an accelerated polynomial rate. This is the SHANG regime, recovered when $\gamma=0$ and the adaptive feedback is removed, yielding the robust stochastic Nesterov-type rate $O(1/k^2)$ up to a noise level.

Adaptive feedback changes this stepsize dynamics. When gradients are large, the forcing term in the preconditioner update can sustain larger admissible $\alpha_k$. If $\alpha_k$ stays uniformly positive, the contraction becomes linear up to a noise neighborhood, as in Corollary~\ref{cor:two-phase-stoch}(ii). Thus the theory predicts two regimes: accelerated polynomial decay when $\alpha_k\sim 1/k$, and stronger transient contraction when feedback supports larger steps.

\paragraph{A synchronous version.}
In \eqref{eq:AdamSHANG}, $P_k$ is used to precondition $g_{k+1}$. One may update $P_{k+1}$ first and use it to precondition $g_{k+1}$: 
\begin{subequations}\label{eq:AdamSHANGs}
	\begin{align}
		\frac{x_{k+1}-x_k}{\alpha_k} &= y_k-x_{k+1} - \beta_k P_k^{-1} g_k, \label{eq:AdamSHANGs-a}\\
		\frac{P_{k+1}-P_k}{\tilde{\alpha}_k} &= -P_k+\gamma_k P_{k+1}^{-1}G_{k+1}^2, \label{eq:AdamSHANGs-c}\\
		\frac{y_{k+1}-y_k}{\tilde{\alpha}_k} &= -P_{k+1}^{-1}g_{k+1}, \label{eq:AdamSHANGs-b}
	\end{align}
\end{subequations}
with initial values $x_0$, $y_0$, and $P_0\succ 0$. Here $\alpha_k>0$ is the step size, $\tilde{\alpha}_k=\frac{\alpha_k}{1+\alpha_k}\le \alpha_k$, and $G_{k+1}^2=\mathrm{diag}\big(g_{k+1}.^2\big)$. 
We call this synchronous variant {\it Adam-SHANG-s}. The practical implementation is summarized in Algorithm~\ref{alg:adamshangs}.

More importantly, {\it Adam-SHANG-s} \eqref{eq:AdamSHANGs} admits an equivalent representation directly comparable to Adam. Let $P_k=\sqrt{V_k}$ and define $m_k$ by
\(
P_k^{-1}m_k=x_k-y_k.
\)
Then \eqref{eq:AdamSHANGs} can be rewritten as
\begin{equation}\label{eq:adamshangs-adam-form}
	\left\{
	\begin{aligned}
		x_{k+1}
		&=
		x_k-\eta_k(\sqrt{V_k})^{-1}\big(m_k+\beta_k g_k\big),\\
		m_{k+1}
		&=
		\theta_k P_{k+1}P_k^{-1}m_k
		+(1-\theta_k)g_{k+1}
		-(1-\theta_k)\beta_k P_{k+1}P_k^{-1}g_k,\\
		V_{k+1}
		&=
		\theta_k P_{k+1}P_k^{-1}V_k
		+(1-\theta_k)\gamma_k G_{k+1}^2,
	\end{aligned}
	\right.
\end{equation}
where $\eta_k=\frac{\alpha_k}{1+\alpha_k}$ and $\theta_k=\frac{1}{1+\alpha_k}$. Representation \eqref{eq:adamshangs-adam-form} makes the relation with Adam transparent. When $P_{k+1}P_k^{-1}\approx I$, setting $\beta_k=0$ and $\gamma_k=1$ yields an Adam-type update. We refer to Appendix~\ref{app:adamshangs-adam-form} for the derivation.

Unlike {\it Adam-SHANG}, the synchronous update depends on the new metric $P_{k+1}$. This makes the admissible stepsize condition more implicit and makes $\mathbb E[P_{k+1}^{-1}g_{k+1}]$ harder to control in a stochastic Lyapunov analysis. We therefore focus the theory on {\it Adam-SHANG} and use {\it Adam-SHANG-s} as an ablation baseline in the experiments. 
The comparison between \emph{Adam-SHANG-s} and Adam helps isolate the effect of the trace-ratio-driven adaptive stepsize within an Adam-style update, while the comparison between \emph{Adam-SHANG-s} and \emph{Adam-SHANG} highlights the effect of using the lagged preconditioner $P_k$ rather than the updated preconditioner $P_{k+1}$.

\section{Numerical Experiments}\label{sec: experiment}
We organize the experiments into three parts. 
The stochastic convex experiments test the theoretically coupled scheme, the predicted decay, and the admissibility of the trace-ratio rule. 
The counterexample is used as a controlled stress test for Adam-type instability. 
The deep learning experiments evaluate the practical nonconvex implementation with simplified parameters.

\paragraph{Convex optimization.}
Following the convex benchmark used in~\citet{gupta2024nesterovaccelerationdespitenoisy}, we consider
\[
f(x)=\sum_{i=1}^d f_d(x_i),\qquad
f_d(t)=
\begin{cases}
	|t|^d, & |t|\le 1,\\
	1+d(|t|-1), & \text{otherwise,}
\end{cases}
\]
with $d=16$, so that $f\in\mathcal{S}_L$ with $L=d(d-1)$. Since the baseline methods do not admit theoretical guarantees under our noise assumption~\eqref{eq:sg-metric-secondmoment}, we adopt a more general additive-multiplicative noise model for a fair comparison:
\(
g_i=(1+\sigma_1 Z_i)\partial_i f(x)+(\sigma_0/\sqrt{d})\,\xi_i,
\)
which satisfies the standard variance bound
\(
\mathbb{E}[\|g\|^2]\le \sigma_0^2+(1+\sigma_1^2)\|\nabla F\|^2.
\)
Here $Z_i,\xi_i\sim\mathcal N(0,1)$ are independent. 

\begin{figure}[!htbp]
	\centering
	\begin{subfigure}{0.32\textwidth}
		\centering
		\includegraphics[width=\linewidth]{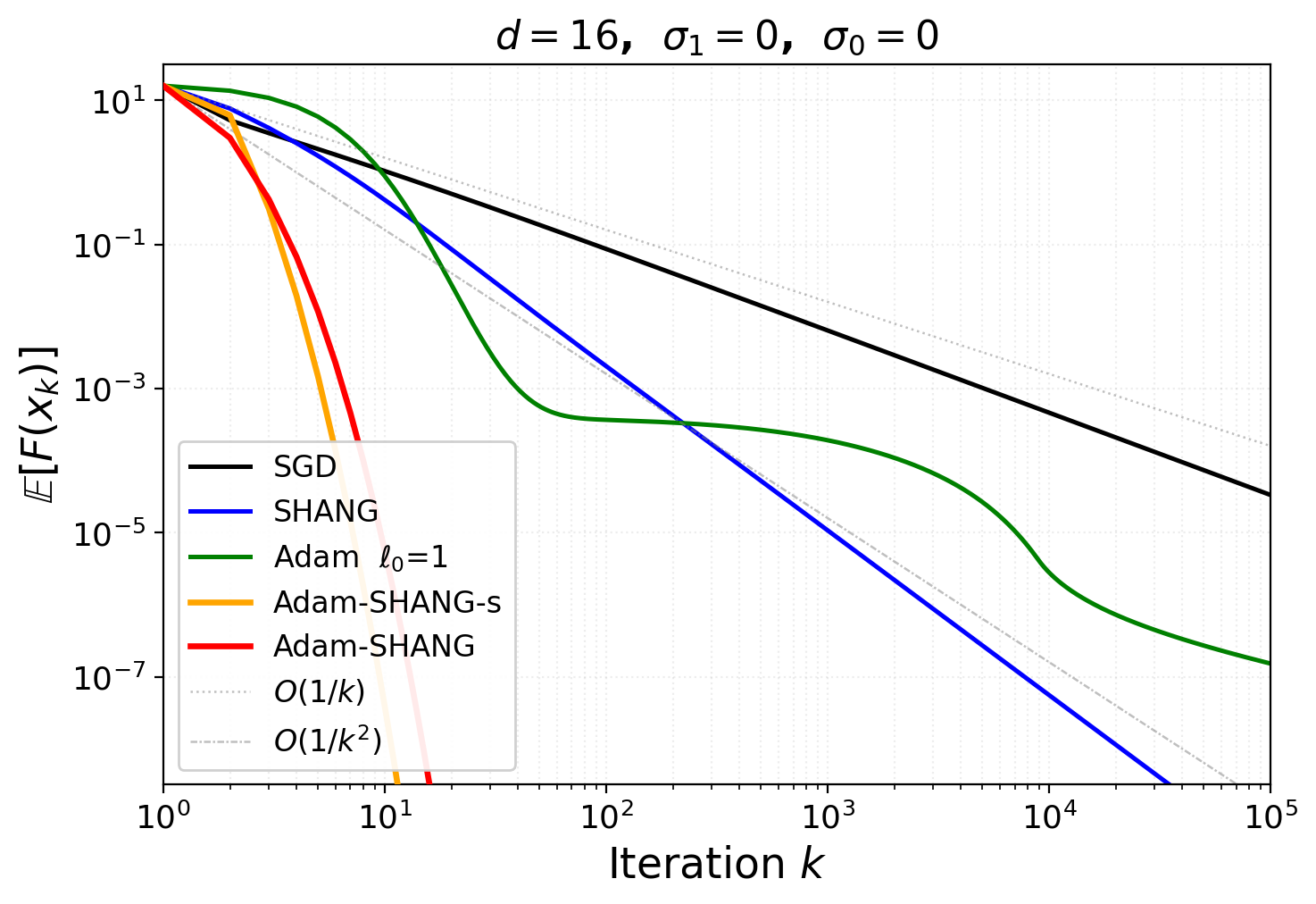}
	\end{subfigure}\hfill
	\begin{subfigure}{0.32\textwidth}
		\centering
		\includegraphics[width=\linewidth]{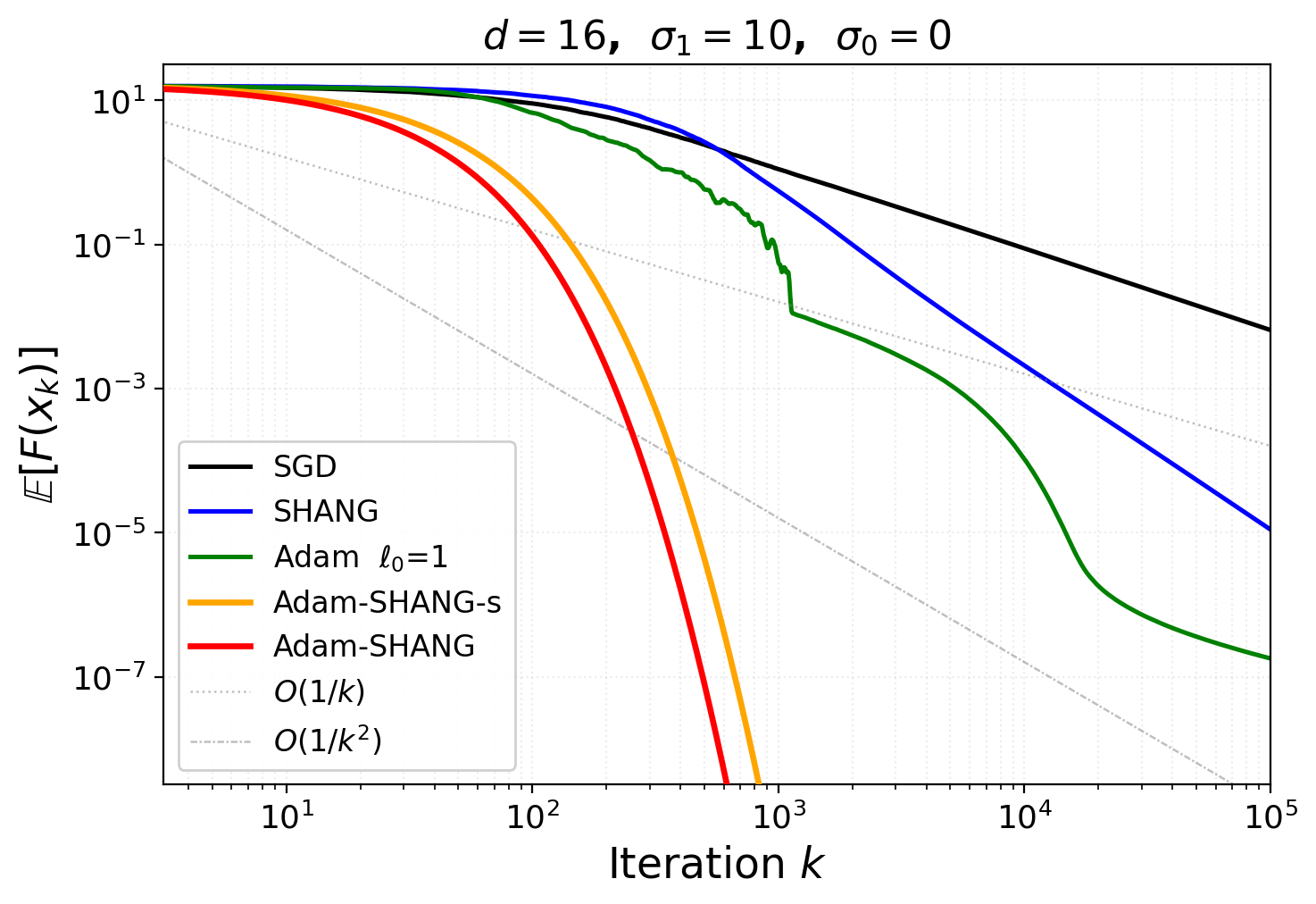}
	\end{subfigure}\hfill
	\begin{subfigure}{0.32\textwidth}
		\centering
		\includegraphics[width=\linewidth]{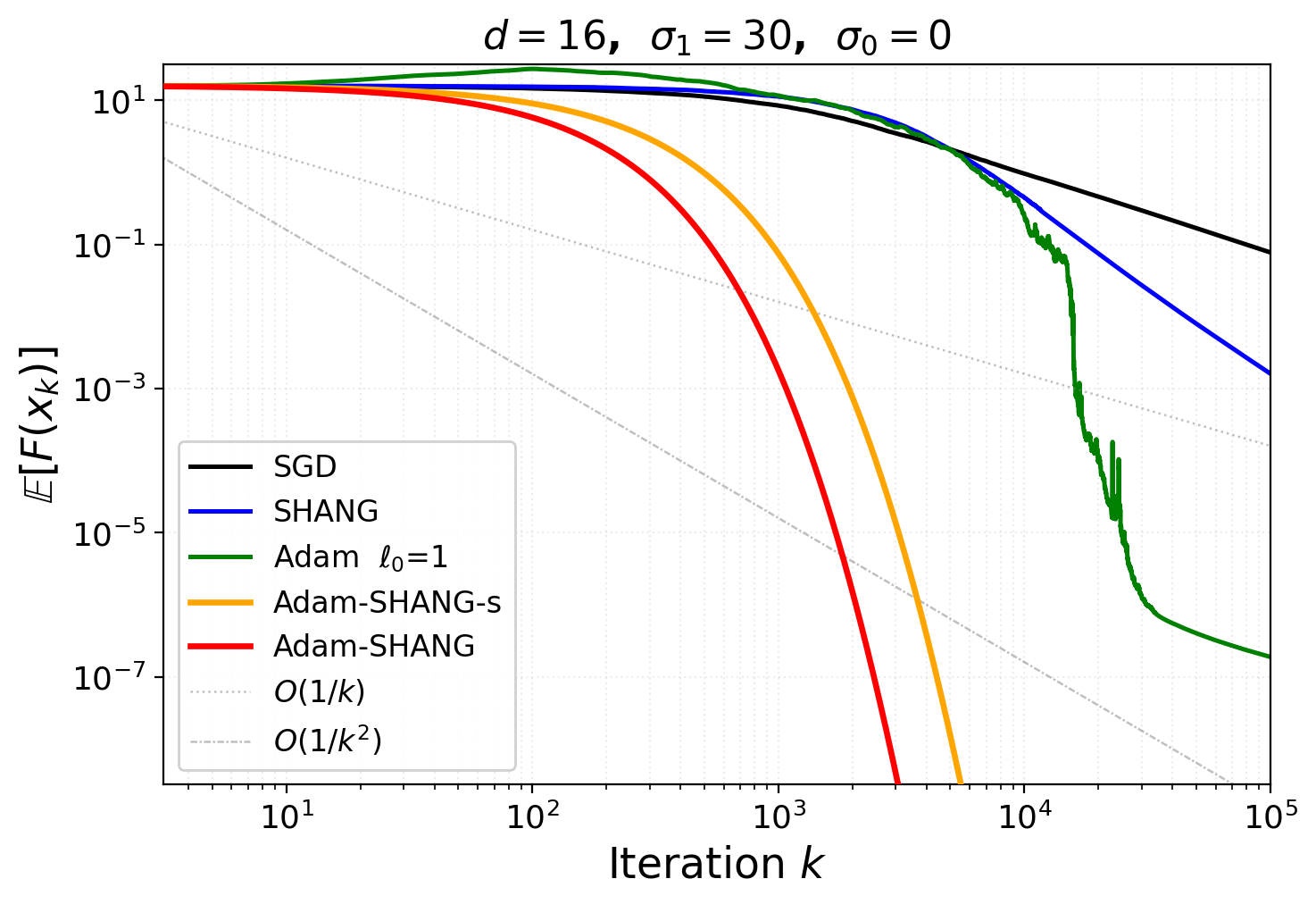}
	\end{subfigure}\hfill
	\centering
	\begin{subfigure}{0.32\textwidth}
		\centering
		\includegraphics[width=\linewidth]{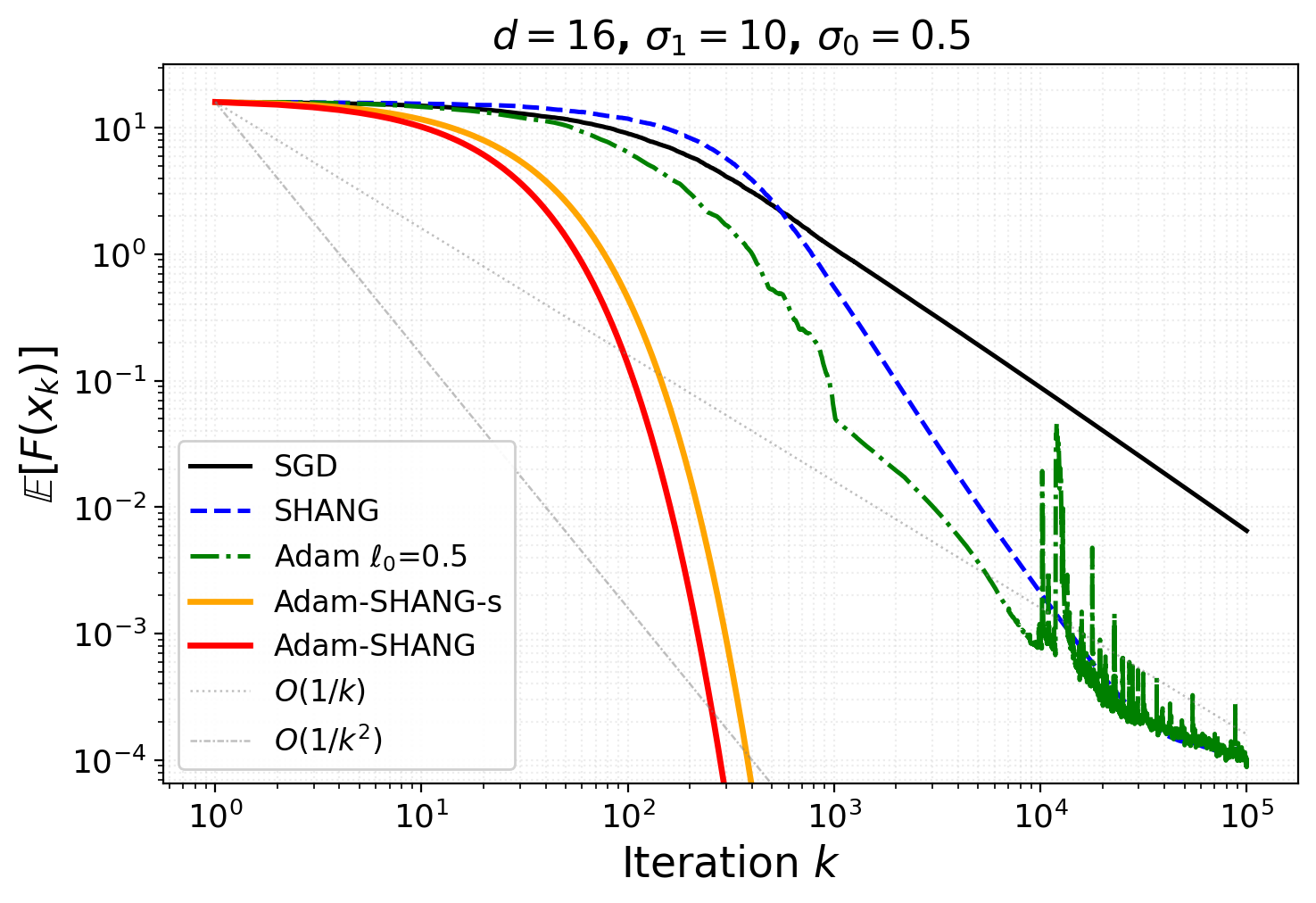}
	\end{subfigure}\hfill
	\begin{subfigure}{0.32\textwidth}
		\centering
		\includegraphics[width=\linewidth]{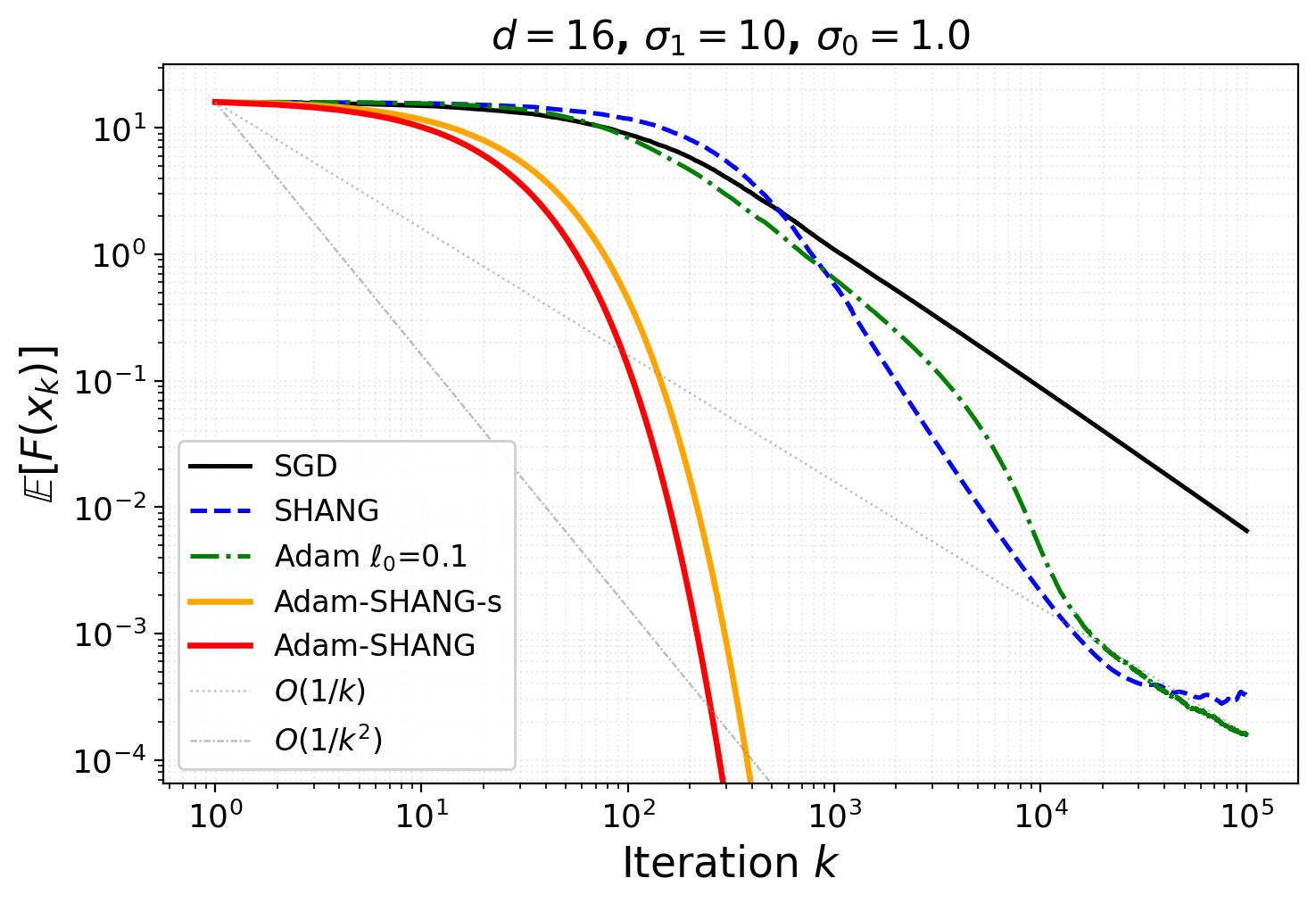}
	\end{subfigure}\hfill
	\begin{subfigure}{0.32\textwidth}
		\centering
		\includegraphics[width=\linewidth]{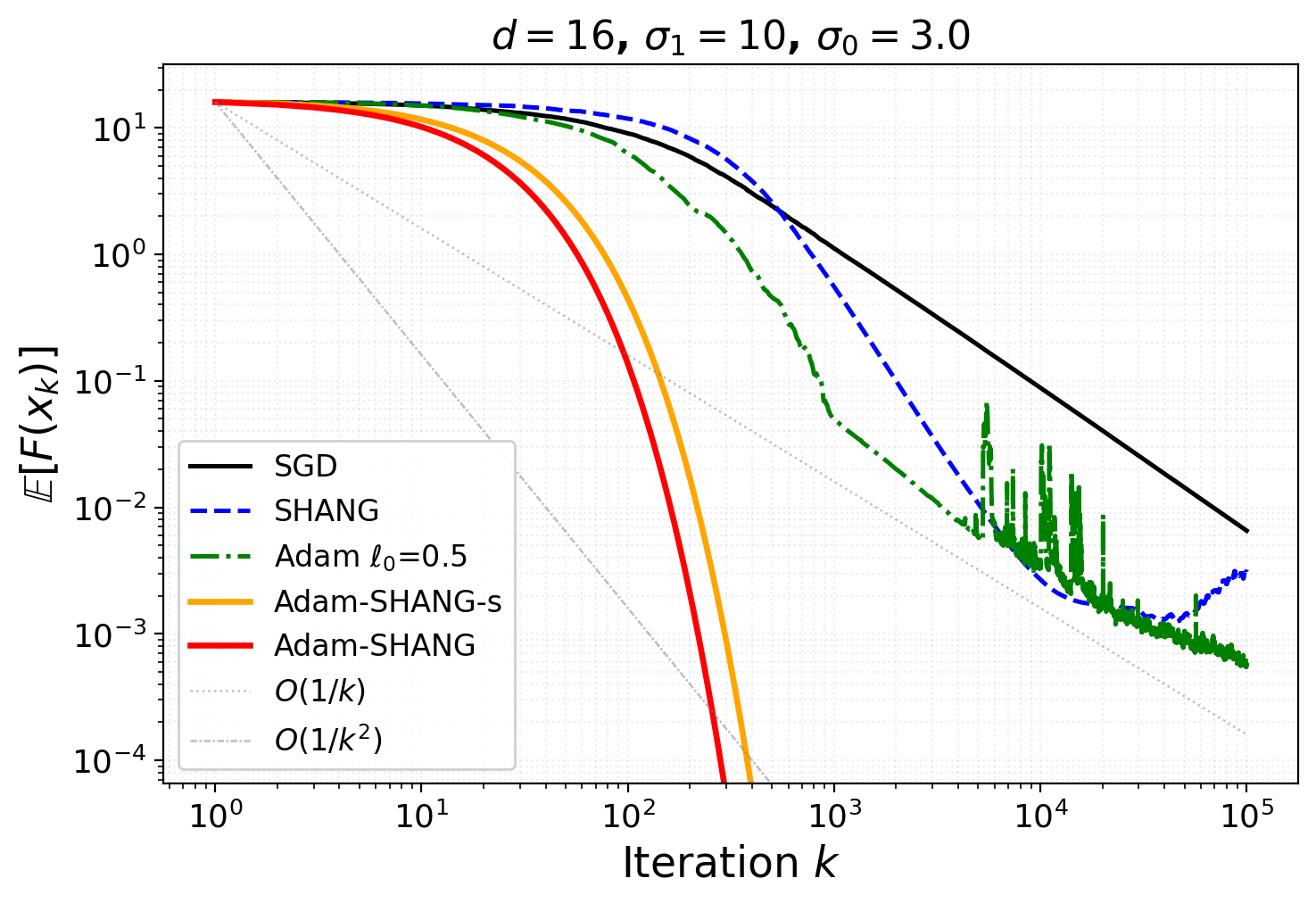}
	\end{subfigure}\hfill
	\caption{\small 
		Convex optimization benchmark with $d=16$. 
		\textbf{Top}: pure multiplicative noise with $\sigma_0=0$ and 
		$\sigma_1\in\{0,10,30\}$. 
		\textbf{Bottom}: additive-multiplicative noise with $\sigma_1=10$ and 
		$\sigma_0\in\{0.5,1,3\}$. 
		For Adam, we use the decaying stepsize $\ell_0/\sqrt{k+1}$, where $\ell_0$ is selected by grid search separately for each setting. 
		Full hyperparameter settings and implementation details are provided in Appendix~\ref{app:convex-example}.}
	\label{fig:convex result}
\end{figure}

Figure~\ref{fig:convex result} shows that, across all tested values of $\sigma_0$ and $\sigma_1$, both {\it Adam-SHANG} and {\it Adam-SHANG-s} decrease the objective faster than SGD, SHANG~\citep{yu2026shangrobuststochasticacceleration}, and Adam, with {\it Adam-SHANG} performing best. As the noise level increases, all methods slow down, but the ordering remains.

In this convex experiment, \emph{Adam-SHANG} uses the trace-ratio rule~\eqref{eq:alpha-trace} rather than the conservative spectral lower bound. Since Lemma~\ref{lem:stoch-descent} requires the admissibility condition~\eqref{equ:Q_kinlemma}, we check this condition directly \emph{a posteriori} along the computed trajectories. With the safety factor in the practical rule, we observe no violation of the monitored sufficient-decay condition in all tested settings; see Appendix~\ref{test:sufficient decay}. We also monitor the alignment condition~\eqref{assump:alignment}. Although it can fail during the initial transient phase, the admissibility condition still holds, confirming that the alignment condition is sufficient but not necessary.

\paragraph{Stress test: classical counterexample.}
We include the synthetic counterexample of \citet{reddi2019convergenceadam} as a controlled stress test. The domain is $x\in[-1,1]$, and the optimal solution is $x^\star=-1$. We consider both the deterministic sequence and the stochastic variant:
$$
f_t(x) =
\begin{cases}
1010x, & t \bmod 101 = 1,\\
-10x, & \text{otherwise,}
\end{cases}
\qquad
f_t(x) =
\begin{cases}
1010x, & \text{w.p. } 0.01,\\
-10x, & \text{otherwise.}
\end{cases}
$$
This benchmark is the canonical example of Adam's non-convergence and is used here only to test whether the proposed methods inherit the same pathology. Figure~\ref{fig:reddi} reports the iterate trajectories.

This benchmark is posed in the online setting, where the goal is to control the regret
$
R_T:=\sum_{t=1}^T f_t(x_t)-\min_x\sum_{t=1}^T f_t(x).
$
As shown by \citet{reddi2019convergenceadam}, low regret in this example does not imply convergence to the minimizer of the averaged objective. In both regimes, Adam diverges to $x=+1$, and {\it Adam-SHANG-s} shows similar instability. By contrast, {\it Adam-SHANG} converges to $x^\star=-1$ in both deterministic and stochastic settings and outperforms AMSGrad, especially in the stochastic case. This suggests that the feedback term $P_k^{-1}G_{k+1}^2$ provides a stabilization mechanism different from the monotone second-moment constraint of AMSGrad proposed in~ \citet{reddi2019convergenceadam}.

\begin{figure}[!htbp]
	\centering
	\begin{subfigure}{0.48\textwidth}
		\centering
		\includegraphics[width=0.8\linewidth]{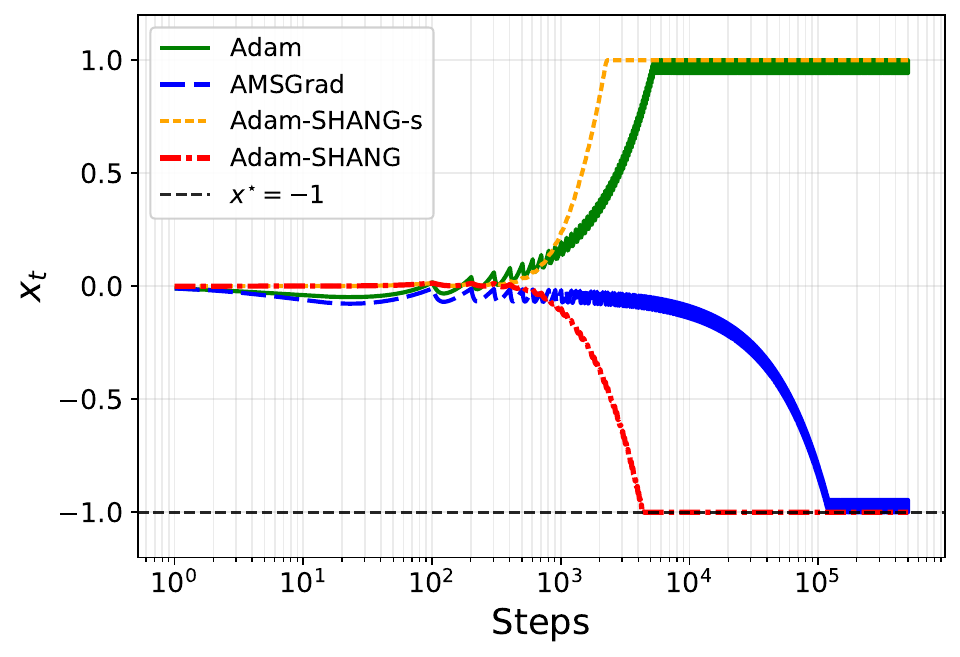}
	\end{subfigure}\hfill
	\begin{subfigure}{0.48\textwidth}
		\centering
		\includegraphics[width=0.8\linewidth]{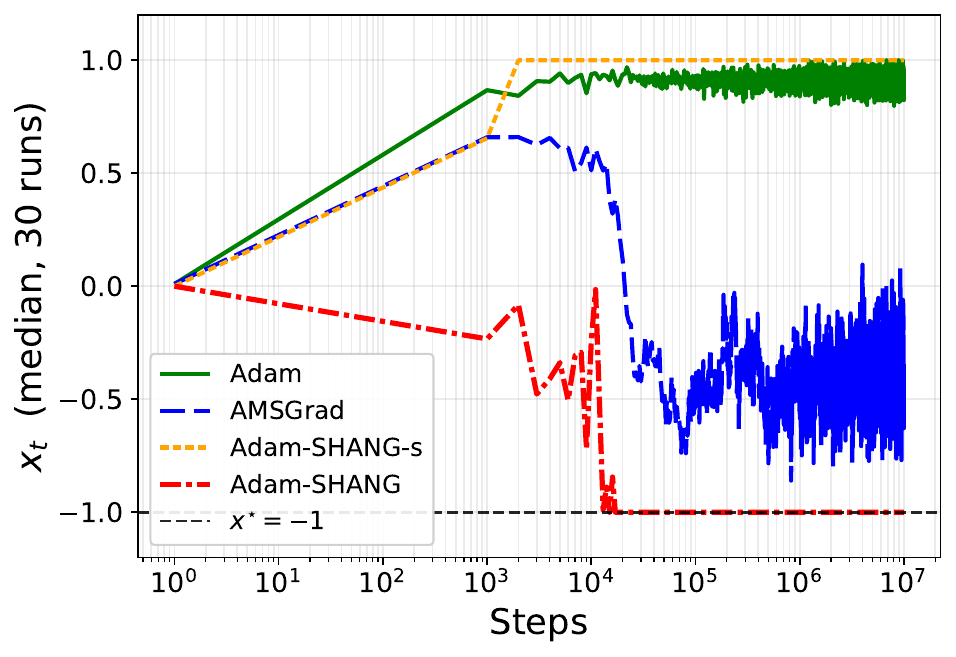}
	\end{subfigure}
	\caption{\small Iterate trajectories $x_t$ on the
		\citet{reddi2019convergenceadam} counterexample.
		Left: deterministic sequence. Right: stochastic sequence
		(median over 30 independent runs). Hyperparameter settings and full results, including average regret curves, are provided in  Appendix~\ref{app:reddi}.}
	\label{fig:reddi}
\end{figure}
\vspace{-6pt}

\paragraph{Deep Learning Tasks}

This section evaluates whether the proposed methods remain competitive on practical deep learning tasks and, in particular, whether the trace-ratio stepsize provides useful automatic scheduling. We compare {\it Adam-SHANG} in the form of Algorithm~\ref{alg:adamshang}, {\it Adam-SHANG-s} (Algorithm~\ref{alg:adamshangs}), SHANG++~\cite{yu2026shangrobuststochasticacceleration}, Adam~\cite{kingma2017adammethodstochasticoptimization}, and AdamW~\cite{loshchilov2019decoupledweightdecayregularization}. Each experiment is repeated with three independent random seeds, and the mean and standard deviation are reported.

\begin{figure}[!htbp]
	\centering
	\begin{subfigure}{0.48\textwidth}
		\centering
		\includegraphics[width=0.825\linewidth]{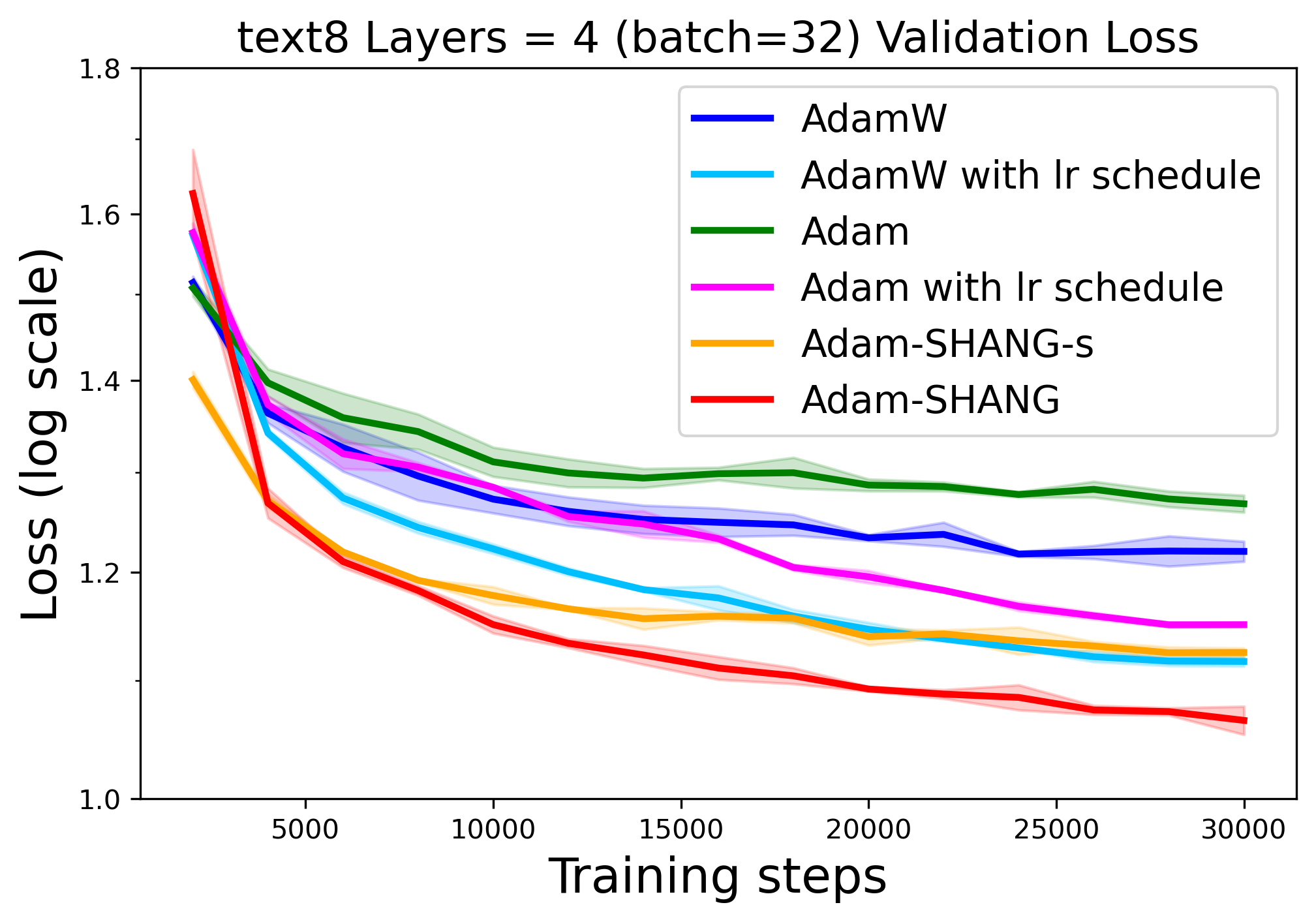}
	\end{subfigure}\hfill
	\begin{subfigure}{0.48\textwidth}
		\centering
		\includegraphics[width=0.825\linewidth]{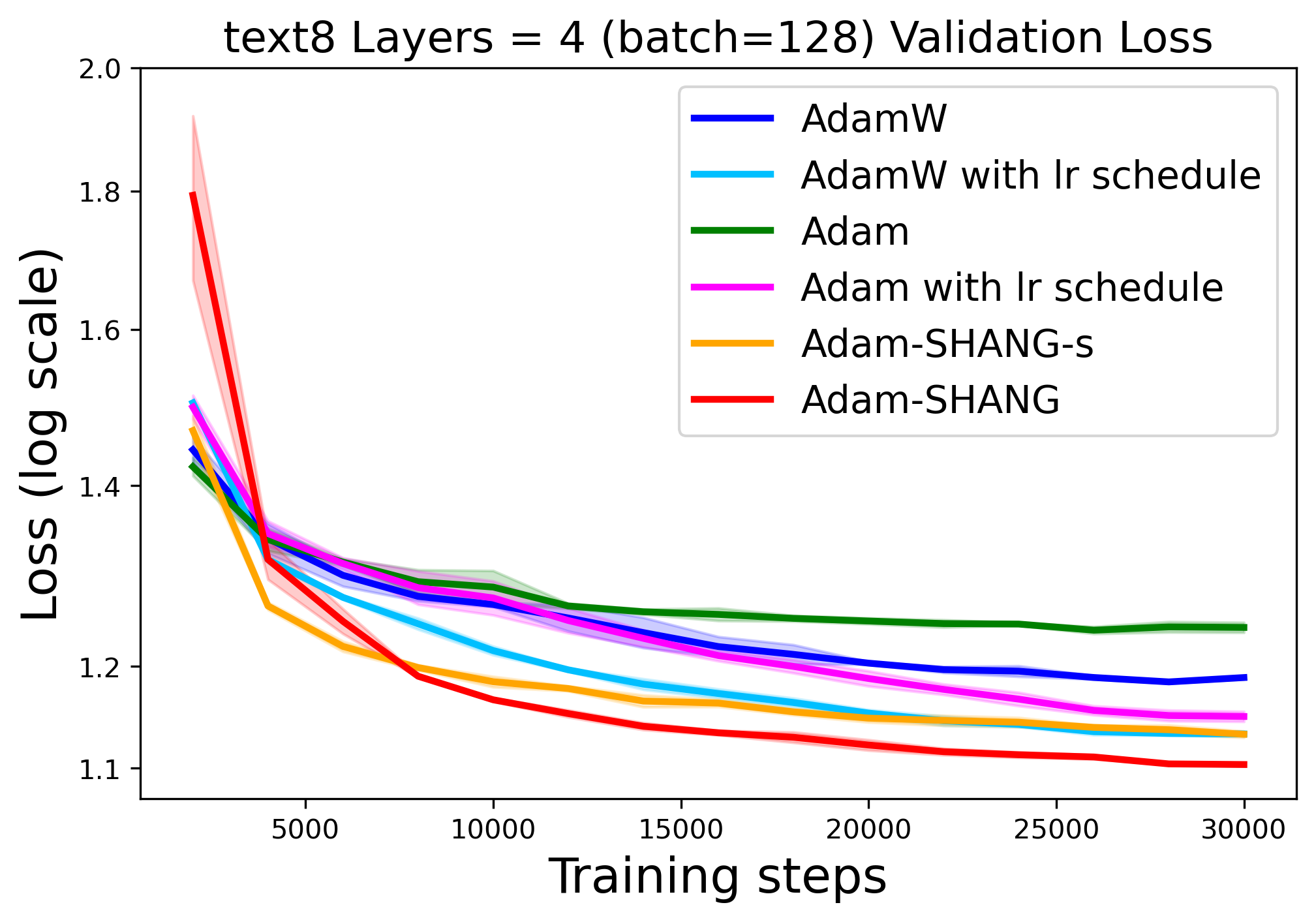}
	\end{subfigure}\hfill
	\caption{\small Validation loss for Transformer language modeling on \texttt{text8}. All methods are trained for $30{,}000$ steps under batch sizes 
		$B \in \{32, 128\}$ to examine sensitivity to gradient variance 
		and effective data throughput. Model architecture, data split, and hyperparameter settings are provided in Appendix~\ref{app:text8}.}
	\label{fig:transformer_text8}
\end{figure}
\vspace{-6pt}

\paragraph{Character-level language modeling on \texttt{text8}.}
We evaluate the proposed methods on character-level language modeling on \texttt{text8}~\cite{Mahoney2009}. Since SHANG++ is non-adaptive and is not competitive on this Transformer-based task, we exclude it from this comparison. Figure~\ref{fig:transformer_text8} compares all methods with and without a linear-warmup cosine-decay schedule. No external schedule is applied to {\it Adam-SHANG} or {\it Adam-SHANG-s}, since their stepsizes are determined by the trace-ratio rule and decay automatically through the adaptive metric. Both variants outperform the unscheduled baselines and remain competitive with the scheduled ones, with {\it Adam-SHANG} achieving the lowest validation loss for both batch sizes. The comparison with scheduled Adam and AdamW indicates that the trace-ratio rule provides effective learning-rate control without manual scheduling. The advantage is more pronounced at $B=32$, suggesting stronger benefits under higher gradient noise.

\paragraph{Image classification on CIFAR-100}\label{test:cifar100}
We further evaluate on image classification using ResNet-50~\cite{HeZhangRenSun2016} on CIFAR-100~\cite{Krizhevsky2009}. 
In contrast to the \texttt{text8} experiment, where scheduled Adam and AdamW are included because external learning-rate schedules have a pronounced effect on Transformer training, the CIFAR experiments are run under a common unscheduled protocol for all methods. This choice keeps the comparison focused on the intrinsic optimizer dynamics.

Figure~\ref{fig:cifar100_resnet} shows that both {\it Adam-SHANG} and {\it Adam-SHANG-s} achieve final test accuracies competitive with Adam and AdamW, while reaching high accuracy earlier. At batch size $32$, {\it Adam-SHANG} attains the baseline accuracy level in roughly half the number of steps. This suggests that the trace-ratio stepsize accelerates the transient optimization phase by extracting useful stepsize information from the evolving preconditioner. This is useful in staged training pipelines, where reaching a target accuracy with fewer steps can reduce the computational cost of each stage.

\begin{figure}[!htbp] 
	\centering
	\begin{subfigure}{0.48\textwidth}
		\centering
		\includegraphics[width=0.825\linewidth]{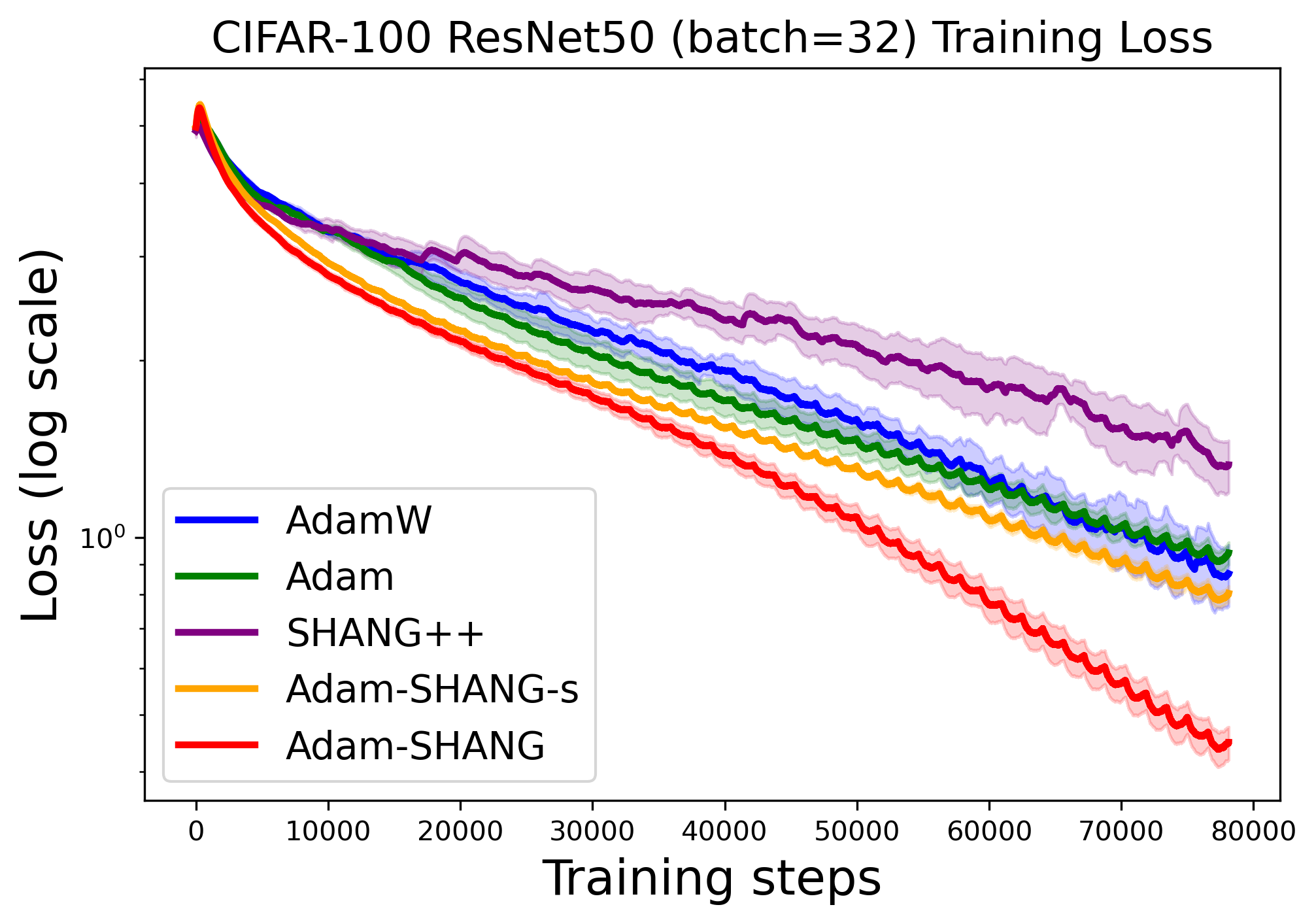}
	\end{subfigure}\hfill
	\begin{subfigure}{0.48\textwidth}
		\centering
		\includegraphics[width=0.825\linewidth]{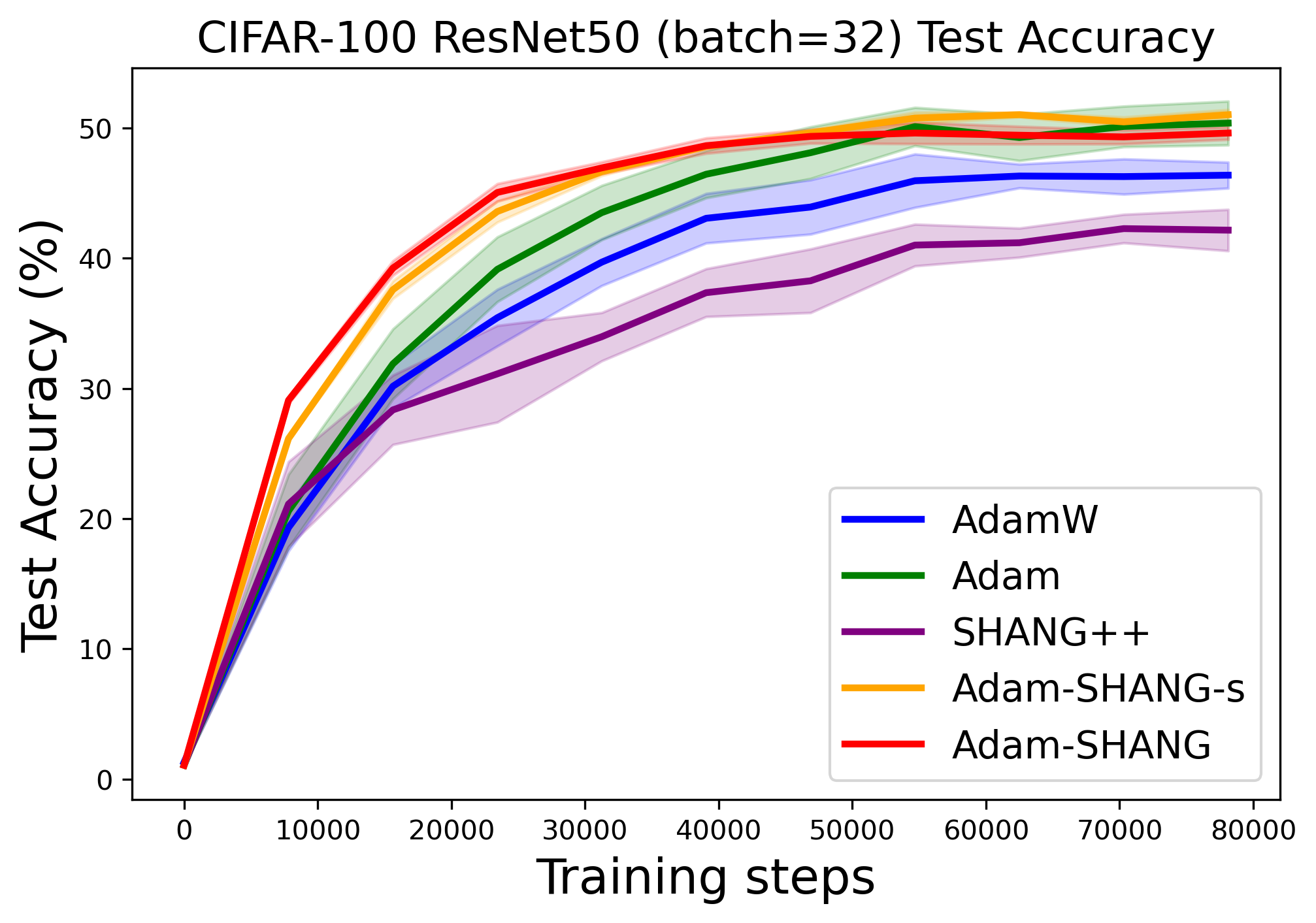}
	\end{subfigure}\hfill
	\begin{subfigure}{0.48\textwidth}
		\centering
		\includegraphics[width=0.825\linewidth]{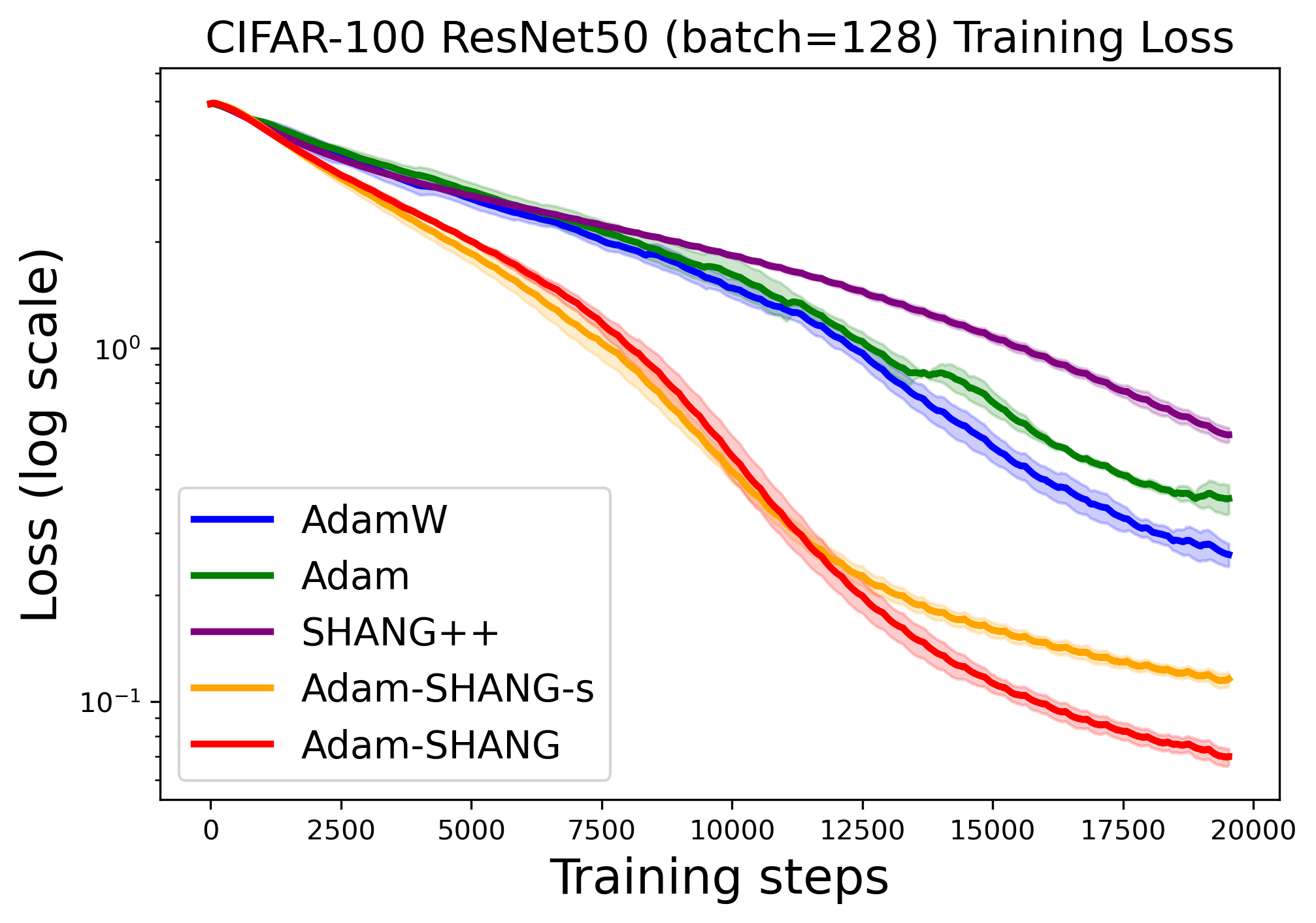}
	\end{subfigure}\hfill
	\begin{subfigure}{0.48\textwidth}
		\centering
		\includegraphics[width=0.825\linewidth]{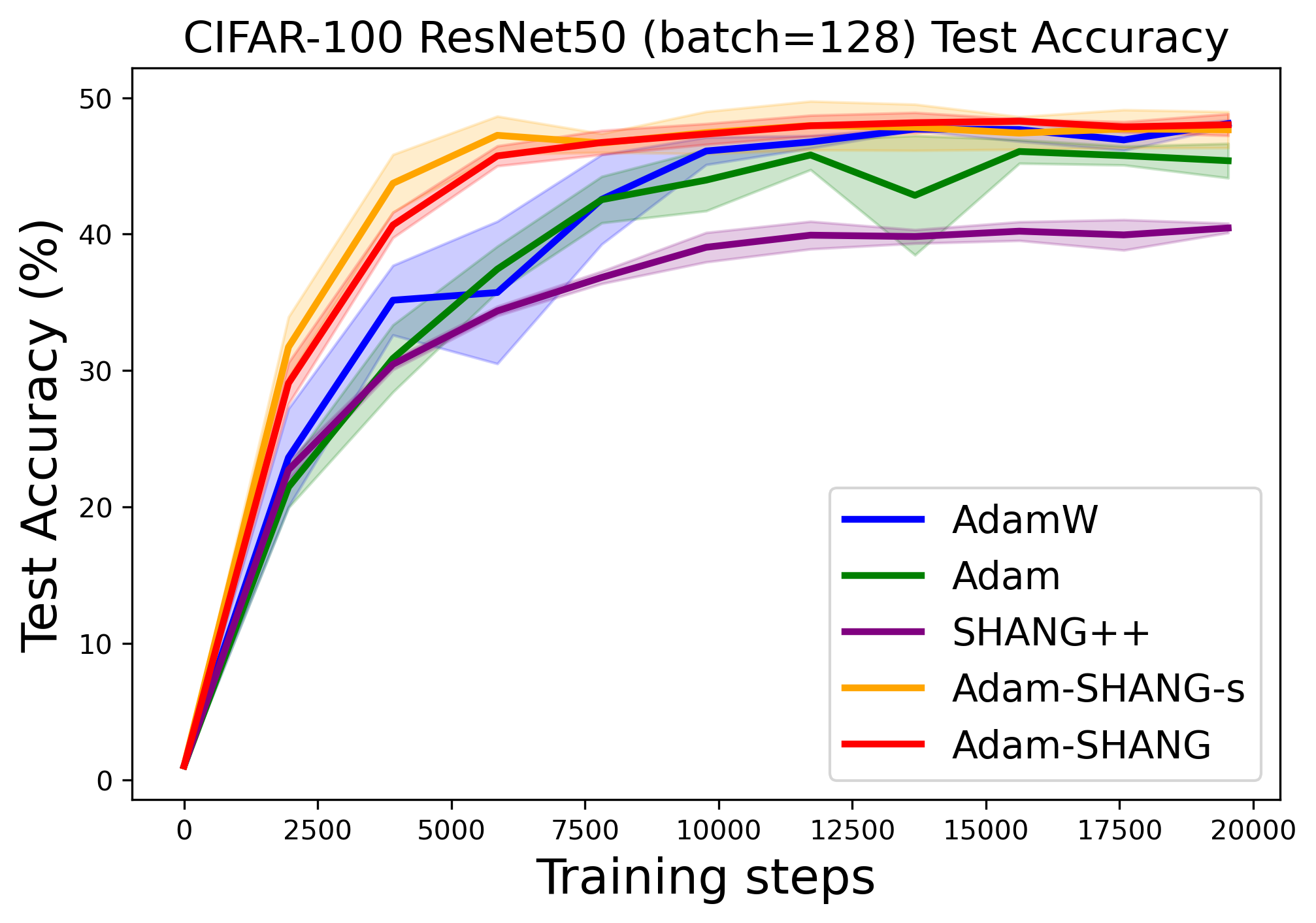}
	\end{subfigure}\hfill
	\caption{\small Training loss and test accuracy for training ResNet-50 on CIFAR-100. All models are trained for $50$ epochs under batch sizes $B \in \{32, 128\}$. Hyperparameter settings are provided in Appendix~\ref{app:cifar100}.}
	\label{fig:cifar100_resnet}
\end{figure}

This automatic decay is motivated by convex theory and is effective within the training horizons considered here. For longer nonconvex training, the trace-ratio rule could be combined with a lower bound on $\alpha_k$ and a later-stage schedule. This would preserve fast initial descent while keeping sufficiently large updates for fine-tuning. We leave this extension for future work.

We also include two supplementary evaluations in the appendix: CIFAR-10 classification with ResNet-34 (Appendix~\ref{test:cifar10}) and an ablation of the $\beta$-correction term (Appendix~\ref{test:ablation of correction}). The ablation uses both {\it Adam-SHANG} and {\it Adam-SHANG-s} with $(\beta=0,\gamma=1)$. The latter is the closest Adam-style counterpart and behaves similarly to Adam with learning rate changed due to the trace-ratio rule of $\alpha_k$. Comparing each full method with its $\beta=0$ variant shows that the correction term improves optimization, with a stronger effect for {\it Adam-SHANG}.

\section{Conclusion}
We proposed \emph{Adam-SHANG}, a Lyapunov-guided Adam-type method for stochastic smooth convex optimization. Its main algorithmic ingredient is a trace-ratio stepsize rule driven by the evolving preconditioner. Compared with Adam, the method combines an adaptive stepsize, a curvature-aware correction, and a splitting of momentum, leading to a scheme with both theoretical guarantees and improved training behavior. Compared with deterministic Adam-HNAG, the trace-ratio rule provides a computable stochastic surrogate for the unavailable full-gradient admissibility ratio. Compared with SHANG, the method adds adaptive feedback through the preconditioner update.

For stochastic smooth convex optimization, we proved convergence in expectation under an admissible stepsize condition, without imposing global monotonicity on the second-moment sequence. This condition can always be satisfied by a conservative spectral bound, while the trace-ratio rule gives a less conservative practical choice motivated by local coordinatewise alignment. Experiments on stochastic convex problems validate the predicted behavior, and deep learning experiments show competitive training performance against Adam and AdamW. Generalization, nonconvex theory, and weaker admissibility criteria remain important directions for future work.

\bibliographystyle{plainnat}
\bibliography{reference}


\appendix


\section{Convergence analysis of {\it Adam-SHANG}}

\subsection{Proof of Lemma~\ref{lem:stoch-descent}}
\label{app:proof-stoch-descent}

\begin{proof}[Proof of Lemma~\ref{lem:stoch-descent}]
	By $L$-smoothness and
	\(
	x_{k+1}^+-x_{k+1}
	=
	-\eta_{k+1}P_k^{-1}g_{k+1},
	\)
	we have
	\[
	\begin{aligned}
		f(x_{k+1}^+)
		\le\;&
		f(x_{k+1})
		-\eta_{k+1}
		\langle \nabla f(x_{k+1}),g_{k+1}\rangle_{P_k^{-1}}
		+\frac{L\eta_{k+1}^2}{2}\|g_{k+1}\|_{P_k^{-2}}^2.
	\end{aligned}
	\]
	Taking expectation and using that $\eta_{k+1}P_k^{-1}$ depends only on $x_{k+1}$ and the past randomness gives
	\[
	\begin{aligned}
		\mathbb{E}\!\left[f(x_{k+1}^+)\right]
		\le\;
		f(x_{k+1})
		-\eta_{k+1}
		\|\nabla f(x_{k+1})\|_{P_k^{-1}}^2
		+\frac{L\eta_{k+1}^2}{2}\,
		\mathbb{E}\!\left[\|g_{k+1}\|_{P_k^{-2}}^2\right].
	\end{aligned}
	\]
	By Assumption~\eqref{eq:sg-metric-secondmoment},
	\[
	\begin{aligned}
		\mathbb{E}\!\left[f(x_{k+1}^+)\right]
		\le\;
		f(x_{k+1})
		-\frac{\eta_{k+1}}{1+\sigma_1^2}
		\mathbb{E}\!\left[\|g_{k+1}\|_{P_k^{-1}}^2\right]
		+\frac{L\eta_{k+1}^2}{2}
		\mathbb{E}\!\left[\|g_{k+1}\|_{P_k^{-2}}^2\right]
		+\frac{\eta_{k+1}\sigma_0^2}{1+\sigma_1^2}.
	\end{aligned}
	\]
	Finally, substituting the step-size condition into the above inequality yields \eqref{eq:stoch_descent}.
\end{proof}

\subsection{Lemma \ref{lem:I3-bound}}
\label{app:proof-I3-stoch}

With the notation of $x_k^+ = x_k - \alpha_k \beta_k P_{k-1}^{-1} g_k$,  {\it Adam-SHANG} \eqref{eq:AdamSHANG} can be rewritten as
\begin{subequations}\label{eq:AdamSHANG2}
	\begin{align}
		\frac{x_{k+1}-x_k^+}{\alpha_k} &= y_k - x_{k+1}, \label{eq:AdamSHANG2-a}\\
		\frac{y_{k+1}-y_k}{\alpha_k} &= -P_k^{-1} g_{k+1}, \label{eq:AdamSHANG2-b}\\
		\frac{P_{k+1}-P_k}{\alpha_k} &= - P_{k+1} + \gamma_{k} P_k^{-1} G_{k+1}^2. \label{eq:AdamSHANG2-c}
	\end{align}
\end{subequations}
The point $x_k^+$ is only used in the analysis and does not change the algorithm. 

Equivalently, the scheme~\eqref{eq:AdamSHANG} can also be written in the form \eqref{eq:AdamSHANG3}.
\begin{subequations}\label{eq:AdamSHANG3}
	\begin{align}
		x_{k+1} & = \frac{1}{1+\alpha_k} x_k^+ + \frac{\alpha_k}{1+\alpha_k} y_k , \label{eq:AdamSHANG3-a}\\
		y_{k+1} & = y_k - \alpha_k P_k^{-1} g_{k+1}, \label{eq:AdamSHANG3-b}\\
		x_{k+1}^+ & = x_{k+1} - \alpha_{k+1} \beta_{k+1} P_k^{-1} g_{k+1}, \label{eq:AdamSHANG3-c}\\
		P_{k+1} & = \frac{1}{1+\alpha_k}  P_k + \frac{\alpha_k \gamma_k}{1+\alpha_k} P_k^{-1} G_{k+1}^2. \label{eq:AdamSHANG3-d}
	\end{align}
\end{subequations}

We now state the lemma.
\begin{lemma}\label{lem:I3-bound}
	Let $f\in\mathcal{S}_L$. Assume $\alpha_k$ is $\mathcal F_k$-measurable. Then
	\begin{equation}\label{eq:I3-bound}
		\mathbb{I}_3
		\le
		\mathbb{E}\!\left[
		-\alpha_k \big( f(x_{k+1}) - f(x^\star)\big)
		+\frac{\alpha_k^2}{2}\|g_{k+1}\|_{P_k^{-1}}^2
		\right].
	\end{equation}
\end{lemma}

\begin{proof}[Proof of Lemma~\ref{lem:I3-bound}]
	Notice that $\mathcal{E}(z_{k+1},P_k) = D_{\mathcal{E}}(z_{k+1}, z^\star; P_k)$ and $\mathcal{E}(z_{k}^+,P_k) = D_{\mathcal{E}}(z_{k}^+, z^\star; P_k)$. Using the three-point identity for the Bregman divergence together with the updates \eqref{eq:AdamSHANG2-a} and \eqref{eq:AdamSHANG2-b}, we expand
	\begin{equation}\label{equ9-app}
		\begin{aligned}
			\mathbb{I}_3 
			&= \mathbb{E}\!\left[
			\langle \nabla_z \mathcal{E}(z_{k+1},P_k),\, z_{k+1}-z_k^+\rangle
			- D_{\mathcal{E}}(z_k^+,z_{k+1};P_k) \right] \\
			&= \mathbb{E}\!\left[
			-\alpha_k \langle \nabla f(x_{k+1}),\, x_{k+1}-x^\star\rangle
			+\alpha_k \langle y_k-y_{k+1},\, g_{k+1}\rangle \right. \\
			&\quad \left.
			- D_{\mathcal{E}}(z_k^+,z_{k+1};P_k)
			+\alpha_k \langle \nabla f(x_{k+1})-g_{k+1},\, y_k-x^\star\rangle \right].
		\end{aligned}
	\end{equation}
	The last term vanishes in expectation by the unbiasedness assumption. Using the definition of the Bregman divergence and the convexity of $f$, we have
	\begin{equation}\label{equ5-app}
		-\langle\nabla f(x_{k+1}),x_{k+1}-x^\star\rangle
		\le
		-\big( f(x_{k+1}) - f(x^\star)\big),
	\end{equation}
	and
	\begin{equation}\label{equ8-app}
		-D_{\mathcal{E}}(z_k^+,z_{k+1};P_k) = -D_{f}(x_k^+,x_{k+1})-\frac{1}{2}\|y_k-y_{k+1}\|_{P_k}^2
		\le -\frac{1}{2}\|y_k-y_{k+1}\|_{P_k}^2.
	\end{equation}
	For the cross term, Cauchy--Schwarz yields
	\begin{equation}\label{equ10-app}
		\alpha_k \langle y_k-y_{k+1},\, g_{k+1}\rangle
		\le
		\frac{1}{2}\|y_k-y_{k+1}\|_{P_k}^2
		+\frac{\alpha_k^2}{2}\|g_{k+1}\|_{P_k^{-1}}^2.
	\end{equation}
	Substituting \eqref{equ5-app}, \eqref{equ8-app}, and \eqref{equ10-app} into \eqref{equ9-app}, we obtain the desired result.
\end{proof}

\subsection{Proof of Theorem~\ref{theorem:discrete stochastic}}
\label{app:proof-main-stoch}

\begin{proof}[Proof of Theorem~\ref{theorem:discrete stochastic}]
	Summing the bounds for $\mathbb{I}_1$, $\mathbb{I}_2$, $\mathbb{I}_3$, and using the definition of $\mathcal{E}(z_{k+1},P_{k+1})$, we obtain
	\[
	\begin{aligned}
		\mathbb{E}\!\left[\mathcal{E}(z_{k+1}^+,P_{k+1})-\mathcal{E}(z_k^+,P_k)\right]
		&\le
		\mathbb{E}\!\left[ -\alpha_k \mathcal{E}(z_{k+1},P_{k+1}) -\frac{\eta_{k+1}}{2(1+\sigma_1^2)}\|g_{k+1}\|_{P_k^{-1}}^2
		+\frac{\eta_{k+1}\sigma_0^2}{1+\sigma_1^2}
		\right. \\
		&\qquad
		\left.+\frac{\alpha_k\gamma_k}{2}\|y_{k+1}-x^\star\|_{P_k^{-1}G_{k+1}^2}^2
		+\frac{\alpha_k^2}{2}\|g_{k+1}\|_{P_k^{-1}}^2
		\right].
	\end{aligned}
	\]
	Applying Lemma~\ref{lem:stoch-descent} at $x_{k+1}$ and multiplying by $-\alpha_k$, we obtain
	\[
	\begin{aligned}
		\mathbb{E}\!\left[-\alpha_k\mathcal{E}(z_{k+1},P_{k+1})\right]
		&\le
		\mathbb{E}\!\left[
		-\alpha_k\mathcal{E}(z_{k+1}^+,P_{k+1})
		-\frac{\alpha_k\eta_{k+1}}{2(1+\sigma_1^2)}\|g_{k+1}\|_{P_k^{-1}}^2
		+\frac{\alpha_k \eta_{k+1}\sigma_0^2}{1+\sigma_1^2}
		\right].
	\end{aligned}
	\]
	Using $\sup_k\|y_k-x^\star\|_\infty\le R$ and $\gamma_k=\alpha_k/R^2$, we further have
	\[
	\frac{\alpha_k\gamma_k}{2}\|y_{k+1}-x^\star\|_{P_k^{-1}G_{k+1}^2}^2
	=
	\frac{\alpha_k^2 }{2R^2} \sum_{i=1}^d \big( y_{k+1,i} - x_i^\star \big)^2 \big(P_k^{-1}\big)_{ii} g_{k+1, i}^2
	\le
	\frac{\alpha_k^2}{2}\|g_{k+1}\|_{P_k^{-1}}^2.
	\]
	With the chosen parameters $2 \alpha_k^2 (1+\sigma_1^2) = \eta_{k+1}$,
	\[
	2\alpha_k^2-(1+\alpha_k)\frac{\eta_{k+1}}{1+\sigma_1^2}\le 0,
	\]
	so the $g_{k+1}$-weighted terms are nonpositive and can be dropped. Therefore,
	\[
	\mathbb{E}\!\left[\mathcal{E}(z_{k+1}^+,P_{k+1})\right]
	\le
	\mathbb{E}\!\left[
	\frac{1}{1+\alpha_k}\mathcal{E}(z_k^+,P_k)
	+2\alpha_k^2\sigma_0^2
	\right].
	\]
	This yields the claimed one-step bound, and iterating the recursion proves the theorem.
\end{proof}

\begin{remark}
For analytical tractability, we assume that there exists $R>0$ such that \[ \sup_{k\ge0}\|y_k-x^\star\|_\infty \le R. \]
We emphasize that this boundedness condition is a standard analytical hurdle common to adaptive algorithms \citep{duchi11a,gupta2018shampoo,liu2024adagradanisotropicsmoothness,an2025asgoadaptivestructuredgradient,kovalev2025,xie2025structuredpreconditionersadaptiveoptimization}. While $x^\star$ is unknown in practice, a common remedy is to incorporate a projection (or clipping) step onto a sufficiently large bounding box:
$$ 
\mathcal{H}_R=\{y\in\mathbb{R}^d:\|y\|_\infty\le R/2\},
$$
where $R/2 >\|x^\star\|_\infty$. Under this modification, the $y$-update becomes:
\begin{equation}
	\begin{aligned}
		y_{k+\frac{1}{2}} &= y_k - \alpha_k P_k^{-1} g_{k+1}, \\
		y_{k+1} &= \operatorname*{argmin}_{y\in\mathcal{H}_R} \|y - y_{k+\frac{1}{2}}\|_{P_{k}}^2. 
	\end{aligned}
\end{equation}
Because $\mathcal{H}_R$ is convex and $P_{k}\succ 0$, this weighted projection is well-defined. Since $x^\star \in \mathcal{H}_R$, the standard properties of the projection operator imply:
$$
\|y_{k+1} - x^\star \|_{P_{k}}^2 \le \|y_{k+\frac{1}{2}} - x^\star \|_{P_{k}}^2.
$$
This inequality shows that the projection can only decrease 
the Lyapunov
$\mathcal{E}(z_{k+1},P_k)$. Consequently, the bound in Lemma~\ref{lem:I3-bound} (derived for the unprojected step) remains a valid upper bound for the projected step, and the convergence proof of Theorem~\ref{theorem:discrete stochastic} remains intact while the boundedness condition is now satisfied by construction.
\end{remark}

\section{Proof and discussion of Corollary~\ref{cor:two-phase-stoch}}
\label{app:two-phase-stoch}

We first prove Corollary~\ref{cor:two-phase-stoch}, and then discuss the resulting two-phase behavior.

\begin{proof}[Proof of Corollary~\ref{cor:two-phase-stoch}]
	
	For convenience, define the pathwise quantities
	\[
	\rho_k
	:=\prod_{\tau=0}^k (1+\alpha_\tau)^{-1},
	\qquad
	\Psi_{k+1}
	:=\sum_{\tau=1}^{k+1}
	\Bigl(\prod_{j=\tau}^{k+1}(1+\alpha_j)^{-1}\Bigr)\alpha_\tau^2.
	\]
	Theorem~\ref{theorem:discrete stochastic} gives
	\[
	\mathbb{E}\!\left[\mathcal{E}(z_{k+1}^+,P_{k+1})\right]
	\le
	\mathbb{E}\!\left[
	\rho_k\,\mathcal{E}(z_0^+,P_0)+2\sigma_0^2\,\Psi_{k+1}
	\right].
	\]
	
	\textbf{Bounding the contraction term.}
	Since $\underline{\alpha}_\tau\le\alpha_\tau$ a.s.\ for every $\tau$ and the map
	$\alpha\mapsto(1+\alpha)^{-1}$ is strictly decreasing, we have the pathwise inequality
	\[
	\rho_k
	=\prod_{\tau=0}^k(1+\alpha_\tau)^{-1}
	\le\prod_{\tau=0}^k(1+\underline{\alpha}_\tau)^{-1}
	=\underline{\rho}_k.
	\]
	Since $\underline{\rho}_k$ is deterministic and $\mathcal{E}(z_0^+,P_0)$ is
	$\mathcal{F}_0$-measurable, taking expectations gives
	\[
	\mathbb{E}\!\left[\rho_k\,\mathcal{E}(z_0^+,P_0)\right]
	\le\underline{\rho}_k\,\mathbb{E}\!\left[\mathcal{E}(z_0^+,P_0)\right].
	\]
	
	\textbf{Bounding the noise accumulation term.}
	For each inner product in $\Psi_{k+1}$, applying the same monotonicity argument
	to the denominator ($\underline{\alpha}_j\le\alpha_j$ implies $(1+\alpha_j)^{-1}\le(1+\underline{\alpha}_j)^{-1}$)
	and the upper bound $\alpha_\tau\le\overline{\alpha}_\tau$ to the squared term, we obtain pathwise
	\[
	\Psi_{k+1}
	\le\sum_{\tau=1}^{k+1}
	\Bigl(\prod_{j=\tau}^{k+1}(1+\underline{\alpha}_j)^{-1}\Bigr)\overline{\alpha}_\tau^2
	=\overline{\Psi}_{k+1}.
	\]
	Combining the two bounds yields
	\[
	\mathbb{E}\!\left[\mathcal{E}(z_{k+1}^+,P_{k+1})\right]
	\le\underline{\rho}_k\,\mathbb{E}\!\left[\mathcal{E}(z_0^+,P_0)\right]
	+2\sigma_0^2\,\overline{\Psi}_{k+1},
	\]
	which is the stated inequality.
	
	\smallskip
	\textbf{Part (i): polynomial envelopes.}
	Suppose the random stepsizes $\{\alpha_k\}$ are almost surely sandwiched 
	between two polynomial sequences, i.e., there exist constants $a_->1$, 
	$a_+>0$, and $b>0$ such that
	\[
	\frac{a_-}{k+b}\le\alpha_k\le\frac{a_+}{k+b}
	\qquad\text{a.s. for all }k\ge 0.
	\]
	This envelope is consistent with the deterministic scaling 
	$\alpha_k\sim 2/k$ established in the 
	companion analysis~\citep{yu2026adamhnagconvergentreformulationadam}.
	
	\emph{Contraction term.}
	We express the telescoping product using the Gamma function.
	Noting that $1+\underline{\alpha}_\tau=(\tau+b+a_-)/(\tau+b)$, we have
	\[
	\underline{\rho}_k
	=\prod_{\tau=0}^k\frac{\tau+b}{\tau+b+a_-}
	=\frac{\Gamma(b+a_-)}{\Gamma(b)}\cdot\frac{\Gamma(k+b+1)}{\Gamma(k+b+a_-+1)}.
	\]
	By Stirling's approximation,
	$\Gamma(k+b+1)/\Gamma(k+b+a_-+1)\sim (k+b)^{-a_-}$ as $k\to\infty$,
	so
	\[
	\underline{\rho}_k\asymp(k+b)^{-a_-}.
	\]
	
	\emph{Inner product.}
	By the same Gamma-function identity,
	\[
	\prod_{j=\tau}^{k+1}(1+\underline{\alpha}_j)^{-1}
	=\frac{\Gamma(\tau+b+a_-)}{\Gamma(\tau+b)}
	\cdot\frac{\Gamma(k+b+2)}{\Gamma(k+b+2+a_-)}.
	\]
	Applying Stirling's approximation to each ratio gives
	\[
	\frac{\Gamma(\tau+b+a_-)}{\Gamma(\tau+b)}\asymp(\tau+b)^{a_-},
	\qquad
	\frac{\Gamma(k+b+2)}{\Gamma(k+b+2+a_-)}\asymp(k+b)^{-a_-},
	\]
	so the inner product is $\asymp\bigl((\tau+b)/(k+b)\bigr)^{a_-}$.
	
	\emph{Noise accumulation term.}
	Substituting the above estimate into $\overline{\Psi}_{k+1}$,
	\[
	\overline{\Psi}_{k+1}
	\asymp
	\frac{a_+^2}{(k+b)^{a_-}}
	\sum_{\tau=1}^{k+1}(\tau+b)^{a_--2}.
	\]
	Since $a_->1$ implies $a_--2>-1$, the sum is comparable to an integral via
	the integral comparison lemma:
	\[
	\sum_{\tau=1}^{k+1}(\tau+b)^{a_--2}
	\asymp
	\int_1^{k+1}(t+b)^{a_--2}\,dt
	=\frac{(k+b+1)^{a_--1}-(1+b)^{a_--1}}{a_--1}
	\asymp(k+b)^{a_--1}.
	\]
	Therefore,
	\[
	\overline{\Psi}_{k+1}
	\asymp(k+b)^{-a_-}\cdot(k+b)^{a_--1}
	=(k+b)^{-1},
	\]
	which proves part~(i).
	
	\smallskip
	\textbf{Part (ii): constant envelopes.}
	Let $\underline{\alpha}_k\equiv\underline{\alpha}>0$ and $\overline{\alpha}_k\equiv\overline{\alpha}>0$.
	
	\emph{Contraction term.}
	This is immediate: $\underline{\rho}_k=(1+\underline{\alpha})^{-(k+1)}$.
	
	\emph{Noise accumulation term.}
	The inner product simplifies to
	\[
	\prod_{j=\tau}^{k+1}(1+\underline{\alpha})^{-1}=(1+\underline{\alpha})^{-(k+2-\tau)}.
	\]
	Hence
	\[
	\overline{\Psi}_{k+1}
	=\overline{\alpha}^2\sum_{\tau=1}^{k+1}(1+\underline{\alpha})^{-(k+2-\tau)}.
	\]
	Substituting $m=k+2-\tau$ (so $m$ runs from $1$ to $k+1$) and summing the
	resulting geometric series,
	\[
	\overline{\Psi}_{k+1}
	=\overline{\alpha}^2\sum_{m=1}^{k+1}(1+\underline{\alpha})^{-m}
	=\overline{\alpha}^2\cdot
	\frac{(1+\underline{\alpha})^{-1}\bigl(1-(1+\underline{\alpha})^{-(k+1)}\bigr)}
	{1-(1+\underline{\alpha})^{-1}}
	=\frac{\overline{\alpha}^2}{\underline{\alpha}}
	\Bigl(1-(1+\underline{\alpha})^{-(k+1)}\Bigr).
	\]
	Substituting both expressions into the main bound yields part~(ii).
	As $k\to\infty$, $\overline{\Psi}_{k+1}\to\overline{\alpha}^2/\underline{\alpha}$,
	confirming that the iterates converge linearly up to a noise neighborhood.
\end{proof}

\paragraph{Interpretation of the two-phase behavior.}
Theorem~\ref{theorem:discrete stochastic} yields a Lyapunov bound of the form
\[
\mathbb{E}\!\left[\mathcal{E}(z_{k+1}^+,P_{k+1})\right]
\lesssim
\text{contraction term}
+
\text{noise accumulation term}.
\]
Corollary~\ref{cor:two-phase-stoch} makes this decomposition explicit by comparing
the realized adaptive stepsizes with deterministic envelopes.

We first recall the intuition from the deterministic companion paper
~\citep{yu2026adamhnagconvergentreformulationadam}. There, the diagonal preconditioner recursion can be interpreted as a balance between a dissipative term and a gradient-driven forcing term. Write $P_k = \mathrm{diag}(p_k)$. In the absence of forcing, the adaptive
scaling formally yields
\[
p_k\asymp k^{-2},\qquad \alpha_k\sim \frac{2}{k},
\qquad \rho_k\asymp k^{-2},
\]
recovering the standard accelerated decay. When the gradient forcing is
significant, it prevents the preconditioner from decaying too quickly and allows
larger admissible values of $\alpha_k$, leading to stronger transient
contraction. As the objective decreases and the gradients become smaller, this
forcing weakens and the dynamics gradually approach the unforced accelerated
regime.

The stochastic setting adds a second effect: the additive noise creates an
accumulation term in the Lyapunov estimate. Under polynomial envelopes
\[
\alpha_k \asymp \frac{1}{k+b},
\]
Corollary~\ref{cor:two-phase-stoch} shows that the contraction term decays like
$(k+b)^{-a_-}$, while the accumulated stochastic term decays only like
$(k+b)^{-1}$. Since $a_->1$, the early behavior is governed by the contraction
term, but the asymptotic rate is limited by noise accumulation. The crossover
iteration is determined by the balance
\[
(k^\star+b)^{a_--1}
\asymp
\frac{\mathcal{E}(z_0^+,P_0)}{\sigma_0^2}.
\]
Thus the polynomial regime is most visible when the initial energy is large
relative to the additive noise level.

When the iterates reach the noise-dominated region, the gradient signal no
longer decays to zero in an effective sense. The adaptive stepsize may then be
bounded below by the residual stochastic forcing, and this behavior is captured
by the constant-envelope case in Corollary~\ref{cor:two-phase-stoch}. In this
regime, the contraction term is geometric, while the stochastic accumulation
saturates at a nonzero level of order
\[
\sigma_0^2\,\frac{\overline{\alpha}^2}{\underline{\alpha}}.
\]
Therefore the method contracts linearly up to a noise neighborhood.

Overall, the picture is as follows. In the early stage, large gradients act as
forcing in the preconditioner recursion and can support larger adaptive
stepsizes, strengthening the transient contraction. As the gradients decrease, this forcing weakens, and in the noiseless limit the dynamics are expected to approach the weak-forcing regime where $\alpha_k\sim 2/k$ and the contraction resembles the deterministic accelerated behavior. In the stochastic setting, however, additive noise prevents this deterministic decay from persisting indefinitely: the polynomial-envelope bound is eventually limited by the $\mathcal{O}(1/k)$ noise-accumulation term, and under persistent residual fluctuations the constant-envelope description gives linear contraction only up to a noise neighborhood.
 
\section{Proof of Lemma~\ref{lem:trace-lower-bound}}
\label{app:practical-alpha}

This appendix justifies the surrogate
\[
\frac{\mathrm{Tr}(P_k^{-1})}{\mathrm{Tr}(P_k^{-2})}
\]
used in Section~\ref{sec:practical alpha} for the ratio
\[
q_k
=
\frac{\sum_{i=1}^d p_{k,i}^{-1} s_{k,i}^2}
{\sum_{i=1}^d p_{k,i}^{-2} s_{k,i}^2},
\qquad
P_k=\mathrm{diag}(p_{k,1},\dots,p_{k,d}),
\qquad
s_{k,i}^2=\mathbb{E}[g_{k+1,i}^2\mid\mathcal F_k].
\]
The argument is based on the following weighted Chebyshev sum inequality~\citep{mitrinovic1993classical}.

\begin{lemma}[Weighted Chebyshev sum inequality]\label{lemma:chebyshev}
	Let $\{w_i\}_{i=1}^d$, $\{a_i\}_{i=1}^d$, and $\{b_i\}_{i=1}^d$ be positive sequences. Assume that $a_i$ and $b_i$ are similarly ordered, namely
	\[
	(a_i-a_j)(b_i-b_j)\ge 0,
	\qquad \forall\, i,j\in\{1,\dots,d\}.
	\]
	Then
	\[
	\Big(\sum_{i=1}^d w_i a_i b_i\Big)\Big(\sum_{j=1}^d w_j\Big)
	\ge
	\Big(\sum_{i=1}^d w_i a_i\Big)\Big(\sum_{j=1}^d w_j b_j\Big).
	\]
\end{lemma}

\begin{proof}
	Consider
	\[
	\sum_{i,j=1}^d w_i w_j (a_i-a_j)(b_i-b_j).
	\]
	Expanding the product gives
	\[
	\begin{aligned}
		\sum_{i,j=1}^d w_i w_j (a_i-a_j)(b_i-b_j)
		&=
		\sum_{i,j=1}^d w_i w_j (a_i b_i + a_j b_j - a_i b_j - a_j b_i) \\
		&=
		\sum_{i,j=1}^d w_i w_j (a_i b_i + a_j b_j) - \sum_{i,j=1}^d w_i w_j (a_i b_j +a_j b_i) \\
		&=
		2\Big(\sum_{i=1}^d w_i a_i b_i\Big)\Big(\sum_{j=1}^d w_j\Big)
		-2\Big(\sum_{i=1}^d w_i a_i\Big)\Big(\sum_{j=1}^d w_j b_j\Big).
	\end{aligned}
	\]
	If $(a_i-a_j)(b_i-b_j)\ge 0$ for all $i,j$, then the left-hand side is nonnegative, and the conclusion follows.
\end{proof}

\begin{proof}[Proof of Lemma~\ref{lem:trace-lower-bound}]
	The desired inequality is equivalent to
	\[
	\Big(\sum_{i=1}^d p_{k,i}^{-1} s_{k,i}^2\Big)\Big(\sum_{i=1}^d p_{k,i}^{-2}\Big)
	\ge
	\Big(\sum_{i=1}^d p_{k,i}^{-2} s_{k,i}^2\Big)\Big(\sum_{i=1}^d p_{k,i}^{-1}\Big).
	\]
	Rewriting the first factor on the left and the second factor on the right gives
	\[
	\Big(\sum_{i=1}^d p_{k,i}^{-2}\, p_{k,i} s_{k,i}^2\Big)\Big(\sum_{i=1}^d p_{k,i}^{-2}\Big)
	\ge
	\Big(\sum_{i=1}^d p_{k,i}^{-2} s_{k,i}^2\Big)\Big(\sum_{i=1}^d p_{k,i}^{-2} p_{k,i}\Big).
	\]
	Applying Lemma~\ref{lemma:chebyshev} with
	\[
	w_i=p_{k,i}^{-2},
	\qquad
	a_i=s_{k,i}^2,
	\qquad
	b_i=p_{k,i},
	\]
	and using the assumption that $\{s_{k,i}^2\}_{i=1}^d$ and $\{p_{k,i}\}_{i=1}^d$ are similarly ordered, we obtain
	\[
	\Big(\sum_{i=1}^d p_{k,i}^{-2}\, p_{k,i} s_{k,i}^2\Big)\Big(\sum_{i=1}^d p_{k,i}^{-2}\Big)
	\ge
	\Big(\sum_{i=1}^d p_{k,i}^{-2} s_{k,i}^2\Big)\Big(\sum_{i=1}^d p_{k,i}^{-2} p_{k,i}\Big),
	\]
	which is exactly the claimed bound.
\end{proof}

\section{Derivation of the Adam-form representation of {\it Adam-SHANG-s}}
\label{app:adamshangs-adam-form}

Ignoring bias correction, Adam iterates as
\begin{equation}\label{eq:adam}
	\left\{
	\begin{aligned}
		x_{k+1} &= x_k - \eta (\sqrt{V_k}+\varepsilon I)^{-1} m_k,\\
		m_{k+1} &= \theta_1 m_k + (1-\theta_1) g_{k+1}, \\
		V_{k+1} &= \theta_2 V_k + (1-\theta_2) G^2_{k+1},
	\end{aligned}
	\right.
\end{equation}
where $\eta>0$ is the learning rate, $\varepsilon>0$ prevents division by zero, $\theta_1,\theta_2\in(0,1)$ are momentum parameters, and $G^2_{k+1}=\mathrm{diag}(g_{k+1}.^2)$.

We now derive the equivalent Adam-form representation \eqref{eq:adamshangs-adam-form} of {\it Adam-SHANG-s}.

\begin{proof}
	Starting from \eqref{eq:AdamSHANGs-a}, we have
	\begin{equation}\label{eq:adamshangs-app-x}
		x_{k+1}-x_k
		=
		-\frac{\alpha_k}{1+\alpha_k}(x_k-y_k)
		-\frac{\alpha_k\beta_k}{1+\alpha_k}P_k^{-1}g_k.
	\end{equation}
	Moreover, substituting $\tilde{\alpha}_k=\frac{\alpha_k}{1+\alpha_k}$ into \eqref{eq:AdamSHANGs-b} gives
	\begin{equation}\label{eq:adamshangs-app-y}
		\frac{y_{k+1}-y_k}{\alpha_k}
		=
		-(y_{k+1}-y_k)-P_{k+1}^{-1}g_{k+1}.
	\end{equation}
	
	Define $m_k$ by
	\[
	P_k^{-1}m_k=x_k-y_k.
	\]
	Combining \eqref{eq:adamshangs-app-x} and \eqref{eq:adamshangs-app-y}, we obtain
	\[
	\frac{m_{k+1}-m_k}{\alpha_k}
	=
	\frac{P_{k+1}-P_k}{\alpha_k}P_k^{-1}m_k
	-m_{k+1}
	+g_{k+1}
	-\beta_k P_{k+1}P_k^{-1}g_k.
	\]
	Substituting \eqref{eq:AdamSHANGs-c} yields
	\[
	m_{k+1}
	=
	\frac{1}{1+\alpha_k}P_{k+1}P_k^{-1}m_k
	+
	\frac{\alpha_k}{1+\alpha_k}g_{k+1}
	-
	\frac{\alpha_k\beta_k}{1+\alpha_k}P_{k+1}P_k^{-1}g_k.
	\]
	Finally, using $P_k=\sqrt{V_k}$ gives
	\[
	V_{k+1}
	=
	\frac{1}{1+\alpha_k}P_{k+1}P_k^{-1}V_k
	+
	\frac{\alpha_k}{1+\alpha_k}\gamma_k G_{k+1}^2.
	\]
	Together with \eqref{eq:adamshangs-app-x}, this proves \eqref{eq:adamshangs-adam-form}.
\end{proof}

When $P_{k+1}P_k^{-1}\approx I$, setting $\beta_k=0$ and $\gamma_k=1$ recovers an Adam-type update without bias correction. In this sense, the term $\beta_k(\sqrt{V_k})^{-1}g_k$ acts as a curvature-aware correction to Adam.

\section{Supplement of Experiments}\label{completetest}
Here are some experimental results that are not presented in the main text.

\begin{algorithm}[t]
	\caption{{\it Adam-SHANG-s}}
	\label{alg:adamshangs}
	
	\begin{algorithmic}[1]
		
	\REQUIRE Initial parameters $x_0$, $y_0$, $P_{-1}=P_0\succ0$, positive parameters $\lambda, \beta, \gamma, \varepsilon>0$
	   
		\FOR{$k = 0,1,2,\ldots$}
		
		\STATE Compute step size:
		\[
		\alpha_k =\lambda \sqrt{\frac{\mathrm{Tr}((P_{k}  + \varepsilon I)^{-1})}{\mathrm{Tr}((P_{k}  + \varepsilon I)^{-2})}},  \qquad \tilde{\alpha}_{k}  = \frac{\alpha_k}{1+\alpha_k}
		\]
		
		\STATE Update iterate
		\[
		x_{k+1} = \frac{1}{1+\alpha_k}x_k + \frac{\alpha_k}{1+\alpha_k} y_{k} - \frac{\alpha_{k}}{1+\alpha_k} \beta (P_{k} + \varepsilon I)^{-1} g_k
		\]
		
		\STATE Compute a stochastic gradient estimator $g_{k+1}$ of $\nabla f(x_{k+1})$
		
		\STATE Update preconditioner
		\[
		P_{k+1}
		=
		\frac{1-\tilde{\alpha}_{k} }{2} P_{k}
		+
		\frac{1}{2}\sqrt{(1-\tilde{\alpha}_{k} )^2 P_{k}^2
			+4\tilde{\alpha}_{k}  \gamma \mathrm{diag}(g_{k+1}.^2)}
		\]
		
		\STATE Update iterate
		\[
		y_{k+1} = y_{k} -\tilde{\alpha}_{k} (P_{k+1} + \varepsilon I)^{-1} g_{k+1}
		\]

		\ENDFOR
		
	\end{algorithmic}
\end{algorithm}

\subsection{Hyperparameter settings for convex optimization}
\label{app:convex-example}

We compare {\it Adam-SHANG}~\eqref{eq:AdamSHANG} and {\it Adam-SHANG-s}~\eqref{eq:AdamSHANGs} with SGD, SHANG~\citep{yu2026shangrobuststochasticacceleration}, and Adam~\citep{kingma2017adammethodstochasticoptimization}.

For SGD, we use the theoretically optimal stepsize
\(
\eta=\frac{1}{(1+\sigma_1^2)L}.
\)

For SHANG, we set $\alpha_k=2/(k+1)$, $P_k=\alpha_k^2(1+\sigma_1^2)^2L$, and
\(
\alpha_k\beta_k P_{k-1}=\frac{1}{(1+\sigma_1^2)L},
\)
which corresponds to the scalar-$P_k$, $\gamma_k=0$ instance of~\eqref{eq:AdamSHANG}.

For Adam, we use the decaying stepsize
\(
\eta_k=\frac{\ell_0}{\sqrt{k+1}},
\)
with $(\beta_1,\beta_2)=(0.9,0.999)$, where $\ell_0$ is tuned by grid search for each setting.

For {\it Adam-SHANG}, we set
\(
\eta_{k+1}=\alpha_k\beta_k=2\alpha_k^2(1+\sigma_1^2),
\gamma_k=\frac{\alpha_k}{R^2},
\)
and choose
\[
\alpha_k=\frac{0.5}{1+\sigma_1^2}
\sqrt{\frac{\mathrm{Tr}(P_k^{-1})}{2L\,\mathrm{Tr}(P_k^{-2})}}.
\]

For {\it Adam-SHANG-s}, we set $\tilde{\alpha}_k=\alpha_k/(1+\alpha_k)$,
\(
\eta_{k+1}=\alpha_k\beta_k=3\tilde{\alpha}_k^2(1+\sigma_1^2), \gamma_k=\frac{\tilde{\alpha}_k}{2R^2},
\)
and choose
\[
\alpha_k=\frac{0.5}{1+\sigma_1^2}
\sqrt{\frac{\mathrm{Tr}(P_k^{-1})}{6L\,\mathrm{Tr}(P_k^{-2})}}.
\]

In both methods, since $x^\star=0$ is known, we update
\[
R_k=\sup_{0\le j\le k}\|y_j\|_\infty
\]
dynamically and use $R=R_k$ in the parameter rules above.

Results are averaged over $200$ independent runs of length $T=10^5$. 

\subsection{Empirical Verification}\label{test:sufficient decay}
\paragraph{Admissibility Condition \eqref{equ:Q_kinlemma}}
In the convex experiments, we instantiate the admissible stepsize using the trace-ratio rule~\eqref{eq:alpha-trace} instead of the conservative spectral lower bound. 
This raises an \emph{a posteriori} question: can the practical choice violate the inequality needed in Lemma~\ref{lem:stoch-descent}? To examine this, we evaluate the condition numerically along the computed trajectory. For {\it Adam-SHANG}, Lemma~\ref{lem:stoch-descent} requires 
\(
\eta_{k+1}=2\alpha_k^2(1+\sigma_1^2) \le \frac{q_k}{(1+\sigma_1^2)L},
\)
or equivalently,
\begin{equation}
	\mathrm{Ratio}
	:=\frac{ \frac{q_k}{(1+\sigma_1^2)L}}{ 2\alpha_k^2(1+\sigma_1^2) }
	\ge 1.
\end{equation}
The admissibility condition is satisfied whenever \(\mathrm{Ratio}\ge 1\).

\begin{figure}[!htbp]
	\centering
	\begin{subfigure}{0.48\textwidth}
		\centering
		\includegraphics[width=\linewidth]{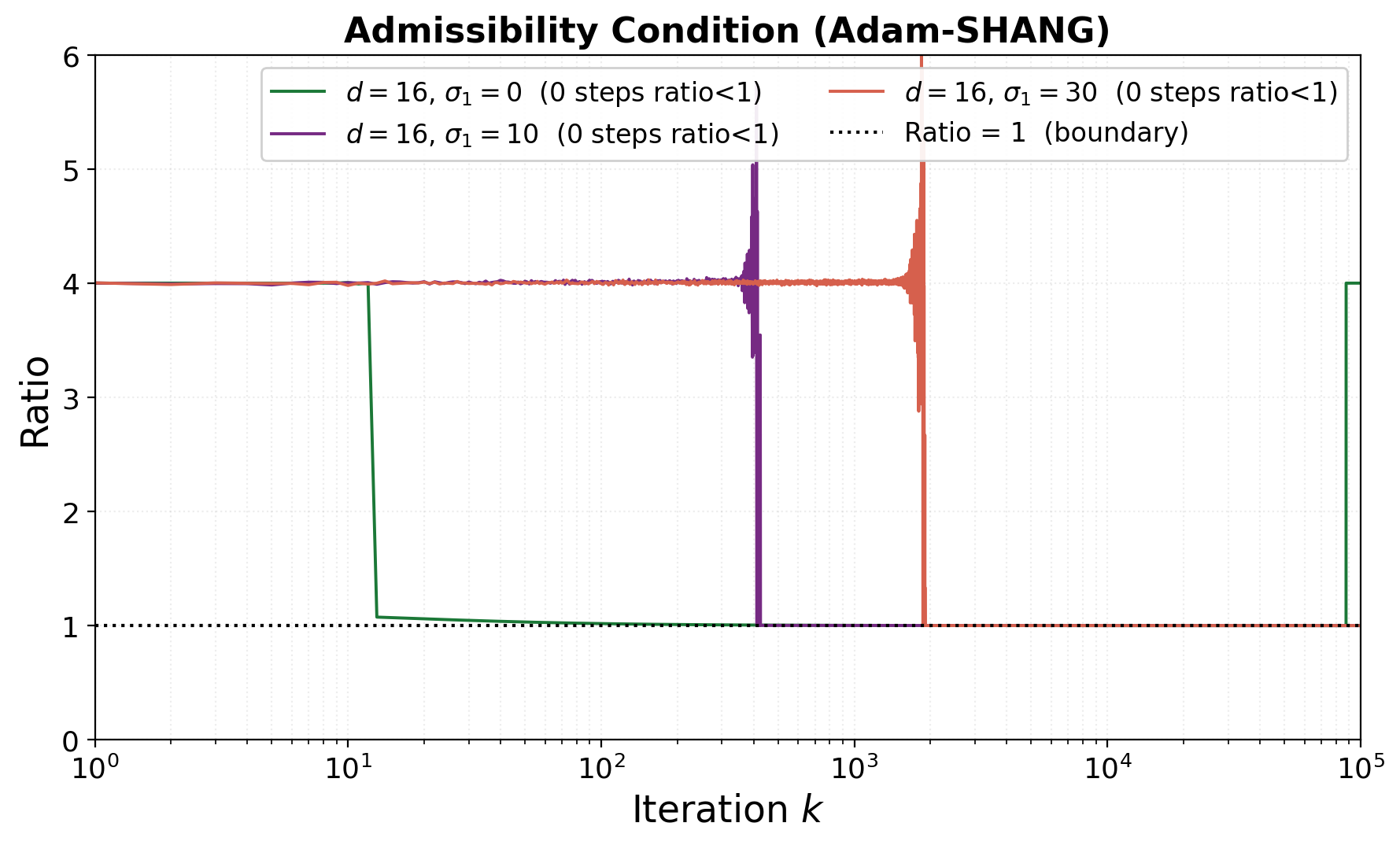}
	\end{subfigure}\hfill
	\begin{subfigure}{0.48\textwidth}
		\centering
		\includegraphics[width=\linewidth]{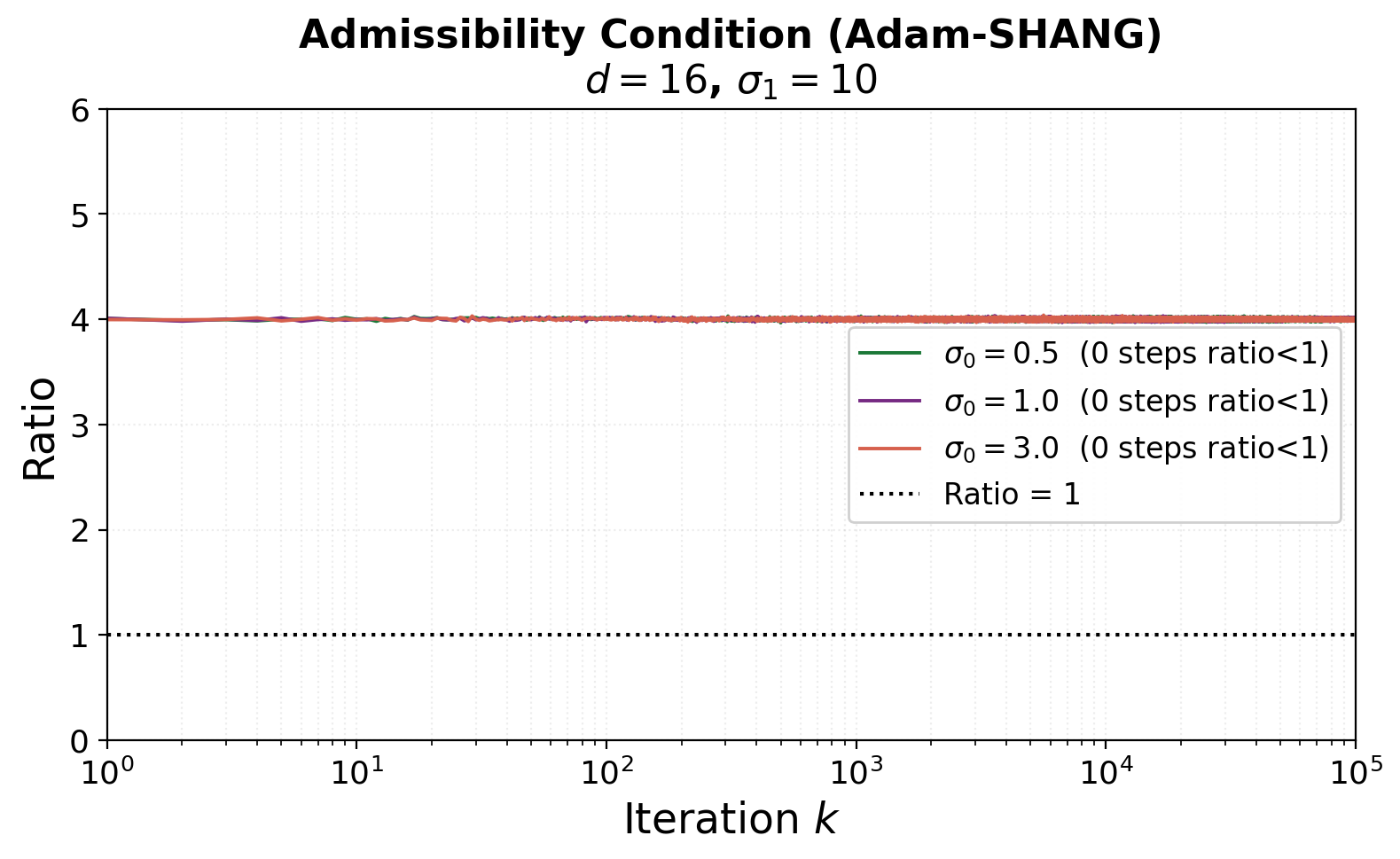}
	\end{subfigure}\hfill
	\caption{Empirical verification of the admissibility condition for {\it Adam-SHANG}. The ratio should be not less than $1$. }
	\label{fig:consistency}
\end{figure}

\paragraph{Observation.}
In all tested cases, the monitored ratio remains above $1$ throughout the whole trajectory. Thus, with the safety factor $\lambda=0.5$ in the practical rule \eqref{eq:alpha-trace}, we do not observe any violation of the admissibility condition \eqref{equ:Q_kinlemma} during the iteration. This supports the use of the practical rule and suggests that it is both reasonable and sufficiently conservative for the experiments considered here.

\paragraph{Coordinatewise alignment \eqref{assump:alignment}.}
We also monitor the coordinatewise alignment condition in
Assumption~\ref{assump:alignment}. 
For each iteration \(k\) and each independent run \(r=1,\ldots,N\), write
\[
P_k^{(r)}=\operatorname{diag}(p_{k,1}^{(r)},\ldots,p_{k,d}^{(r)}),
\qquad
s_{k,i}^{(r)}
:=
\mathbb E\!\left[
g_{k+1,i}^2\mid x_{k+1}^{(r)}
\right].
\]
Here \(r\) indexes the independent random runs used to estimate the expected
trajectory. 

In the additive-multiplicative noise model used in the convex experiments,
this conditional second moment is computable exactly:
\[
s_{k,i}^{(r)}
=
(1+\sigma_1^2)
\big(\partial_i f(x_{k+1}^{(r)})\big)^2
+
\frac{\sigma_0^2}{d}.
\]

For a fixed trajectory \(r\), a coordinate pair \((i,j)\) violates the ordering
condition at iteration \(k\) if
\[
\big(s_{k,i}^{(r)}-s_{k,j}^{(r)}\big)
\big(p_{k,i}^{(r)}-p_{k,j}^{(r)}\big)<0.
\]
We define the ordering violation rate by
\[
\rho_k
:=
\frac{1}{N\binom{d}{2}}
\sum_{r=1}^{N}
\#\left\{
1\le i<j\le d:
\big(s_{k,i}^{(r)}-s_{k,j}^{(r)}\big)
\big(p_{k,i}^{(r)}-p_{k,j}^{(r)}\big)<0
\right\}.
\]
Thus, \(\rho_k\) is the fraction of all coordinate pairs, averaged over
independent runs, that violate the ordering condition at iteration \(k\). In our convex
experiments \(d=16\), so each trajectory contributes
\(\binom{16}{2}=120\) coordinate pairs. For the stochastic runs we use
\(N=200\) trajectories, so \(\rho_k\) is computed from
\(200\times120=24{,}000\) pairwise comparisons at each iteration. In the
deterministic case \(\sigma_1=\sigma_0=0\), we use \(N=1\).

\begin{figure}[!htbp]
	\centering
	\begin{subfigure}{0.48\textwidth}
		\centering
		\includegraphics[width=\linewidth]{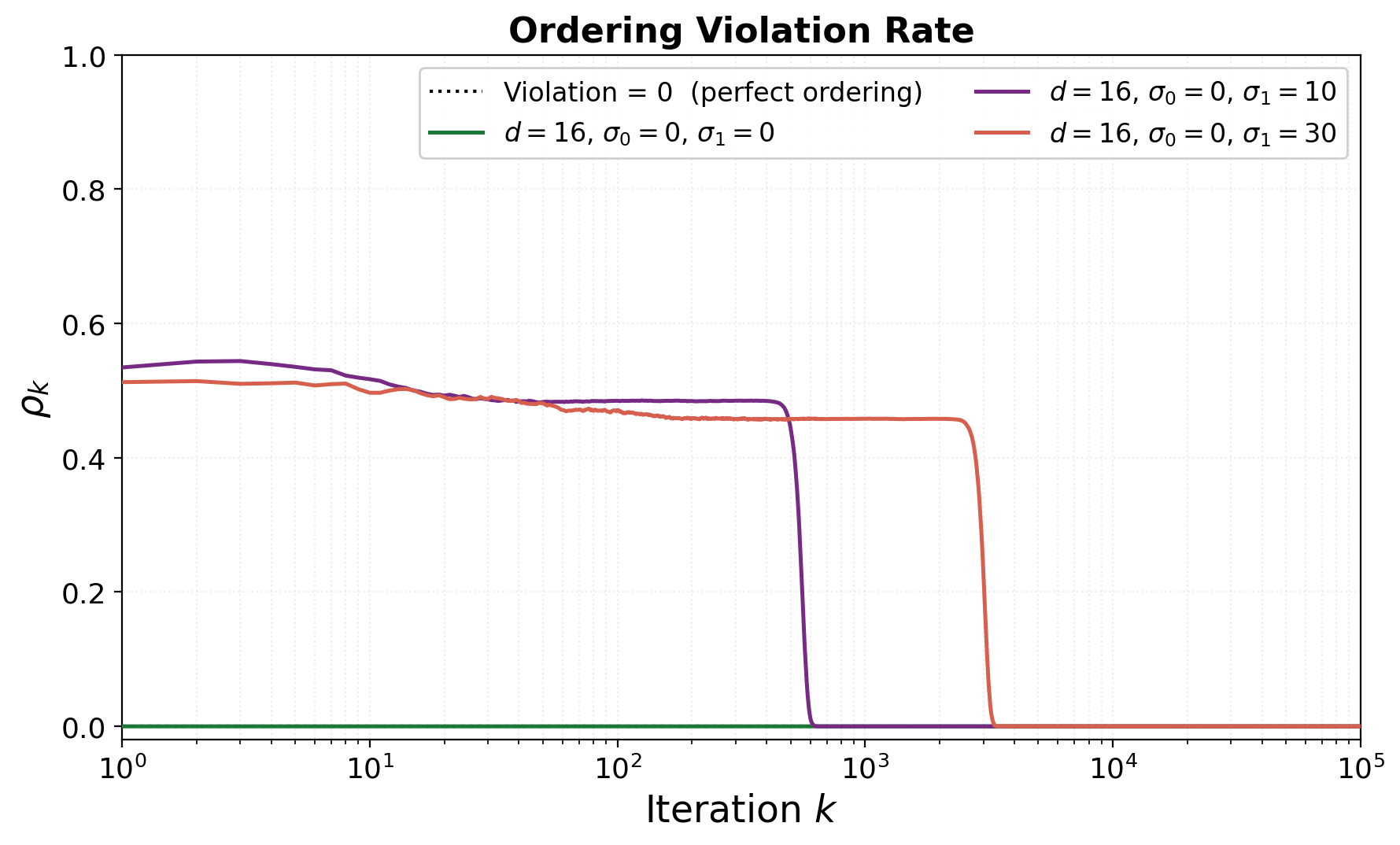}
		\caption{Pure multiplicative noise.}
	\end{subfigure}\hfill
	\begin{subfigure}{0.48\textwidth}
		\centering
		\includegraphics[width=\linewidth]{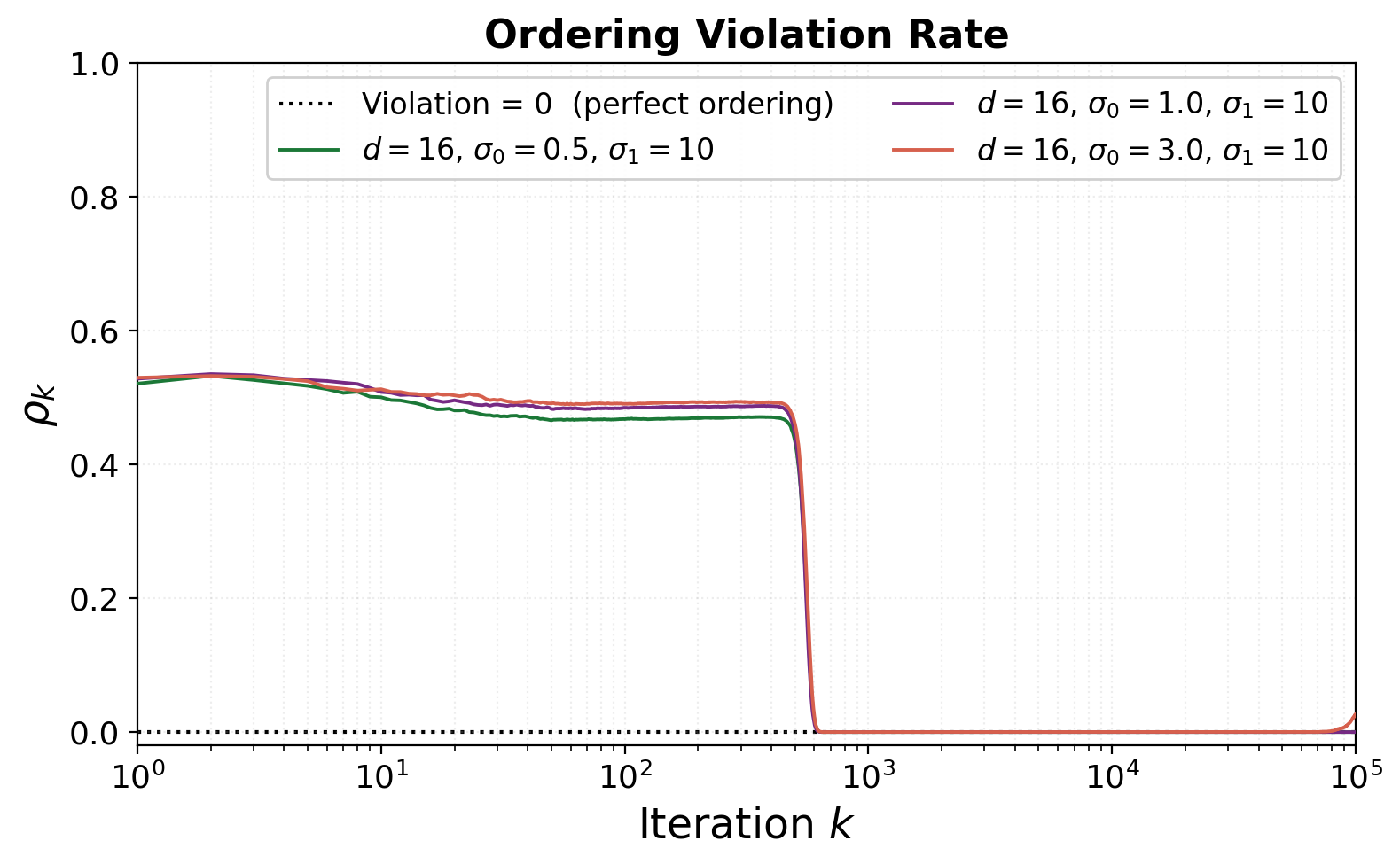}
		\caption{Additive-multiplicative noise.}
	\end{subfigure}
	\caption{Ordering violation rate for Assumption~\ref{assump:alignment}.
	A value of \(0\) corresponds to perfect coordinatewise alignment.}
	\label{fig:ordering-violation}
\end{figure}

Figure~\ref{fig:ordering-violation} reports this diagnostic. The ordering violation rate remains nonzero, and can be close to the random-ordering level, during a substantial pre-asymptotic phase. 
This reflects the fact that \(P_k\) is built from past squared gradients, whereas \(s_k\) represents the current conditional second-moment signal; before the metric and gradient scales stabilize, their coordinatewise orderings need not agree. 
After the dynamics enter a more stable regime, the violation rate drops to zero, indicating that the preconditioner ordering eventually aligns with the current second-moment ordering.

This diagnostic is supplementary. 
The convergence argument does not require exact pairwise alignment at every iteration; it requires the learning-rate admissibility condition in Lemma~\ref{lem:stoch-descent}. 
As shown in Figure~\ref{fig:consistency}, with the safety factor used in the practical rule, the monitored sufficient-decay ratio remains above one throughout the entire trajectory. 
Thus, even when the strict ordering condition is violated during the pre-asymptotic phase, the practical trace-ratio rule still satisfies the monitored admissibility condition in all tested settings.

\subsection{Supplement of Stress test}\label{app:reddi} 

This online convex optimization setting lies outside our theoretical
framework, so the parameter coupling derived in Section~\ref{sec:Adam-SHANG}
does not apply. We instead use the decoupled parametrization
$\alpha_k = \lambda \sqrt{\mathrm{Tr}(P_k^{-1})/\mathrm{Tr}(P_k^{-2})}$
with $\beta$ and $\gamma$ tuned independently via grid search on the
deterministic sequence; the same
$(\lambda,\beta,\gamma)=(0.001, 10^{-4}, 0.05)$ are shared by both
\textit{Adam-SHANG} and \textit{Adam-SHANG-s}. Adam and AMSGrad
use $(\eta,\beta_1,\beta_2)=(0.01,0.9,0.99)$
following \citet{reddi2019convergenceadam}.

A notable observation from Figure~\ref{fig:redditest_deterministic} is
that the average regret $R_t/t$ converges to zero for all methods in
the deterministic setting, yet the iterate trajectories tell a starkly
different story: Adam drifts to $x=+1$ while \textit{Adam-SHANG}
converges to $x^\star=-1$. This separation between regret and iterate
behavior is precisely the phenomenon highlighted by
\citet{reddi2019convergenceadam}: vanishing average regret does not
imply convergence to the optimal solution.

For the stochastic setting, we report both the mean
(Figure~\ref{fig:redditest_mean}) and median
(Figure~\ref{fig:redditest_median}) over 30 independent runs.
The gap between mean and median for AMSGrad reveals a bimodal behavior: a fraction of runs converge to $x^\star=-1$ while others remain near $x=+1$, pulling the mean away from the optimum. \textit{Adam-SHANG} does not exhibit this sensitivity; both its mean
and median trajectories converge consistently to $x^\star=-1$, outperforming even AMSGrad, which was specifically designed to remedy Adam's non-convergence.

We do not offer a full theoretical explanation for this robustness, but
offer three observations that may guide future work. First, the
lagged-update structure of \textit{Adam-SHANG} decouples the metric
update from the parameter step, which may prevent the self-reinforcing
drift that traps Adam-type methods near $x=+1$. Second, the nonlinear
feedback in the preconditioner update, the injection term $P_k^{-1}G^2$
rather than $G^2$, attenuates the influence of large past gradients on
$P_k$, naturally suppressing the accumulation of dominant historical
curvature estimates that drives Adam's divergence in this setting; this
self-correcting property is structurally absent in both Adam and
AMSGrad. Third, and perhaps more fundamentally, the lagged-update
discretization has roots in numerical analysis: semi-implicit schemes
are well known to provide superior stability over fully explicit
counterparts~\citep{ascher1995implicit}, and this property may carry
over to the optimization setting. Together, these observations suggest
that the lagged-update discretization offers a structurally different
stabilization pathway from the long-term memory approach of
\citet{reddi2019convergenceadam}, a direction that warrants further
theoretical investigation.

\begin{figure}[!htbp]
	\centering
	\begin{subfigure}{0.48\textwidth}
		\centering
		\includegraphics[width=\linewidth]{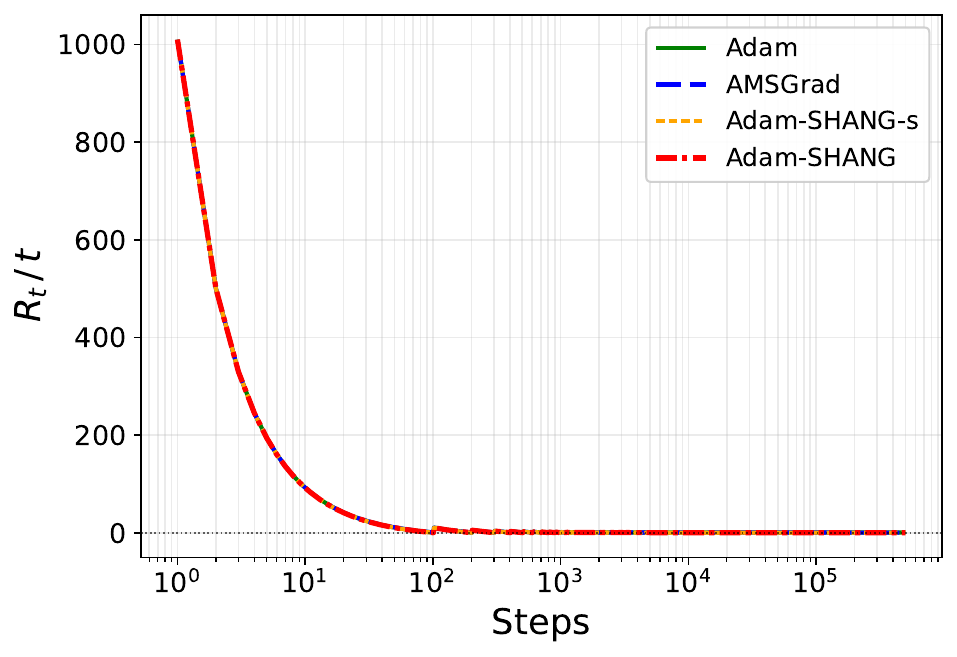}
	\end{subfigure}\hfill
	\begin{subfigure}{0.48\textwidth}
		\centering
		\includegraphics[width=\linewidth]{reddi_det_iterates}
	\end{subfigure}
	\caption{\small The deterministic version of the counterexample
		from \citet{reddi2019convergenceadam}. Each panel shows average
		regret $R_t/t$ (left) and iterate trajectory $x_t$ (right).}
	\label{fig:redditest_deterministic}
\end{figure}

\begin{figure}[!htbp]
	\centering
	\begin{subfigure}{0.48\textwidth}
		\centering
		\includegraphics[width=\linewidth]{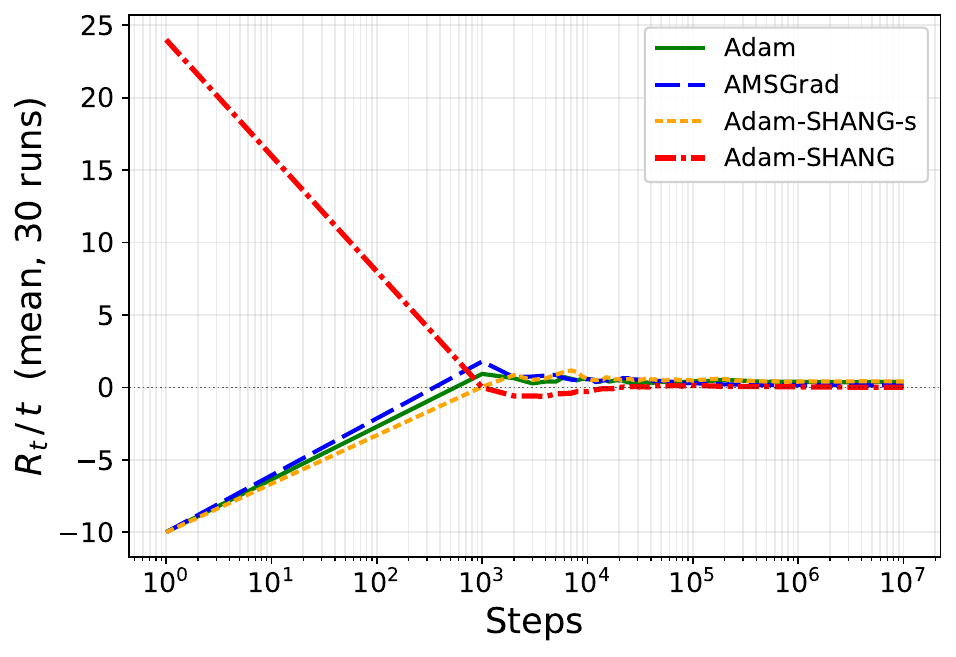}
	\end{subfigure}\hfill
	\begin{subfigure}{0.48\textwidth}
		\centering
		\includegraphics[width=\linewidth]{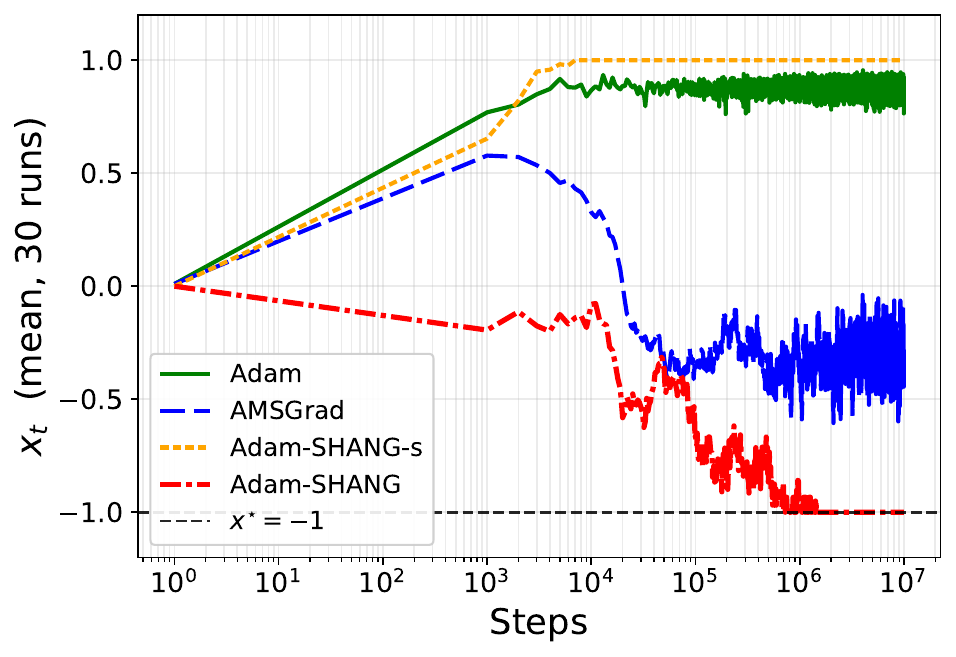}
	\end{subfigure}
	\caption{\small The stochastic version of the counterexample from
		\citet{reddi2019convergenceadam}. Each panel reports the \emph{mean}
		average regret $R_t/t$ (left) and the \emph{mean} iterate trajectory
		$x_t$ (right) over 30 independent runs. The large initial regret of
		\textit{Adam-SHANG} reflects a small fraction of runs with an
		anomalously large first step caused by $g_0 = 1010$ at initialization;
		this is a non-typical artifact, as confirmed by the median in
		Figure~\ref{fig:redditest_median}.}
		\label{fig:redditest_mean}
\end{figure}

\begin{figure}[!htbp]
	\centering
	\begin{subfigure}{0.48\textwidth}
		\centering
		\includegraphics[width=\linewidth]{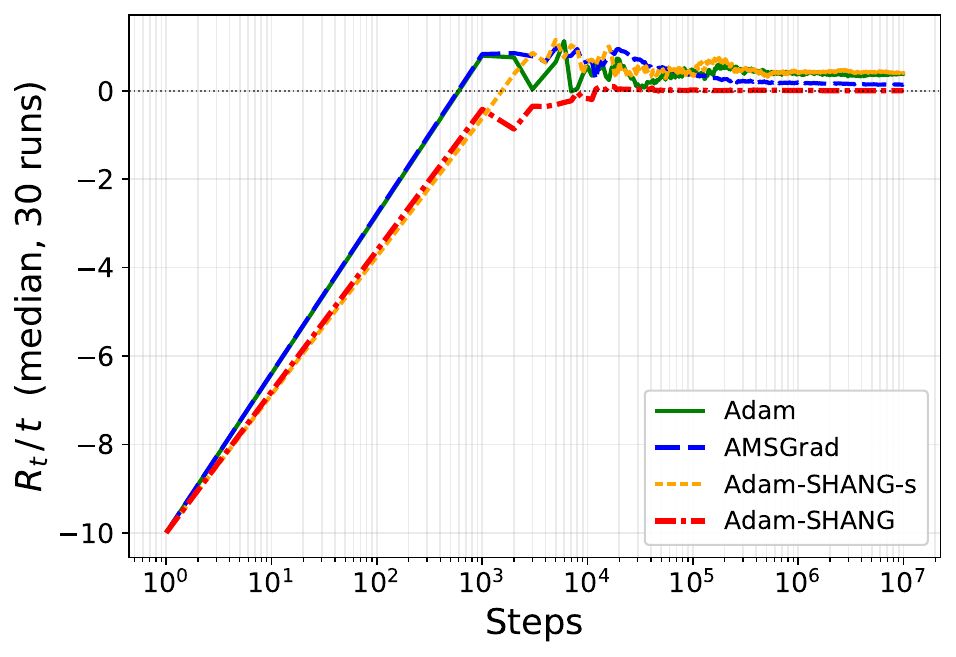}
	\end{subfigure}\hfill
	\begin{subfigure}{0.48\textwidth}
		\centering
		\includegraphics[width=\linewidth]{reddi_stoch_median_iterates}
	\end{subfigure}
	\caption{\small The stochastic version of the counterexample from
		\citet{reddi2019convergenceadam}. Each panel reports the \emph{median}
		average regret $R_t/t$ (left) and the \emph{median} iterate trajectory
		$x_t$ (right) over 30 independent runs.}
	\label{fig:redditest_median}
\end{figure}

\subsection{Experimental and hyperparameter settings for \texttt{text8}}
\label{app:text8}

We evaluate character-level language modeling on \texttt{text8}~\cite{Mahoney2009} using a $4$-layer Transformer encoder~\cite{vaswani2023attentionneed} with Pre-LayerNorm, $8$ attention heads, hidden dimension $256$, feedforward dimension $1024$, and about $3.17$M parameters. Training is autoregressive with causal masking and sequence length $T=256$. We use the full \texttt{text8} dataset ($\approx 10^8$ characters), split chronologically into $90\%$ training and $10\%$ validation sets. All methods are trained for $30{,}000$ steps with gradient clipping set to $1.0$.

All hyperparameters are selected independently for each method by grid search. For \textit{Adam-SHANG}, we use $(\lambda,\beta,\gamma)=(0.5,0.05,0.001)$; for \textit{Adam-SHANG-s}, we use $(\lambda,\beta,\gamma)=(0.5,0.05,0.01)$. Both methods use weight decay $10^{-2}$.

For the baselines, AdamW uses $\eta=0.005$, $(\beta_1,\beta_2)=(0.9,0.999)$, and weight decay $10^{-2}$; Adam uses $\eta=0.005$, $(\beta_1,\beta_2)=(0.9,0.999)$, and weight decay $10^{-5}$; 





\subsection{Hyperparameter settings for CIFAR-100}
\label{app:cifar100}

We summarize here the hyperparameter choices used in the CIFAR-100 experiments. For the Adam-SHANG variants, \textit{Adam-SHANG} and \textit{Adam-SHANG-s} use $(\lambda,\beta,\gamma)=(0.1,0.1,0.005)$ and $(\lambda,\beta,\gamma)=(0.1,0.05,0.001)$, respectively, with both configurations selected by grid search. For the baseline optimizers, Adam and AdamW use $\eta=10^{-3}$ with $(\beta_1,\beta_2)=(0.9,0.999)$, while SHANG++ uses $(\alpha,\gamma,\rho)=(0.5,10,1.5)$. Weight decay is set to $10^{-2}$ for \textit{Adam-SHANG}, \textit{Adam-SHANG-s}, and AdamW, and to $10^{-5}$ for Adam and SHANG++, following standard practice for these methods. Each model is trained for $50$ epochs with three independent random seeds, and we report the mean and standard deviation over the runs.

Moreover, standard data augmentation is applied for CIFAR-100, including normalization, random cropping, and random horizontal flipping.

\subsection{Image classification on CIFAR-10}\label{test:cifar10}
As a complementary evaluation, we further assess the optimization performance on standard image classification tasks using ResNet-34~\cite{HeZhangRenSun2016} on CIFAR-10~\cite{Krizhevsky2009}. Standard data augmentation is applied, including normalization, random cropping, and random horizontal flipping. Experiments are conducted with batch sizes $32$ and $128$ to examine the effect of stochastic gradient noise.

\begin{figure}[!htbp]
	\centering
	\begin{subfigure}{0.48\textwidth}
		\centering
		\includegraphics[width=\linewidth]{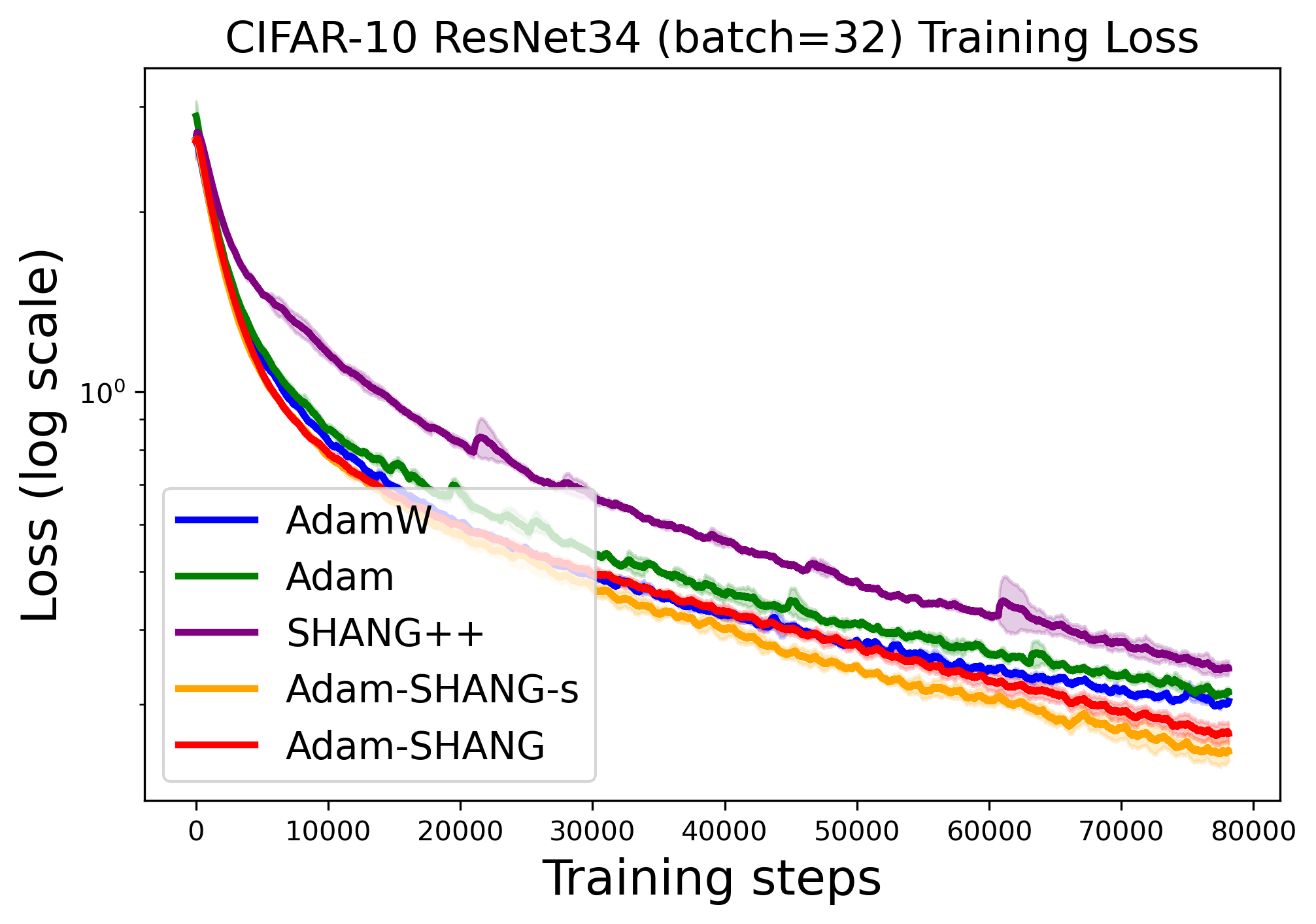}
	\end{subfigure}\hfill
	\begin{subfigure}{0.48\textwidth}
		\centering
		\includegraphics[width=\linewidth]{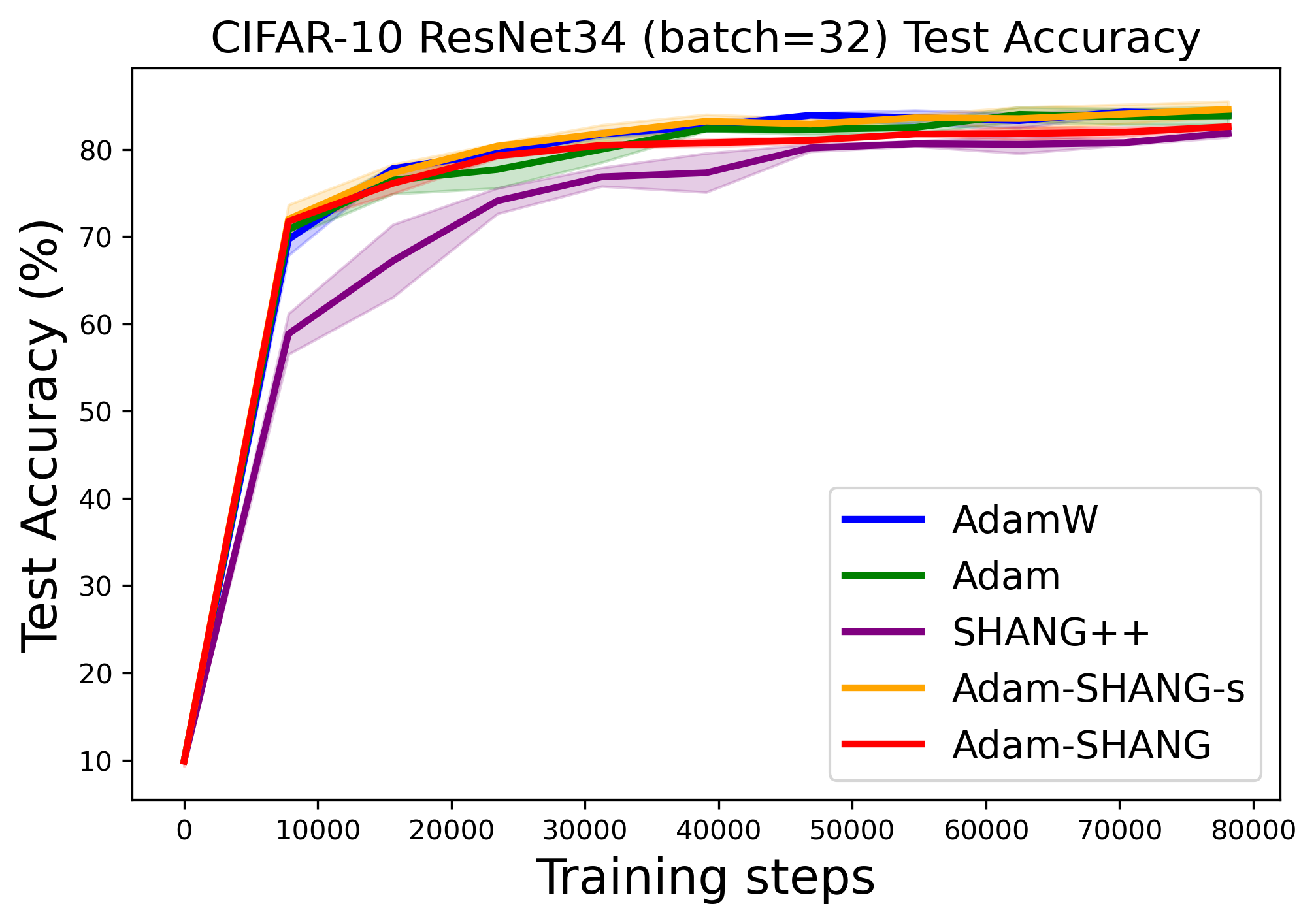}
	\end{subfigure}\hfill
	\begin{subfigure}{0.48\textwidth}
		\centering
		\includegraphics[width=\linewidth]{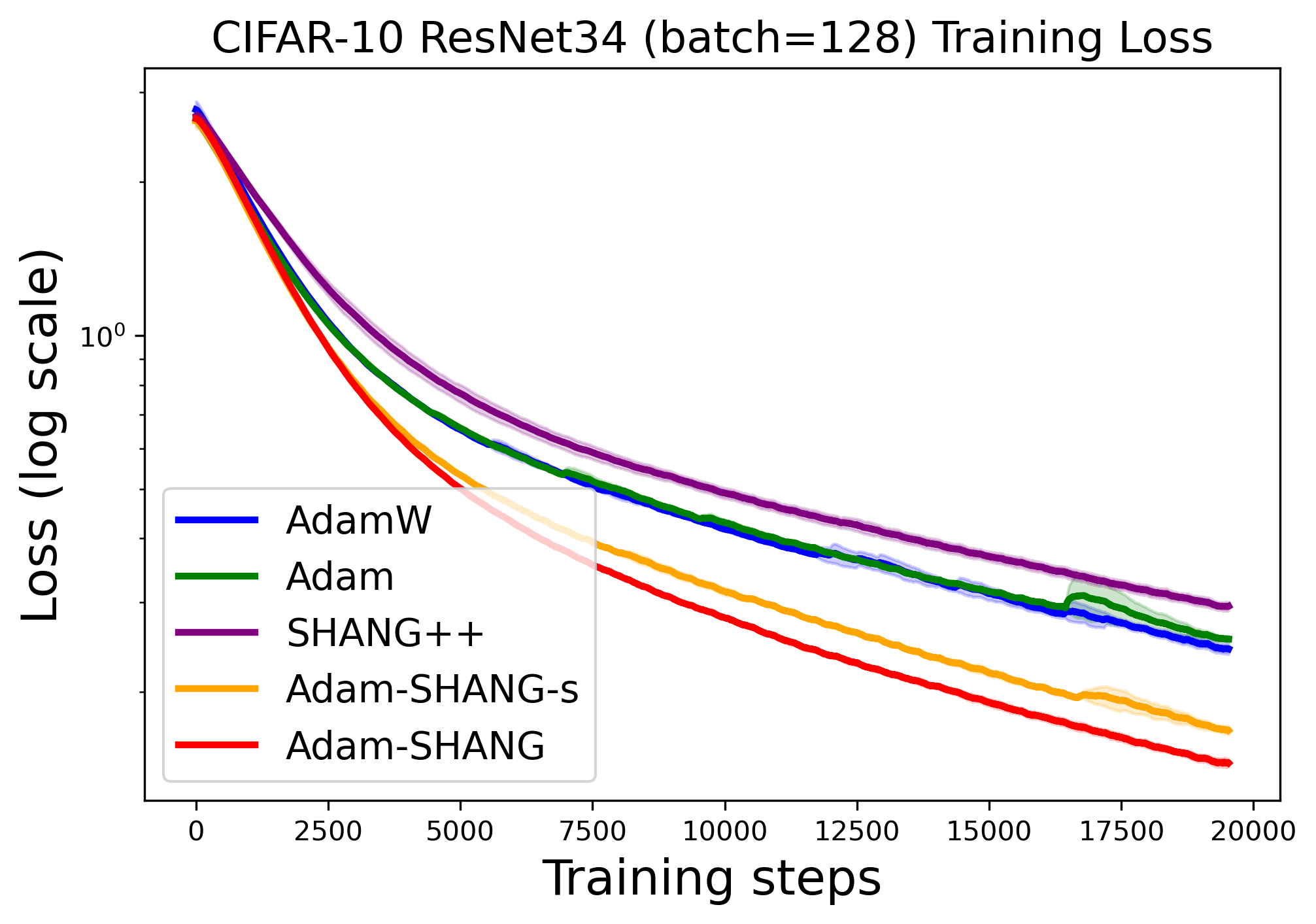}
	\end{subfigure}\hfill
	\begin{subfigure}{0.48\textwidth}
		\centering
		\includegraphics[width=\linewidth]{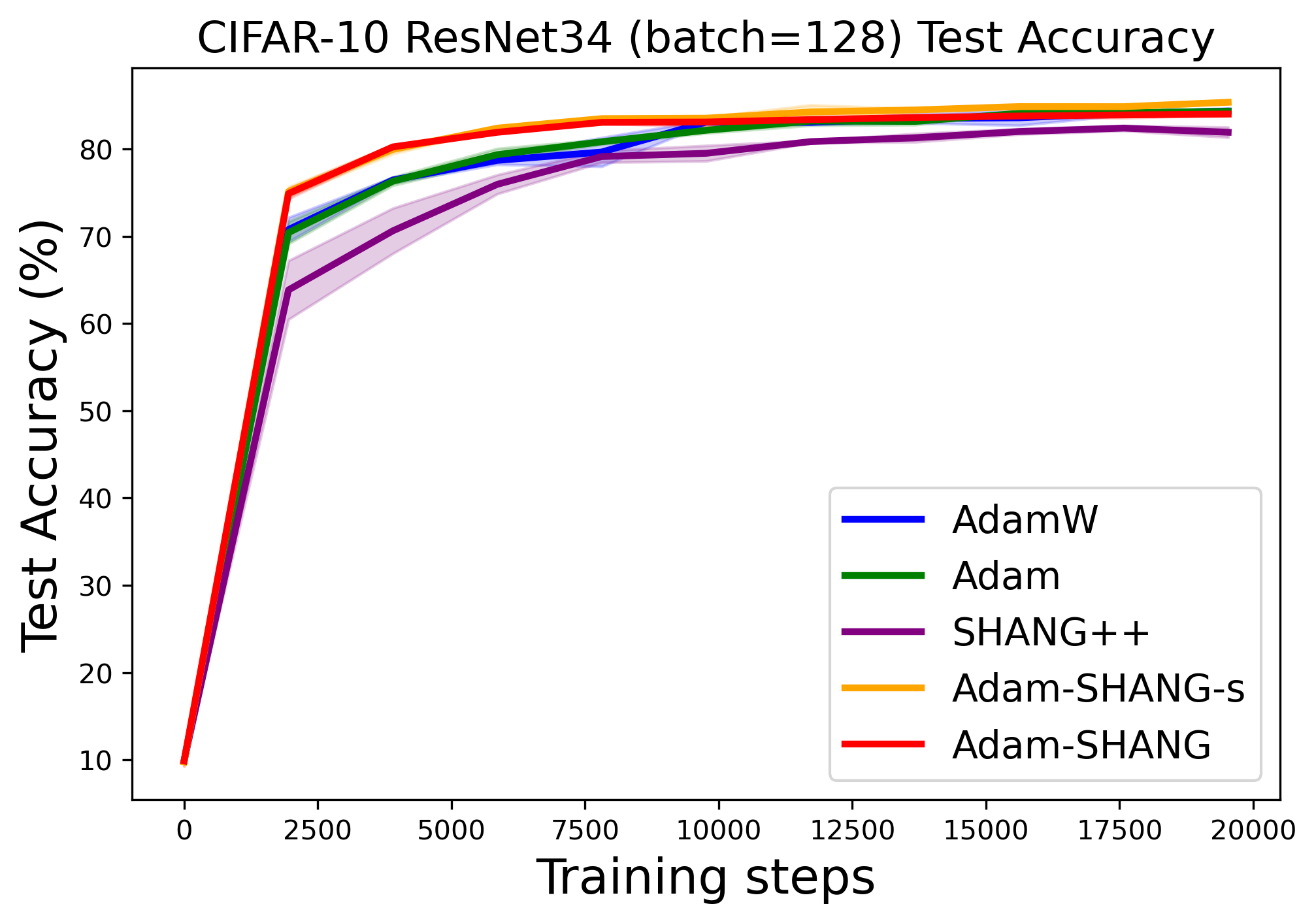}
	\end{subfigure}\hfill
	\caption{\small Training loss and test accuracy of different optimizers for training ResNet-34 on CIFAR-10. }
	\label{fig:cifar10_resnet}
\end{figure}

For the Adam-SHANG variants, \textit{Adam-SHANG} and \textit{Adam-SHANG-s} use $(\lambda,\beta,\gamma)=(0.1,0.05,0.05)$ and $(\lambda,\beta,\gamma)=(0.1,0.1,1)$, respectively, selected via grid search. For the baseline optimizers, Adam and AdamW use $\eta=10^{-3}$ with $(\beta_1,\beta_2)=(0.9,0.999)$, while SHANG++ uses $(\alpha, \gamma, \rho)=(0.5, 5, 1.5)$. We set weight decay to $10^{-2}$ for \textit{Adam-SHANG}, \textit{Adam-SHANG-s}, and AdamW, and to $10^{-5}$ for Adam and SHANG++, following common practice. Each model is trained for $50$ epochs with three independent random seeds, and we report the mean and standard deviation over the runs.

Figure~\ref{fig:cifar10_resnet} shows that both \textit{Adam-SHANG} and \textit{Adam-SHANG-s} achieve optimization behavior comparable to standard adaptive optimizers. In particular, \textit{Adam-SHANG-s} yields slightly better test accuracy in this setting, while \textit{Adam-SHANG} remains competitive in training loss reduction. These observations indicate that the proposed reformulations, although primarily motivated by theory, remain practically meaningful in stochastic deep learning tasks.


\begin{figure}[!htbp]
	\centering
	\begin{subfigure}{0.48\textwidth}
		\centering
		\includegraphics[width=\linewidth]{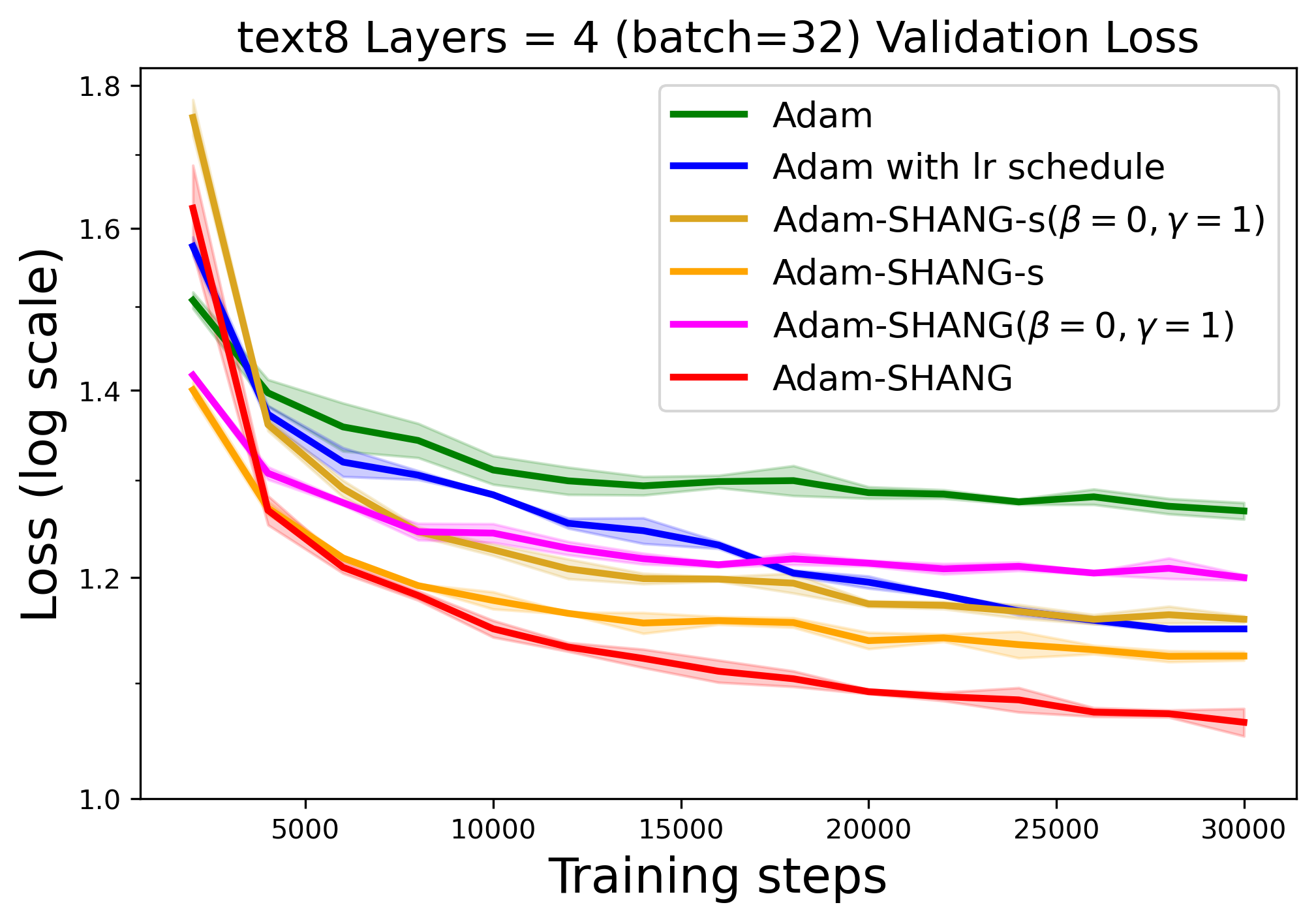}
	\end{subfigure}\hfill
	\begin{subfigure}{0.48\textwidth}
		\centering
		\includegraphics[width=\linewidth]{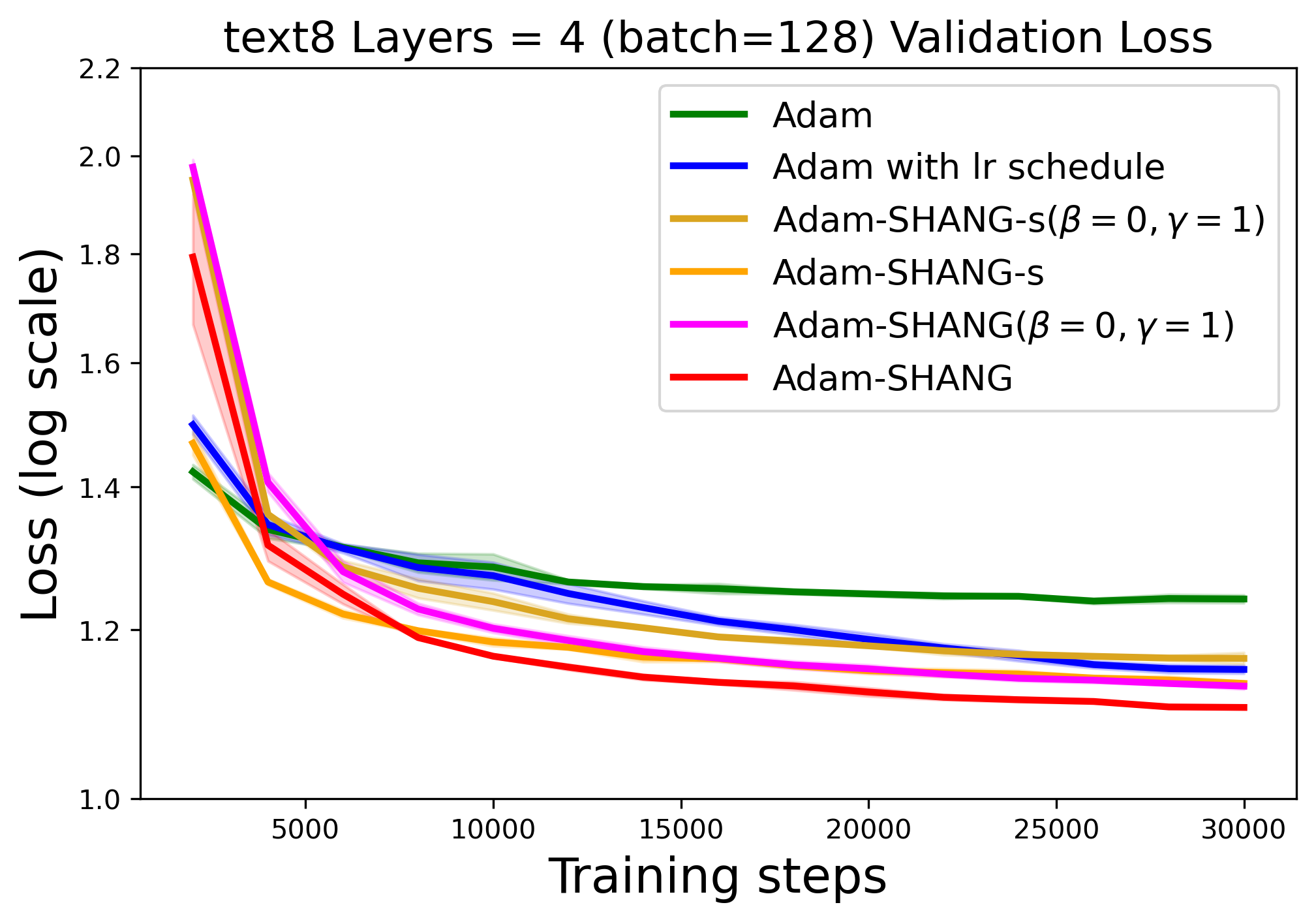}
	\end{subfigure}\hfill
	\caption{\small Ablation on the curvature-aware correction term for training the 4-layer Transformer on \texttt{text8}}
\label{fig:test beta_text8}
\end{figure}

\begin{figure}[!htbp]
	\centering
	\begin{subfigure}{0.48\textwidth}
	\centering
	\includegraphics[width=\linewidth]{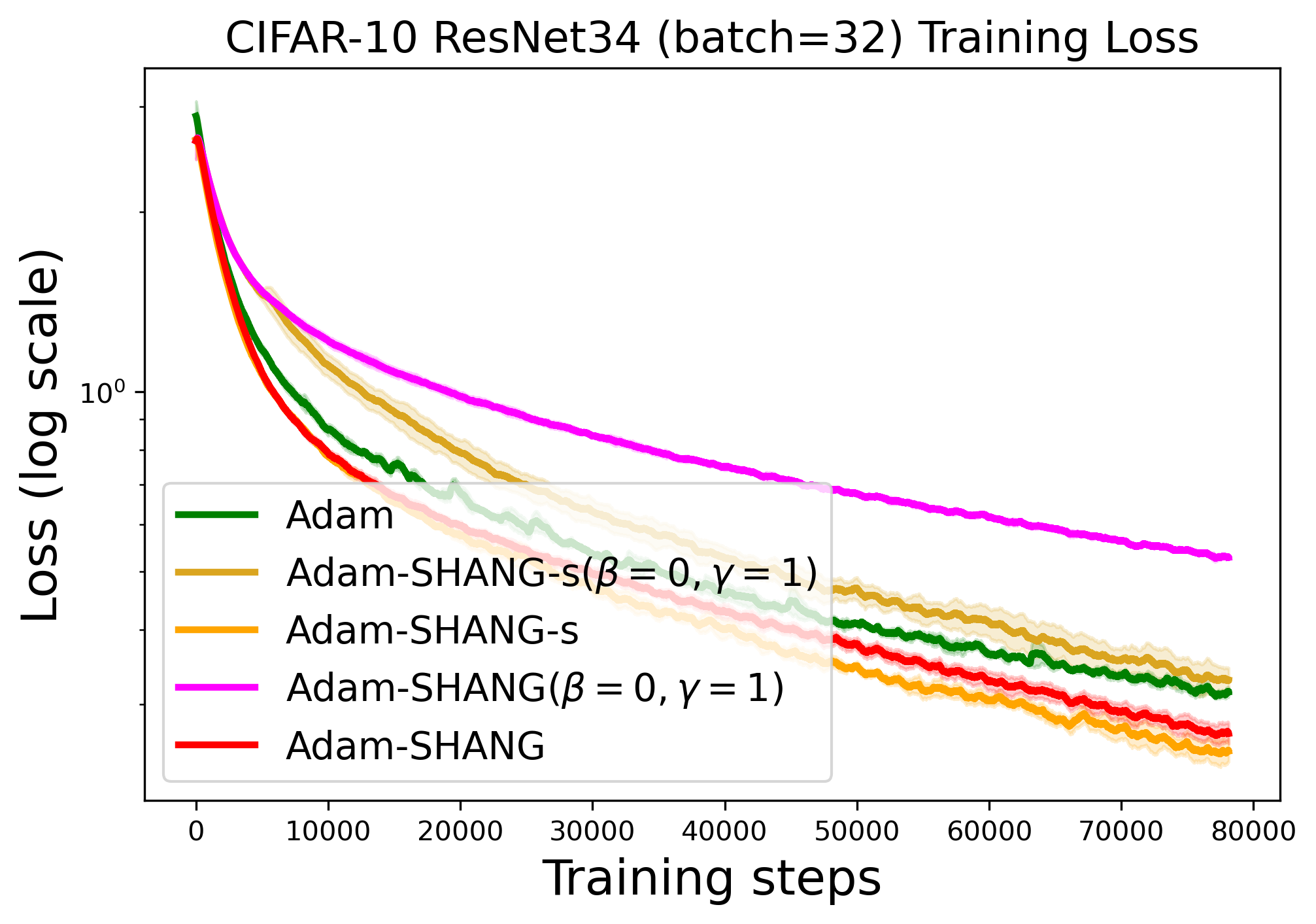}
\end{subfigure}\hfill
\begin{subfigure}{0.48\textwidth}
	\centering
	\includegraphics[width=\linewidth]{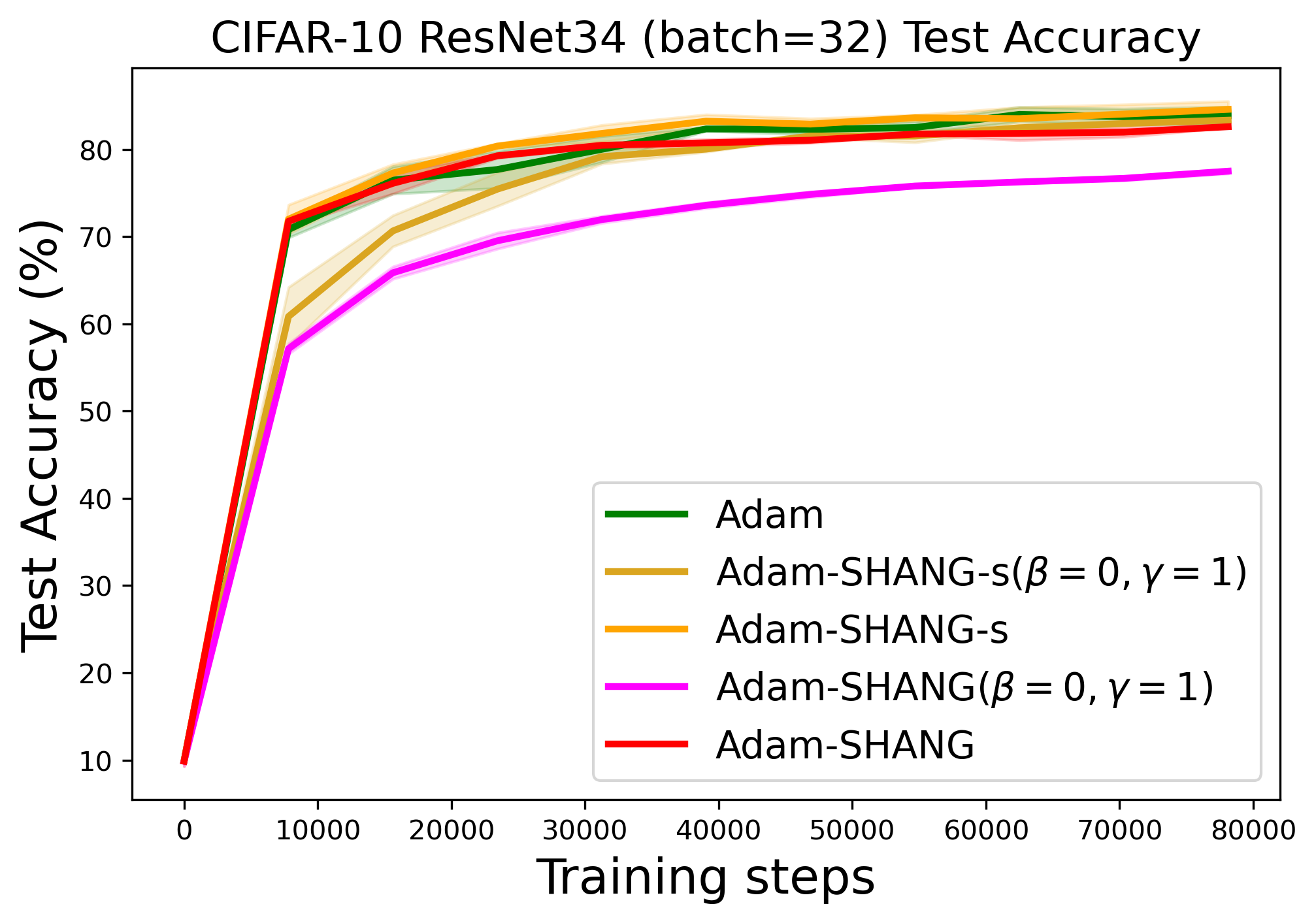}
\end{subfigure}\hfill
\begin{subfigure}{0.48\textwidth}
	\centering
	\includegraphics[width=\linewidth]{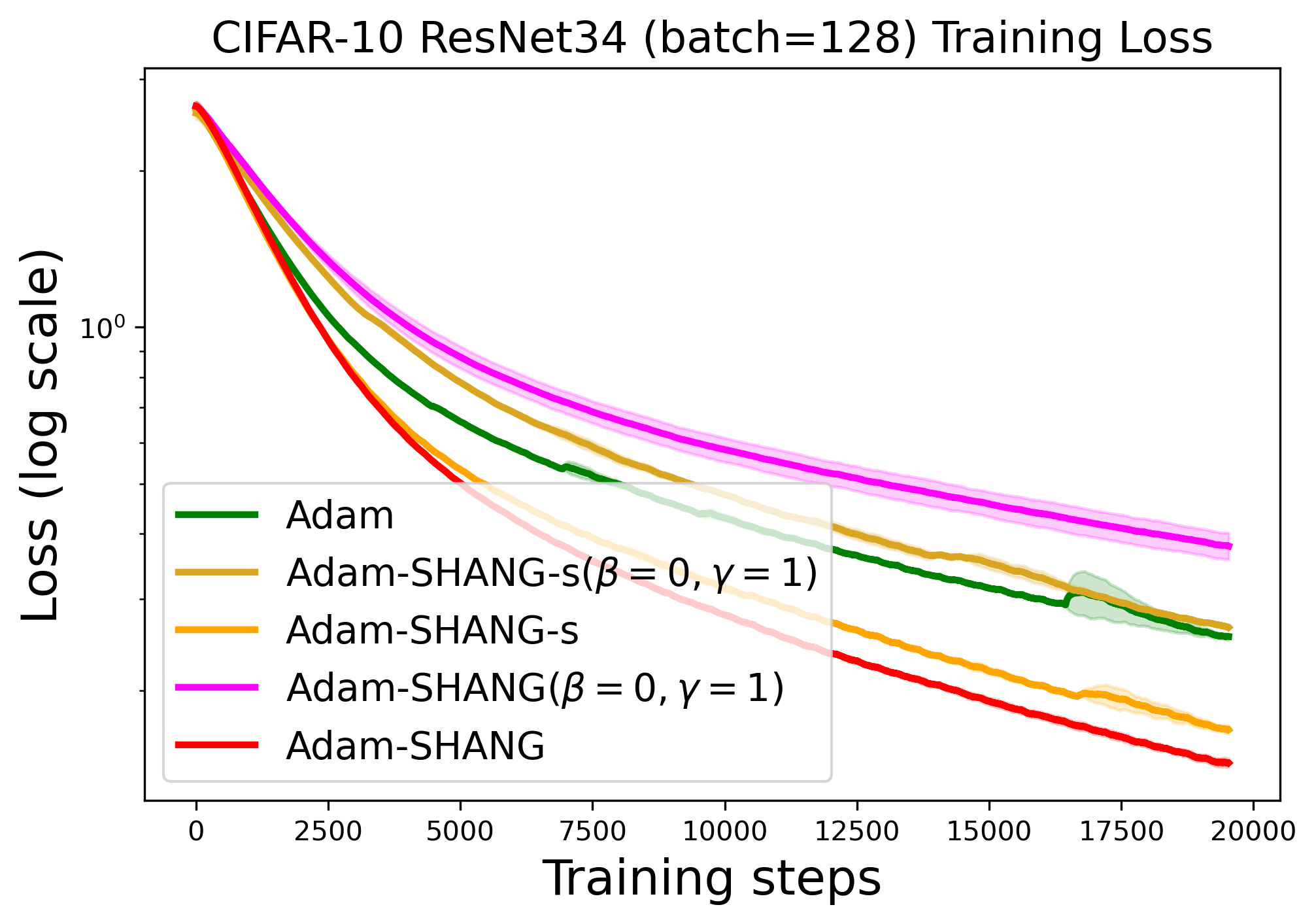}
\end{subfigure}\hfill
\begin{subfigure}{0.48\textwidth}
	\centering
	\includegraphics[width=\linewidth]{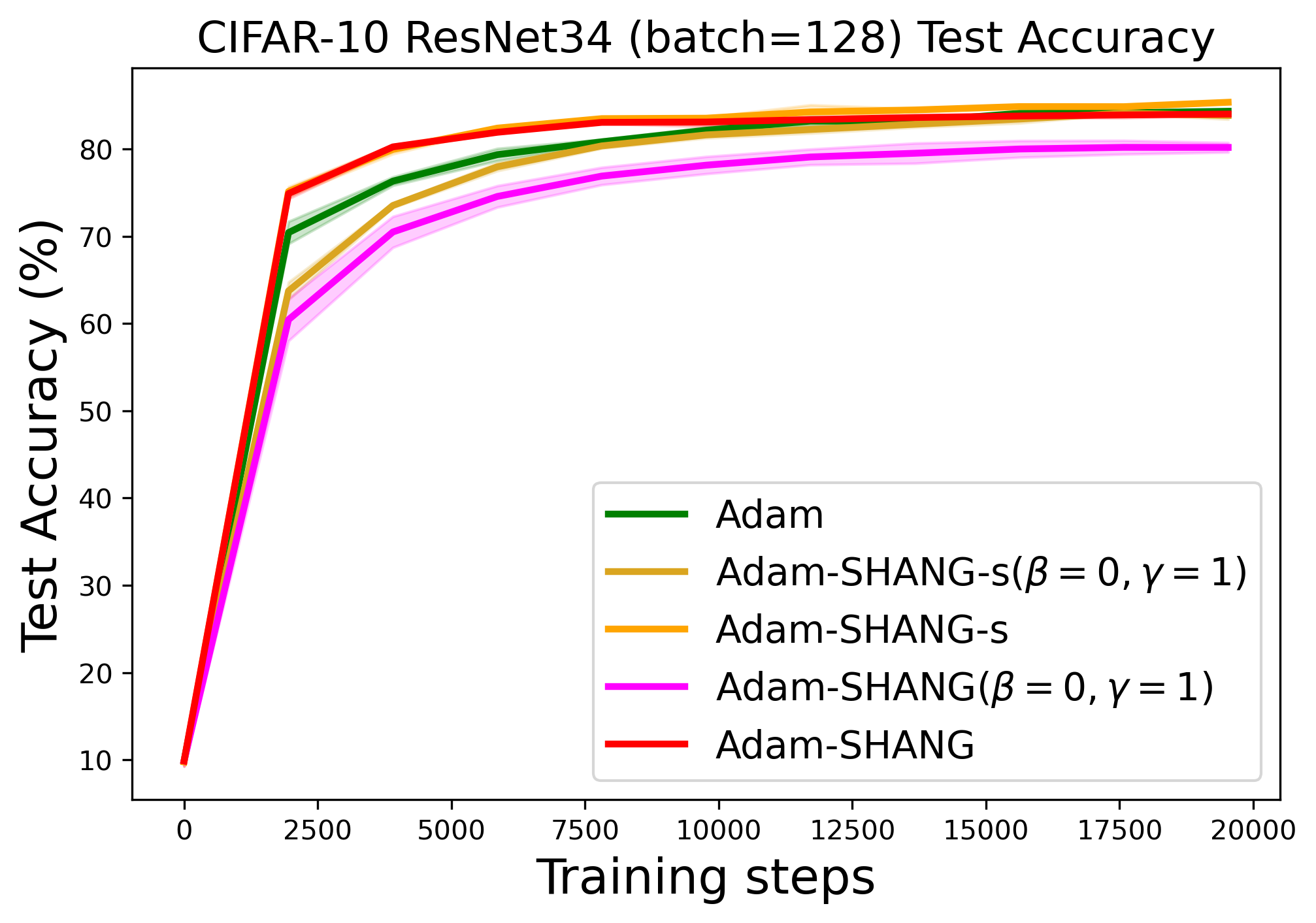}
\end{subfigure}\hfill
	\caption{\small Ablation on the curvature-aware correction term for training ResNet-34 on CIFAR-10.} 
\label{fig:test beta_cifar}
\end{figure}

\begin{figure}[!htbp]
	\centering
	\begin{subfigure}{0.48\textwidth}
		\centering
		\includegraphics[width=\linewidth]{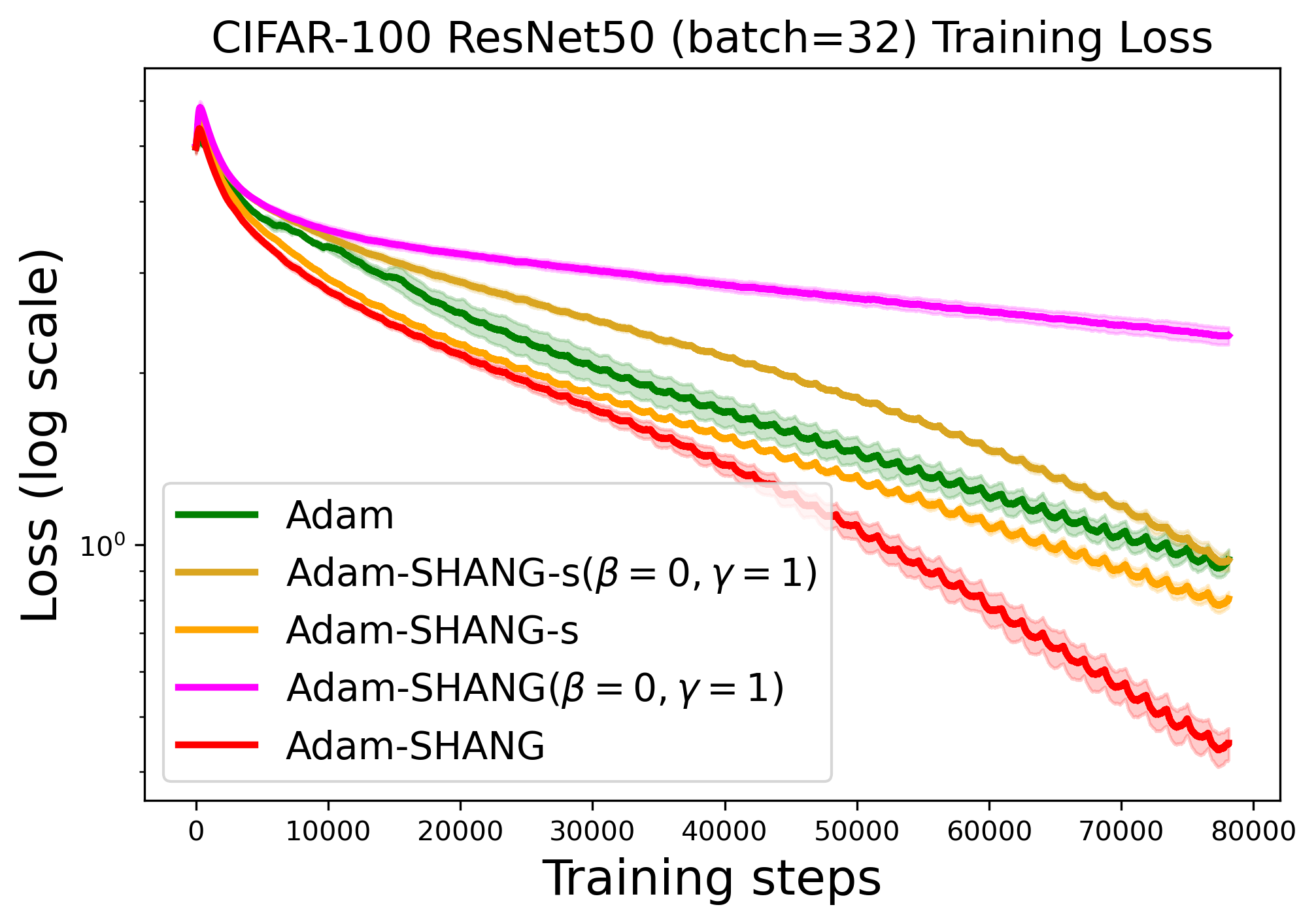}
	\end{subfigure}\hfill
	\begin{subfigure}{0.48\textwidth}
		\centering
		\includegraphics[width=\linewidth]{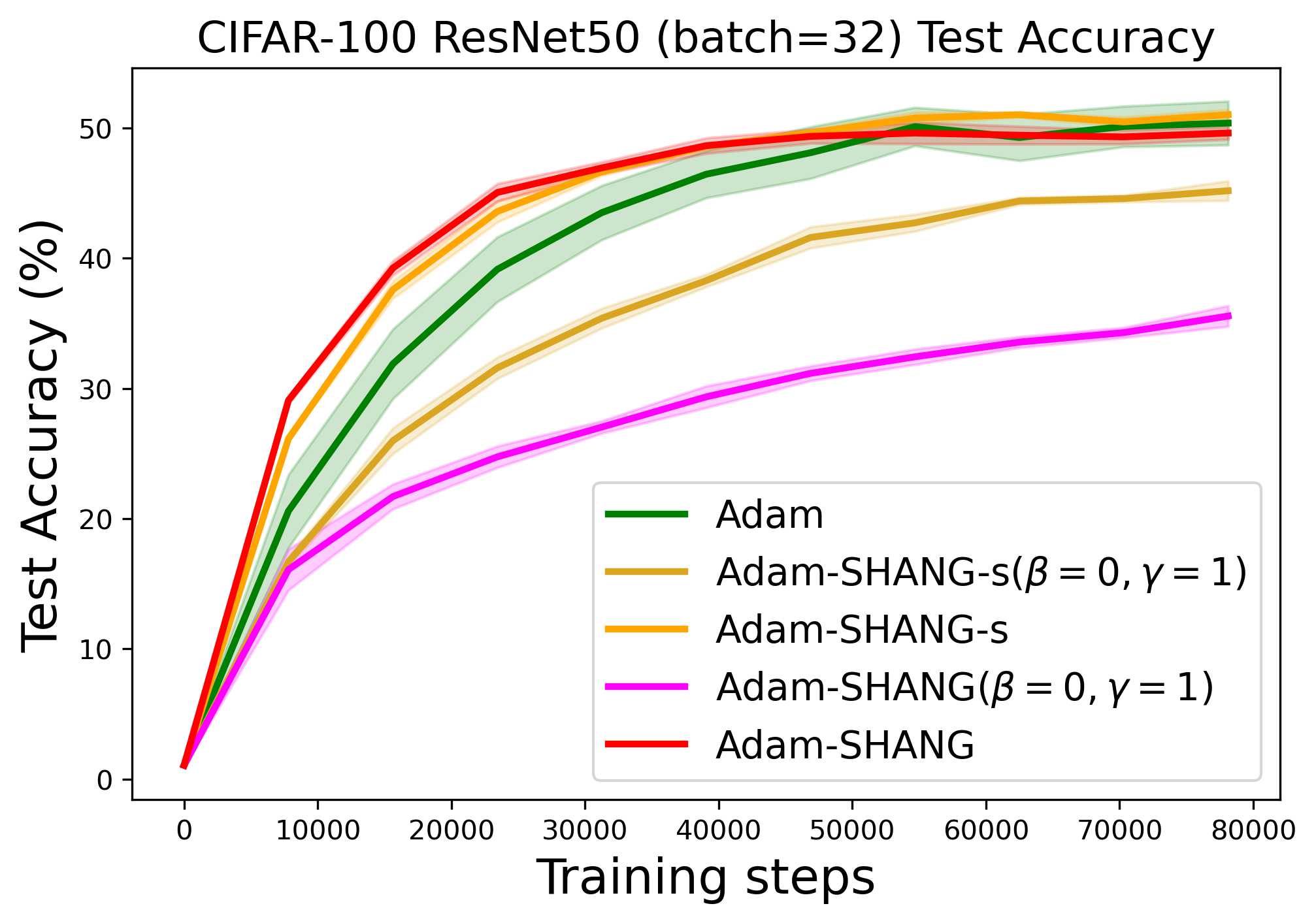}
	\end{subfigure}\hfill
	\begin{subfigure}{0.48\textwidth}
		\centering
		\includegraphics[width=\linewidth]{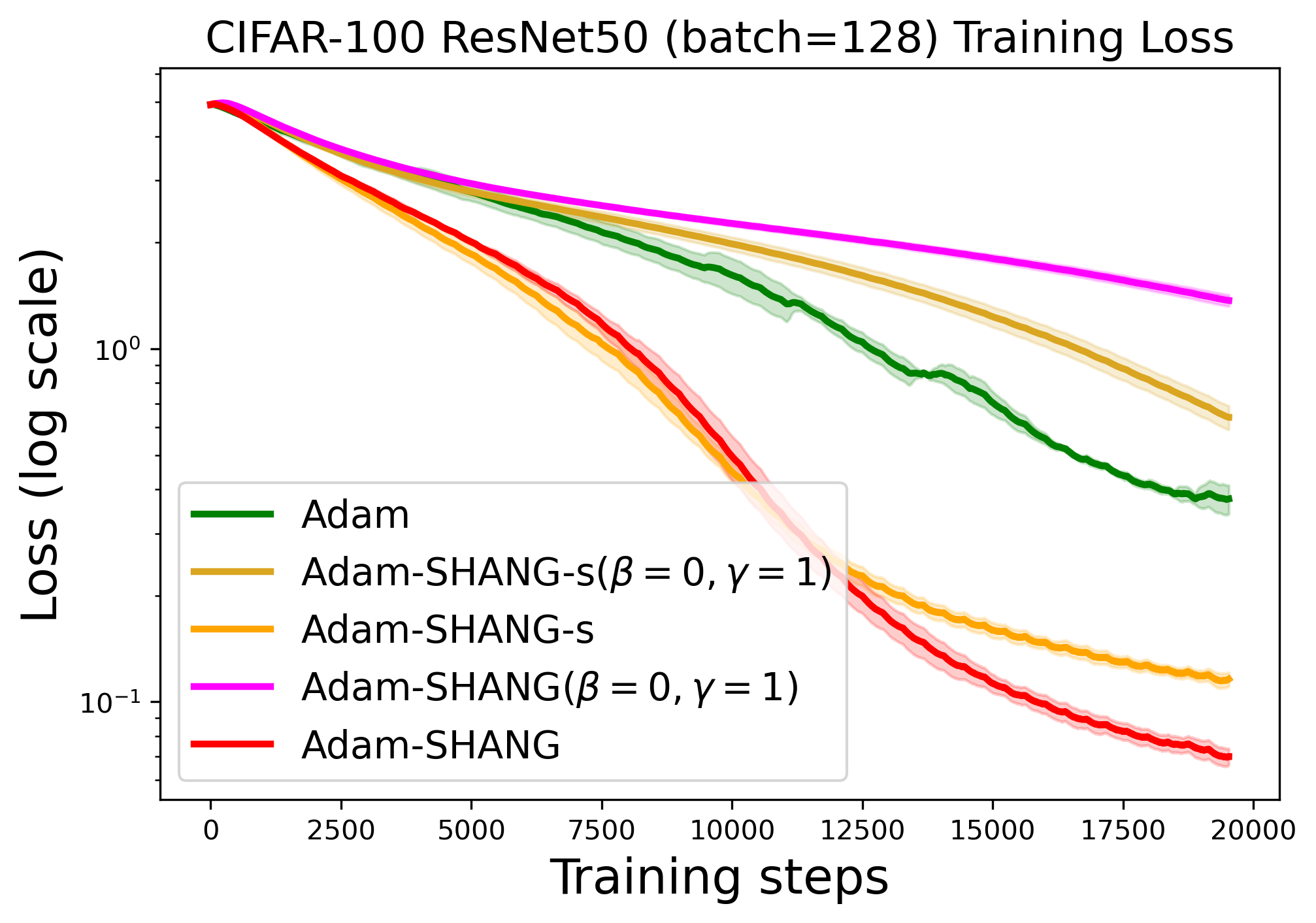}
	\end{subfigure}\hfill
	\begin{subfigure}{0.48\textwidth}
		\centering
		\includegraphics[width=\linewidth]{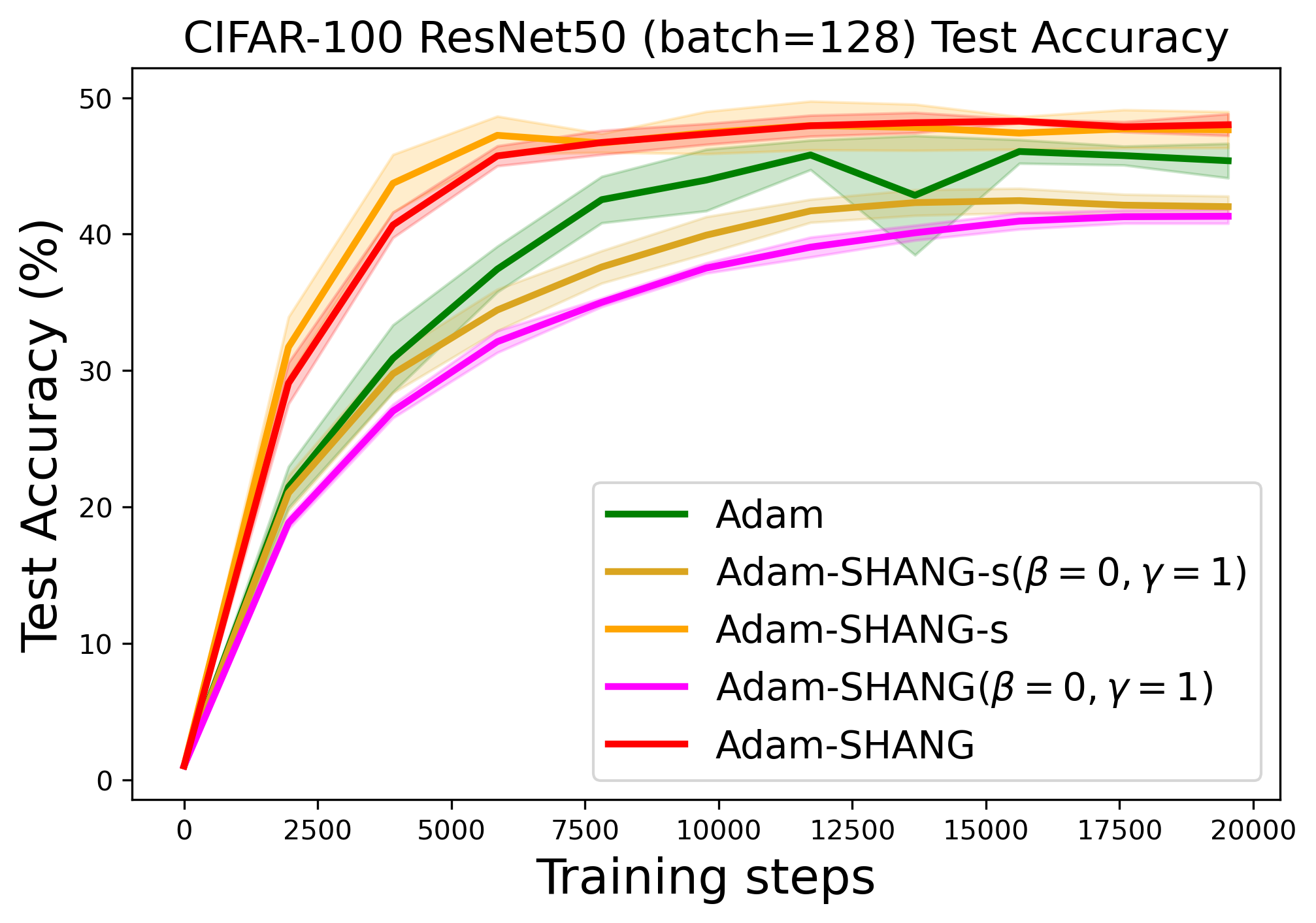}
	\end{subfigure}\hfill
	\caption{\small Ablation on the curvature-aware correction term for training ResNet-50 on CIFAR-100.} 
	\label{fig:test beta_cifar100}
\end{figure}

\subsection{Ablation on the curvature correction}\label{test:ablation of correction}

To examine the role of the curvature-aware correction term, we include two additional ablation variants in the deep learning experiments: {\it Adam-SHANG}$(\beta=0,\gamma=1)$ and {\it Adam-SHANG-s}$(\beta=0,\gamma=1)$. 
In these ablations, the base scale $\lambda$ is kept the same as in the corresponding main experiments. 
Setting $\beta=0$ removes the extra gradient correction term, while setting $\gamma=1$ keeps the preconditioner update in a normalized Adam-style scale. 
In particular, {\it Adam-SHANG-s}$(\beta=0,\gamma=1)$ provides the closest Adam-style counterpart within our framework, whereas {\it Adam-SHANG}$(\beta=0,\gamma=1)$ isolates the effect of removing the correction term in the semi-implicit discretization. 
This comparison is designed to assess whether the empirical advantages of the full methods are related to the curvature-aware correction mechanism highlighted by the theory.

The ablation results in Figures~\ref{fig:test beta_text8}, \ref{fig:test beta_cifar}, and \ref{fig:test beta_cifar100} further clarify the role of the curvature-aware correction term. On \texttt{text8}, all variants within our framework outperform Adam without an external learning-rate schedule, highlighting that such a schedule is crucial for Adam in Transformer training. Although {\it Adam-SHANG-s}$(\beta=0,\gamma=1)$ is the closest Adam-style counterpart in our framework, its built-in decay of $\alpha_k$ makes its behavior closer to Adam with external decay, while still yielding faster loss reduction in the early stage.

On the more standard image-classification tasks CIFAR-10 and CIFAR-100, the five methods behave more similarly overall. Even so, for both batch sizes, the full versions of {\it Adam-SHANG} and {\it Adam-SHANG-s} consistently achieve lower training loss than their corresponding $\beta=0$ variants. The effect is especially pronounced for {\it Adam-SHANG}, where removing the correction leads to markedly slower optimization and a clearly worse final loss. For {\it Adam-SHANG-s}, the degradation is milder but remains consistent across training, indicating that the synchronous Adam-style discretization is intrinsically more robust while still benefiting from the curvature-aware correction.

Beyond the overall advantage of the full methods, the ablations also reveal a distinction between optimization and generalization effects. Across all three tasks, the curvature-aware correction produces the most consistent gains in training or validation loss, whereas the improvement in test accuracy is more task-dependent. In particular, the effect is relatively modest on CIFAR-10, but becomes clearer on the more challenging \texttt{text8} and CIFAR-100 tasks. Overall, these results suggest that the correction term plays a stable role in improving optimization, while its contribution to generalization depends more strongly on the complexity and noise level of the task.

We also note that these ablation variants use the same learning rate as the corresponding main experiments and are not separately re-tuned by grid search. Therefore, the comparison should be interpreted as a controlled ablation designed to isolate the role of the curvature correction, rather than as the best achievable performance of the $\beta=0$ variants after task-specific hyperparameter tuning. This point is particularly relevant for {\it Adam-SHANG}$(\beta=0,\gamma=1)$, whose performance may partly reflect the fact that the inherited learning rate is no longer optimal once the correction term is removed.

\end{document}